\documentclass[sglanonrev]{informs4}
\RequirePackage{tgtermes}
\RequirePackage{newtxtext}
\RequirePackage{newtxmath}
\RequirePackage{endnotes}

\OneAndAHalfSpacedXII 

\usepackage{algorithm}
\usepackage{algpseudocode}
\usepackage{tikz}

\usepackage{natbib}
 \bibpunct[, ]{(}{)}{,}{a}{}{,}%
 \def\bibfont{\small}%
\EquationsNumberedThrough    

\TheoremsNumberedThrough     
\ECRepeatTheorems  %

\MANUSCRIPTNO{MOOR-0001-2024.00}

\newtheorem{Defi}{Definition}    
\usepackage{subfigure}  
\usepackage[linesnumbered, ruled, vlined, algo2e]{algorithm2e}  
\OneAndAHalfSpacedXI  
\usepackage{booktabs}  
\usepackage{multirow}  
\usepackage{anyfontsize}  
\usepackage{bbm}  
\usepackage{amssymb}  

\begin{document}


\RUNAUTHOR{Jia et~al.}

\RUNTITLE{The Distributionally Robust CIRP}

\TITLE{The Distributionally Robust Cyclic Inventory Routing Problem}

\ARTICLEAUTHORS{%
\AUTHOR{Menglei Jia}
\AFF{School of Maritime Economics and Management, Dalian Maritime University, Dalian 116026, China, \EMAIL{jia\_menglei@163.com}}
\AFF{Department of Management Science, Antai College of Economics and Management, Shanghai Jiao Tong University, Shanghai 200030, China}

\AUTHOR{Albert H. Schrotenboer}
\AFF{Operations, Planning, Accounting and Control Group, Department of Industrial Engineering \& Innovation Sciences, Eindhoven University of Technology, 5612 AZ Eindhoven, Netherlands, \EMAIL{a.h.schrotenboer@tue.nl}}

\AUTHOR{Ahmadreza Marandi}
\AFF{Operations, Planning, Accounting and Control Group, Department of Industrial Engineering \& Innovation Sciences, Eindhoven University of Technology, 5612 AZ Eindhoven, Netherlands, \EMAIL{a.marandi@tue.nl}}

\AUTHOR{Feng Chen$^*$}
\AFF{Sino-US Global Logistics Institute, Antai College of Economics and Management, Shanghai Jiao Tong University, Shanghai 200030, China, \EMAIL{fchen@sjtu.edu.cn}}
} 

\ABSTRACT{%
We study the cyclic inventory routing problem that involves joint decisions on vehicle routing and inventory replenishment on an infinite, cyclic horizon. It considers a single warehouse and a set of geographically dispersed retailers. We model retailer demand as random variables with uncertain distributions belonging to a moment-based ambiguity set. We develop a distributionally robust optimization formulation that minimizes the worst-case expected cost over the ambiguity set, while ensuring service reliability through a distributionally robust chance constraint. Our main results are that we prove that the worst-case expected inventory cost is attained under a multi-point distribution, which can be identified a posteriori via linear programming, and that the distributionally robust chance constraint can be reformulated into near-equivalent deterministic forms. This yields a deterministic reformulation of the original problem. To solve it, we design a nested branch-and-price framework, in which the first level partitions retailers into clusters, and the second level concerns routing and replenishment decisions within each cluster. Computational experiments on both synthetic instances and real-world data from SAIC Volkswagen Automobile Co., Ltd. demonstrate the effectiveness and efficiency of the proposed approach.
}%





\maketitle


\section{Introduction}\label{s1}
Stochastic inventory routing problems (SIRPs) are fundamental to a wide range of real-world logistics systems. It involves a one-to-many distribution network comprising a central warehouse, a geographically dispersed set of retailers, and a fleet of vehicles stationed at the warehouse to replenish retailer inventories, while the retailers face stochastic demand \citep[see, e.g.,][]{gaur2004periodic}. A central feature of SIRPs is the trade-off between the retailers’ inventory-related costs (including holding and backorder costs) and the transportation costs incurred during replenishment: frequent small-quantity deliveries lower inventory costs and improve customer service while less-frequent high-quantity deliveries lower transportation cost and improve vehicle utilization. The SIRP address this trade-off by jointly optimizing inventory decisions and vehicle routing, thereby providing an integrated framework for coordinating replenishment and distribution operations. As a result, they have found broad applicability in domains such as hydrogen distribution \citep{hasturk2024stochastic}, supermarket supply systems \citep{gaur2004periodic}, heating-oil delivery \citep{trudeau1992stochastic}, medical supply distribution \citep{rave2025cyclic}, the automobile aftermarket \citep{jia2024scenario}, and perishable-product logistics \citep{crama2018stochastic}, among others.

However, implementing joint inventory replenishment and transportation decisions over long operational periods poses challenges.  In practice, SIRPs are often solved using a rolling-horizon approach. At each decision epoch, a distribution plan for the next periods is optimized based on the currently observed inventory levels and some measure of future demand. Then, only the first few periods of the resulting plan are executed, after which the problem is re-optimized using updated inventories for the next epoch \citep[see, e.g.,][]{jaillet2002delivery, coelho2014heuristics, jia2024scenario}. This straightforward approach has two main drawbacks, however. First, it generates end-of-horizon  effects \citep{malicki2021cyclic, zojaji2022cyclic}: to reduce costs, replenishment at the end of each horizon is often delayed, producing unrealistically low terminal inventories and potential service disruptions. Second, independently optimizing each horizon reduces plan stability and repeatability, requiring warehouses and retailers to reconcile schedules at every iteration \citep{coelho2012consistency, malicki2021cyclic}. This is impossible in many practical contexts, as  up- and downstream planning processes in the supply chain are depending on it. 
In response, researchers and practitioners are focusing on cyclic schedules that overcome these drawbacks. That is, they design a repeatable distribution plan in which the end of each cycle aligns with the start of the next. This eliminates end-of-horizon effects, reduces schedule reconciliation, improves operational stability by allowing drivers and planners to build familiarity with routes, and provides retailers with predictable delivery intervals. 

Designing cyclic inventory replenishment and transportation schedules gives rise to cyclic inventory routing problems (CIRPs), for which incorporating stochastic demand is particularly challenging: schedules must remain repeatable across cycles while still being able to deal with retailer demand variability both within and across cycles. Early research addressed this tension indirectly by first solving a deterministic CIRP and subsequently adjusting the resulting schedules to hedge against uncertainty \citep[see, e.g.,][]{gaur2004periodic,aghezzaf2008robust}. More recent work on stochastic CIRPs has converged around three principal modeling streams. The first stream compresses stochasticity analytically by deriving closed-form expressions or constructing deterministic or parameterized surrogates for demand-dependent terms, such as order-up-to levels, safety stocks, or expected holding, backlog, and expedited transportation costs, thereby embedding the main effects of uncertainty within a deterministic cyclic framework \citep[see, e.g.,][]{raa2021multi,malicki2021cyclic,raa2023shortfall,hasturk2024stochastic}. The second stream employs chance constraints, which enforce probabilistic service-level and/or vehicle-capacity requirements, yielding nonlinear chance-constrained models that are subsequently reformulated or approximated as tractable mixed-integer programs while retaining explicit probabilistic guarantees \citep[see, e.g.,][]{sonntag2023stochastic,hasturk2024stochastic}. The third stream relies on scenario-based two-stage stochastic programming, as in \citet{rave2025cyclic}, where a finite set of demand scenarios is generated, first-stage cyclic routing and inventory decisions are fixed, and second-stage emergency shipments are optimized for each scenario, thereby directly capturing distributional impacts on both repeating schedules and scenario-specific operations. Solution approaches across these streams are predominantly heuristic or approximation-based; exact methods remain rare, although \citet{sonntag2023stochastic} demonstrate a branch-price-and-cut implementation.

Distributionally robust optimization (DRO) provides a promising framework for modeling stochastic demand in CIRPs; however, its application in this context remains limited. Existing CIRP studies that address stochastic demand, whether through analytical compression of stochasticity, chance constraints, or scenario-based stochastic programming, uniformly assume that the underlying demand distribution is known a priori. DRO relaxes this assumption by enabling decision-making under partial distributional information and optimizing against the worst-case distribution within a prespecified ambiguity set \citep{rahimian2019distributionally, kuhn2025distributionally, jiang2025optimized}. This approach effectively addresses distributional ambiguity, which arises because the true demand distribution is typically unknown or cannot be uniquely identified \citep{cai2025distributionally}, while simultaneously enhancing the robustness of the resulting solutions. This feature is particularly valuable for accommodating cycle-to-cycle demand fluctuations. Moreover, unlike conventional approaches that often assume stationary demand and may yield suboptimal decisions under nonstationary conditions, DRO offers sufficient flexibility to capture within-cycle demand nonstationarity, enabling more responsive cyclic schedules. Despite these advantages, to the best of our knowledge, DRO has only been applied to the classical inventory routing problem (IRP), but not to the CIRP. \citet{cui2023inventory} model a finite-horizon IRP with transportation cost budgets and inventory bounds, minimizing the risk of inventory violations, while routing decisions are fixed at the horizon start and replenishment follows a linear decision rule based on realized demand. Their model assumes a single vehicle with unlimited capacity and is reformulated as a deterministic linear program solved exactly. \citet{che2025robust} consider a finite-horizon IRP with fixed routing and replenishment decisions, minimizing transportation costs as well as worst-case expected holding, backorder, and overcapacity penalties. They propose a globalized DRO framework with soft budget constraints to hedge against ambiguity set misspecification and derive tractable reformulations solved using CPLEX. Compared with these studies, our work provides several advances. First, by incorporating cyclic scheduling for an infinite-horizon setting, we explicitly account for residual inventories at the end of each cycle, which eliminates the need for given initial inventories and avoids EOH effects. Second, relative to \citet{cui2023inventory}, we generalize the model to multi-vehicle planning, and relative to \citet{che2025robust}, we enable adaptive replenishments that respond to realized demand. Finally, we develop a nested branch-and-price algorithm to solve the resulting problem efficiently.

In this work, we present the distributionally robust CIRP (DR-CIRP), which employs a moment-based ambiguity set that constrains the support, mean, and variance of uncertain demand. The objective is to minimize the combined fixed vehicle cost, flexible transportation cost, and worst-case expected inventory costs, including both holding and backorder costs, over all distributions within the ambiguity set. To ensure reliable transportation service, we impose distributionally robust chance constraints, requiring that vehicle capacities are satisfied with high probability under the worst-case distribution. Our main theoretical result is to show that the worst-case expected inventory cost is attained by a multi-point distribution, which can be identified a posteriori via linear programming, and that the distributionally robust chance constraints can be reformulated as equivalent deterministic conditions. 

These results yield a tractable deterministic reformulation of the DR-CIRP, denoted as Deter-CIRP. To solve the Deter-CIRP, we develop a nested branch-and-price framework in which the first level partitions retailers into clusters, while the second level decouples routing and replenishment within each cluster into two subproblems linked through a master problem. Most existing IRP solution approaches tightly couple routing and replenishment by embedding replenishment logic directly into routing formulations. While such integrated designs can facilitate computational tractability, they restrict the flexibility needed to implement advanced inventory control policies and, as a result, may fail to respond adequately to demand variability, leading to less responsive replenishment and higher total costs. The decoupling in our nested branch-and-price framework gives rise to a new variant of the inventory management problem: cyclic replenishment for a single retailer under distributional ambiguity, which, to the best of our knowledge, has not been previously explored in the inventory management literature.

We summarize our key contributions as follows.
\begin{itemize}
    \item We propose a distributionally robust CIRP that models demand uncertainty via moment-based ambiguity sets and safeguards service reliability through distributionally robust chance constraints. We prove that (i) the worst-case expected inventory cost is realized under discrete multi-point demand distributions that can be efficiently identified via tractable linear programming, and (ii) the chance constraints admit deterministic reformulations.
    
    \item We propose a novel decomposition framework that decouples routing and replenishment decisions through a master coordination problem, enabling the integration of advanced inventory control strategies into inventory routing. This leads to the identification of a new inventory management problem, cyclic replenishment for a single retailer under distributional ambiguity, which has not been studied in the literature, and we develop an efficient lazy-constraint-based solution method for its solution.

    \item Our framework accommodates multiple service policies with minimal structural modifications. We formally define and analyze three representative service policies that extend and generalize existing CIRP models, thereby facilitating policy evaluation and selection under diverse operational settings.
    
    \item We validate the practical relevance and robustness of our framework through computational experiments on synthetic instances and a real-world case study using data from SAIC Volkswagen Automotive Co., Ltd. The results show that our best-performing approach reduces transportation costs by 32.27\% on average compared with the company’s actual operational practice.
\end{itemize}

The remainder of the paper is organized as follows. Section~\ref{s2} introduces the problem setting and formulates the distributionally robust model. Section~\ref{s3} derives deterministic reformulations of the worst-case inventory cost and chance constraints. Section~\ref{s4} develops the nested branch-and-price decomposition framework. Section~\ref{s5} reports computational results, and Section~\ref{s6} concludes the paper and outlines directions for future research. Proofs of all propositions and theorems are provided in Online Appendix~\ref{oa-proofs}.

\subsection{Notations}
We adopt the following notations throughout the paper. We denote by $\mathbb{R}$ the set of real numbers and by $\mathbb{R}^+$ the set of nonnegative reals, while $\mathbb{Z}$ and $\mathbb{Z}^+$ represent the sets of integers and nonnegative integers, respectively. For $a,b\in \mathbb{Z}$, we write $[a,b]:=\{a,a+1,\dots,b\}$ for the closed integer interval and $[a,b):=\{a,a+1,\dots,b-1\}$ for the half-open one, where by convention $[1,X]$ is abbreviated as $[X]$. The symbol $(a,b)$ always denotes an ordered pair rather than an open interval. For any $x\in\mathbb{R}$, we define its positive and negative parts by $x^+:=\max\{x,0\}$ and $x^-:=\max\{-x,0\}$. When the indices of a variable are unambiguous from the context, they may be omitted (e.g., we may write $x$ for $x_{ijkv}$). For a directed graph $\mathcal{G}=(O, A)$ and a node $o\in O$, we denote by $\eta^+(\mathcal{G},o)=\bigl\{(o,o')\in A: o'\in O\bigr\},\ \eta^-(\mathcal{G},o)=\bigl\{(o',o)\in A: o'\in O\bigr\}$ the sets of arcs outgoing from and incoming to $o$, respectively.
\section{Problem Description}\label{s2}
Consider a spatio-temporal network $\mathcal{G}=\mathcal{G}^{\rm spatial}\times\mathcal{G}^{\rm temporal}$, where the directed graphs $\mathcal{G}^{\rm spatial}$ and $\mathcal{G}^{\rm temporal}$ represent the spatial and temporal networks, respectively. The spatial network is given by $\mathcal{G}^{\rm spatial}=\bigl(O^{\rm spatial}, A^{\rm spatial}\bigr)$, where the node set $O^{\rm spatial}=\{0\}\cup[N]$ comprises the warehouse (node 0) and a set of retailers $[N]$. The arc set $A^{\rm spatial}\subseteq\bigl\{(i,j): i,j\in O^{\rm spatial}, i\neq j\bigr\}\cup\bigl\{(0,0)\bigr\}$ specifies feasible travel arcs, with the self-loop $(0,0)$ representing an empty tour that starts and ends at the warehouse. The temporal network is defined as $\mathcal{G}^{\rm temporal}=\bigl(O^{\rm temporal}, A^{\rm temporal}\bigr)$, where $O^{\rm temporal}=[T]$ indexes a discrete cyclic planning horizon that repeats indefinitely (i.e., period $T$ is followed by period $1$). The arc set $A^{\rm temporal}\subseteq O^{\rm temporal}\times O^{\rm temporal}$ encodes feasible temporal transitions, corresponding to arcs connecting two successive replenishment periods.

For each retailer $i \in [N]$ in each period $t \in [T]$, demand is represented by a discrete random variable $\zeta_i^t$ that is independent across retailers and periods. Let $\boldsymbol{\zeta}^i=[\zeta_i^1,\ldots,\zeta_i^T]$ denote the demand vector of retailer $i$, and let $\boldsymbol{\zeta}=[\boldsymbol{\zeta}^1,\ldots,\boldsymbol{\zeta}^N]$ denote the joint demand vector across all retailers. Each retailer maintains its own inventory to satisfy the stochastic demand, and $h$ and $b$ denote the per-unit holding and backorder costs, respectively. We assume that the warehouse has unlimited inventory and serves as the source for replenishing retailer inventories through a homogeneous fleet of $V$ vehicles, indexed by $[V]$, each with capacity $Q$. The vehicles depart from and return to the warehouse for each tour, which is limited by a maximum travel distance $L$, and $c_{ij}$ denotes the distance between spatial nodes $i$ and $j$. Each unit of distance traveled incurs a variable transportation cost $\rho$, and each active vehicle incurs a fixed cost $p$. Our objective is to determine an infinite-horizon cyclic plan over $[T]$ that jointly specifies inventory replenishment decisions in the temporal network and routing decisions in the spatial network to minimize the total expected cost, including the fixed vehicle activation, variable transportation, and inventory holding and backorder costs, subject to constraints on vehicle capacity and maximum travel distance per tour.

Specifically, we assume that the true distribution $\mathbb{P}$ of $\boldsymbol{\zeta}$ is unknown, but belongs to the moment-based ambiguity set $\mathcal{P}$, following \citet[Section~4.1]{ghosal2020distributionally}:
\begin{align}
\mathcal{P}=\left\{ \mathbb{P} \left |
    \begin{aligned}
    &\mathbb{P}\bigl(\zeta_t^i\in [\underline{\zeta}_t^i,\bar{\zeta}_t^i]\bigr)=1&&\forall i\in[N],t\in[T]\\
    &\mathbb{E}_{\mathbb{P}}[\zeta_t^i]=\mu_t^i&&\forall i\in[N],t\in[T]\\
    &\mathbb{E}_{\mathbb{P}} \bigl[|\zeta_t^i-\mu_t^i|\bigr]\leq \sigma_t^i&&\forall i\in[N],t\in[T]\\
    \end{aligned} \right.\nonumber
\right\}.
\end{align}
Each support set $[\underline{\zeta}_t^i,\bar{\zeta}_t^i]$ satisfies $\underline{\zeta}_t^i,\bar{\zeta}_t^i\in\mathbb{R}^+$. The mean $\mu_t^i$ lies in $(\underline{\zeta}_t^i,\bar{\zeta}_t^i)$; otherwise, the distribution collapses to a degenerate point and the problem reduces to a deterministic case. The parameter $\sigma_t^i$ specifies an upper bound on the mean absolute deviation. Under this ambiguity set, our objective is to determine a cyclic plan that minimizes the worst-case expected total cost over all distributions $\mathbb{P} \in \mathcal{P}$.

Figure~\ref{fig:coupled} illustrates the interaction between the spatial and temporal networks. Spatial decisions, which retailers to visit and in what sequence, are determined in the spatial network but must be made in each temporal period. Temporal decisions, when and how much to replenish, are determined in the temporal network but must be specified for each retailer in the spatial network. The two layers are coupled in two ways: (i) any positive replenishment decision in the temporal network must be supported by a feasible route in the spatial network, and (ii) vehicle capacity constraints link replenishment quantities in the temporal network to the corresponding spatial routes. We now present the formulation in three parts: spatial routing, temporal replenishment, and their integration.

\begin{figure}[!htb]
\centering
  \includegraphics[width=0.7\linewidth]{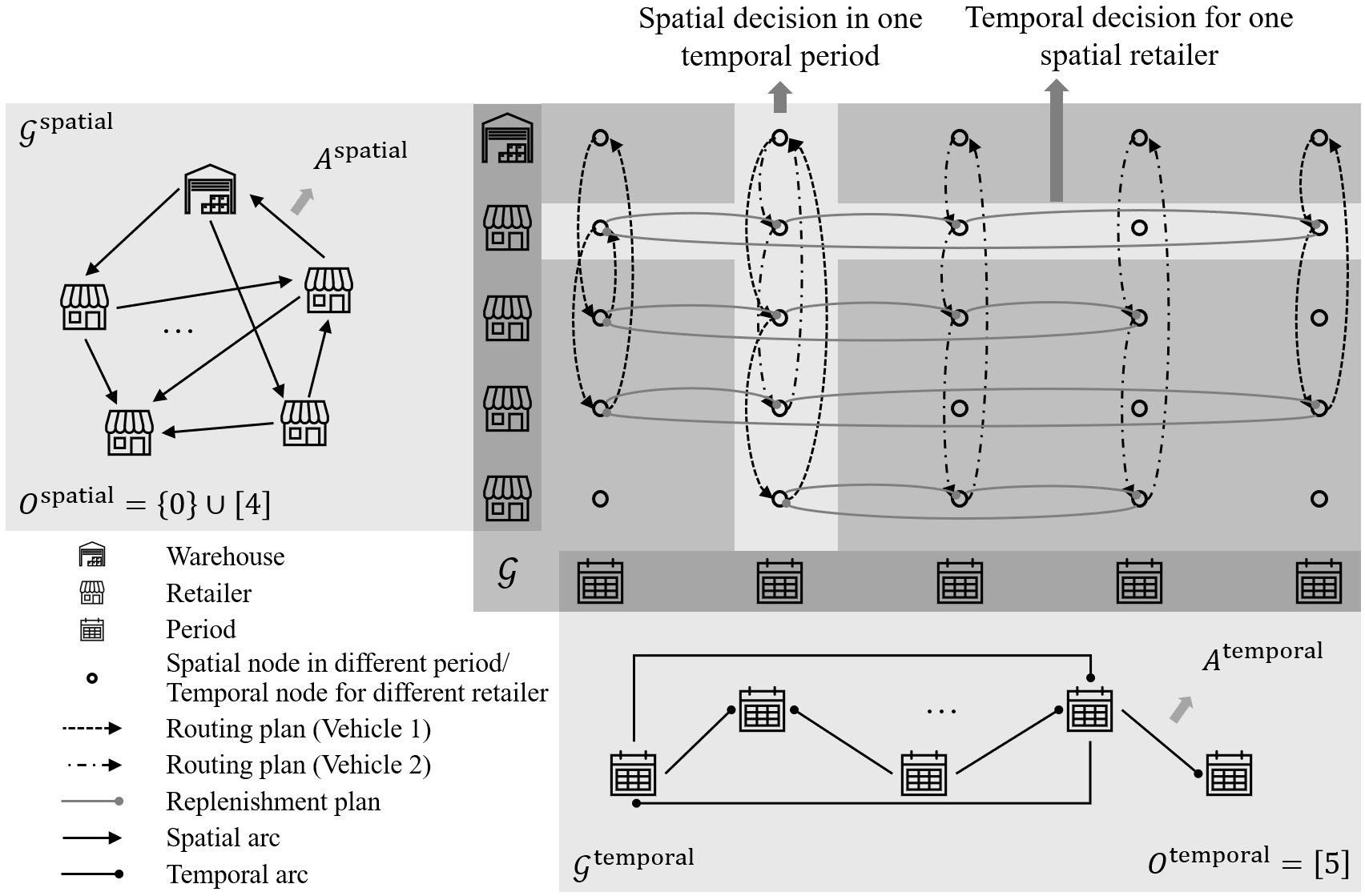}
  \caption{Interaction between spatial routing and temporal replenishment networks in DR-CIRP.}
  \label{fig:coupled}
\end{figure}

\subsection{Spatial Routing}\label{s2.1}
In the spatial dimension of the network, a central task is to determine efficient routing plans. We consider a fixed partition of the retailer set $[N]$ into mutually exclusive clusters, with each cluster permanently assigned to a dedicated vehicle over the planning horizon. Let $z_v\in\{0,1\}$ indicate whether vehicle $v\in [V]$ is activated, and let $w_{iv}\in\{0,1\}$ indicate whether retailer $i\in[N]$ is assigned to vehicle $v$. We further define binary variables $x_{ijtv}\in\{0,1\}$ to represent routing decisions, where $x_{ijtv}=1$ if vehicle $v$ traverses arc $(i,j)\in A^{\rm spatial}$ in period $t\in[T]$, and 0 otherwise. The feasible set $\mathcal{X}$ is then characterized by the following constraints: at-most-once visitation (\ref{X_feasible_set}a), warehouse departure/return (\ref{X_feasible_set}b), flow conservation (\ref{X_feasible_set}c), subtour elimination (\ref{X_feasible_set}d), a tour-length limit (\ref{X_feasible_set}e), and binary restrictions (\ref{X_feasible_set}f):
\begin{align}
\mathscr{X}=\left\{ x_{ijtv} \left|
    \begin{aligned}
    &\sum_{(i,j)\in\eta^+(\mathcal{G}^{\rm spatial},i)}x_{ijtv}\leq 1  &&\forall i\in[N], t\in[T],v\in [V],&(a)\\
  &\sum_{i\in [N]\cup\{0\}} x_{0itv} = \sum_{i\in [N]\cup\{0\}} x_{i0tv} = 1 &&\forall t\in[T],v\in [V],&(b)\\
  &\sum_{(i,j)\in \eta^+(\mathcal{G}^{\rm spatial},i)} x_{ijtv} = \sum_{(j,i)\in \eta^-(\mathcal{G}^{\rm spatial},i)} x_{jitv} &&\forall i\in [N], t\in[T], v\in[V],&(c)\\
  &\sum_{i\in S}\sum_{j\in S, j\neq i}x_{ijtv}\leq |S|-1 &&\forall S\subset [N], 2\leq|S|\leq N,t\in[T], v\in[V],&(d)\\
  & \sum_{(i,j)\in A^{\rm spatial}}c_{ij}x_{ijtv}\leq L,&&\forall v\in[V],&(e)\\
  & x_{ijtv}\in \{0,1\} &&\forall (i,j)\in A^{\rm spatial}, t\in[T],v\in[V].&(f)
    \end{aligned} \right .
\right\} \label{X_feasible_set}
\end{align}

The associated spatial costs, comprising the fixed vehicle cost and the average transportation cost per period, are defined as
\begin{align}
    C^{\rm vehicle}=\sum_{v\in [V]} p\, z_v \ \ {\rm and}\ \ C^{\rm trans}=\frac{1}{T}\sum_{v\in [V]}\sum_{t\in[T]}\sum_{(i,j)\in A^{\rm spatial}} \rho\, c_{ij}\, x_{ijtv},\nonumber
\end{align}
respectively. The corresponding spatial constraint set is formulated as follows:
\begin{align}
    &\sum_{v\in [V]} w_{iv}=1 \quad \forall i\in[N], \label{eq:partition}\\
    &\sum_{i\in[N]} w_{iv}\le N z_v \quad \forall v\in [V],\label{eq:wz}\\
    &\sum_{t\in[T]}\sum_{(i,j)\in\eta^+(\mathcal{G}^{\rm spatial},i)}x_{ijtv} \leq M w_{iv} \quad \forall i\in[N], v\in[V],\label{eq:assign_consistency}\\
    &w_{iv}\in\{0,1\}\quad  \forall i\in[N],v\in [V],\label{eq:domain1}\\
    &z_v\in\{0,1\}\quad  \forall v\in [V], \label{eq:domain2}\\
    &x_{ijtv}\in\mathscr{X}\quad\forall (i,j)\in A^{\rm spatial}, t\in [T], v\in [V],\label{eq:domain3}
\end{align}
where Constraints \eqref{eq:partition} ensure that each retailer is assigned to exactly one vehicle, Constraints \eqref{eq:wz} link vehicle activation to retailer assignments, Constraints \eqref{eq:assign_consistency} enforces assignment-to-route consistency, and Constraints \eqref{eq:domain1}-\eqref{eq:domain3} specify the variables domains.

\subsection{Temporal Replenishment}\label{s2.2}
In the temporal network, both the replenishment plan and replenishment quantities must be determined. We begin by focusing on the replenishment plan. Since each visit incurs a transportation cost, it is generally not optimal to replenish in every period. Instead, replenishments are scheduled only in selected periods to minimize the overall cost. This gives rise to \textit{cyclic replenishment intervals}, defined by consecutive replenishment visits, where each interval begins with a delivery and ends with the subsequent one; these intervals are also referred to as “inventory cycles” in \citet{chen2023robust}. We adopt this interval-based perspective to formulate the problem. To this end, we now introduce two supporting definitions.
\begin{definition}[Cyclic interval]
For $(t_1,t_2)\in A^{\rm temporal}$, the cyclic interval is defined as the ordered sequence of periods from $t_1$ to $t_2$, or wrapping around cyclically if $t_2 < t_1$. Formally,
\[
[t_1,t_2]:=
\begin{cases}
\{t_1,t_1{+}1,\ldots,t_2\}, & t_1\le t_2,\\
\{t_1,t_1{+}1,\ldots,T,1,\ldots,t_2\}, & t_1>t_2.
\end{cases}
\]
For instance, with $T=5$, one has $[3,5]=\{3,4,5\}$ and $[4,2]=\{4,5,1,2\}$.
\end{definition}

\begin{definition}[Sequence-based arithmetic and ordering]
For any interval $[t_1, t_2]$ with $(t_1,t_2)\in A^{\rm temporal}$, arithmetic and logical operations on its elements are defined relative to their positions in the ordered sequence $[t_1, t_2]$, rather than their numeric values. Specifically, for $t, t^{\prime} \in [t_1, t_2]$:
\begin{itemize}
    \item [($i$)] The difference $t^{\prime} - t$ denotes the number of forward steps from $t$ to $t^{\prime}$ when traversing $[t_1,t_2]$.
    \item [($ii$)] We write $t^{\prime} \succ t$ if $t^{\prime}$ appears after $t$ in this sequence.
    \item [($iii$)] Successor and predecessor operations (e.g., $t+1$ or $t-1$) are interpreted relative to the sequence order.
\end{itemize}
For example, in $[4,2] = \{4,5,1,2\}$, one has $1 - 4 =\bigl|\{5,1\}\bigr|= 2$ since $1$ lies two steps after $4$, and $1 \succ 4$ under the sequence order. Likewise, $5 + 1 = 1$ because $1$ follows $5$ in the cyclic sequence.
\end{definition}

We define $y_{it_1t_2v}\in\{0,1\}$ as an binary variable that equals 1 if retailer $i$ is visited by vehicle $v$ in period $t_1$ and its next visit occurs in period $t_2$. This decision variable thus encodes the replenishment interval $[t_1,t_2]$. The set of feasible replenishment plans, denoted by $\mathcal{Y}$, is characterized by coverage (\ref{Y_feasible_set}a), per-period visit limits (\ref{Y_feasible_set}b), flow conservation (\ref{Y_feasible_set}c), cycle restriction (Definition \ref{def: cycle}) (\ref{Y_feasible_set}d), and binary requirements (\ref{Y_feasible_set}e):
\begin{align}
\mathscr{Y}=\left\{y_{it_1t_2v} \left |
    \begin{aligned}
  & \sum_{v\in[V]}\sum_{(t_1,t_2)\in A^{\rm temporal}} y_{it_1t_2v} \geq 1 &&\forall i\in[N], &(a)\\
  & \sum_{(t_1,t_2)\in \eta^+(\mathcal{G}^{\rm temporal},t_1)} y_{it_1t_2v} \leq 1 &&\forall i\in[N],t_1\in [T], v\in[V], &(b)\\
  &\sum_{(t_1,t_2)\in\eta^+(\mathcal{G}^{\rm temporal},t_1)}y_{it_1t_2v} = \\
  &\qquad\qquad\qquad\sum_{(t_2, t_1)\in \eta^-(\mathcal{G}^{\rm temporal},t_1)} y_{it_2t_1v} &&\forall i\in[N], t_1\in [T], v\in[V], &(c)\\
  & \sum_{v\in[V]}\sum_{t_1\in [T]}\sum_{t_2\in[1,t_1]}y_{it_1t_2v} = 1 &&\forall i\in[N], &(d)\\
  & y_{it_1t_2v}\in \{0,1\} &&\forall i\in[N],  (t_1,t_2)\in A^{\rm temporal}, v\in[V].&(e)
    \end{aligned} \right .
\right\} \label{Y_feasible_set}
\end{align}

\begin{Defi}[Cycle restriction]
\label{def: cycle}
A replenishment plan satisfies the cycle restriction if the visit schedule of each retailer is entirely contained within a single planning cycle $[T]$ and repeats identically across subsequent cycles. For instance, when $[T] = \{1,\dots,5\}$, the sequence $1 \rightarrow 3 \rightarrow 5 \rightarrow 2 \rightarrow 4 \rightarrow 1$ spans two consecutive cycles and thus violates the restriction, whereas the sequence $1 \rightarrow 3 \rightarrow 5 \rightarrow 1$ is fully contained within one single cycle and is therefore admissible. This restriction ensures temporal consistency and periodicity of replenishment. Consequently, under the cycle restriction, any two plans composed of the same set of periods are regarded as equivalent representations of the same replenishment plan.
\end{Defi}

In addition to the inventory replenishment planning, we assume that the inventory replenishment quantities are governed by an order-up-to policy. Specifically, let $s_{it}\in\mathbb{Z}^{+}$ denote the order-up-to level for retailer $i$ in period $t$. Further, let $u_{itv}(\boldsymbol{\zeta}^i)$ denote the replenishment quantity delivered by vehicle $v$ in period $t$, and let $I_{it}(\boldsymbol{\zeta}^i)$ denote the inventory of retailer $i$ at the end of period $t$. The temporal constraints can then be formulated as follows:
\begin{align}
    &s_{it_1} \leq \sum_{v\in[V]}\sum_{(t_1,t_2)\in\eta^+(\mathcal{G}^{\rm temporal},t_1)}y_{it_1t_2v}\sum_{t\in[t_1,t_2-1]}\bar{\zeta}_t^i\quad \forall i\in[N], t_1\in[T],\label{st:sandy}\\
    &u_{itv}(\boldsymbol{\zeta}^i)=\sum_{(t_1,t_2)\in \eta^+(\mathcal{G}^{\rm temporal},t_1)} y_{it_1t_2v}\,\bigl(s_{it_2}-s_{it_1}+\sum_{t'\in [t_1,t_2{-}1]}\zeta_{t'}^i\bigr)^+,\label{eq:u}\\
    &I_{it}(\boldsymbol{\zeta}^i)=I_{i,t-1}(\boldsymbol{\zeta}^i)+\sum_{v\in [V]} u_{itv}(\boldsymbol{\zeta}^i)-\zeta_t^i,\qquad \forall i\in[N], t\in[T], \label{eq:I}\\
    &y_{it_1t_2v}\in\mathscr{Y}\quad\forall i\in[N], (t_1,t_2)\in A^{\rm temporal}, v\in[V],\label{eq:domain4}\\
    &s_{it}\in\mathbb{Z}^{+}\quad\forall i\in[N], t\in[T].\label{eq:domain5}
\end{align}
Constraints \eqref{st:sandy} captures the relationship between the order-up-to levels $s_{it}$ and the replenishment plan $y_{it_1t_2v}$, Constraints \eqref{eq:u} defines the replenishment quantity delivered by vehicle $v$ in period $t$, Constraints \eqref{eq:I} ensure the inventory balance, and Constraints \eqref{eq:domain4} and \eqref{eq:domain5} specify the feasible domains of the decision variables. Accordingly, the average worst-case expected inventory cost per period is defined as
\begin{align}
    C^{\rm DR\text{-}inv}=\frac{1}{T}\sum_{i\in[N]}\sum_{v\in [V]}\sum_{(t_1,t_2)\in A^{\rm temporal}} y_{it_1t_2v}
\sup_{\mathbb{P}\in\mathcal{P}} \mathbb{E}_{\mathbb{P}}\!\bigl[\, h\!\!\sum_{t\in [t_1,t_2{-}1]} \!\!I_{it}(\boldsymbol{\zeta}^i)^+ + b\!\!\sum_{t\in [t_1,t_2{-}1]} \!\!I_{it}(\boldsymbol{\zeta}^i)^- \bigr].\nonumber
\end{align}

\subsection{Integration of Spatial and Temporal Decisions}\label{s2.3}
Spatial and temporal decisions in the DR-CIRP are coupled through two key restrictions. The first is the \textit{visit consistency} constraint, which links routing variables $x_{ijtv}$ with replenishment plan variables $y_{it_1t_2v}$. Intuitively, if retailer $i$ is replenished by vehicle $v$ in period $t$, then $i$ must appear on $v$’s route in that period:
\begin{align}
\sum_{(i,j)\in \eta^+(\mathcal{G}^{\rm spatial},i)} x_{ijtv}
\ \ge\
\sum_{(t_1,t_2)\in \eta^+(\mathcal{G}^{\rm temporal},t_1)} y_{it_1t_2v}
\qquad \forall i\in[N],\, t\in[T],\, v\in [V].\label{eq:visit_consistency}
\end{align}
This condition is one-sided: a retailer may be visited without receiving a replenishment.  

The second coupling arises from considerations of vehicle capacity. We impose a \textit{distributionally robust capacity feasibility} constraint that ensures the shipment quantities implied by temporal decisions can be feasibly accommodated on the spatial route. Specifically, the aggregated load cannot exceed vehicle capacity $Q$ with high probability across all distributions in the ambiguity set $\mathcal{P}$. Formally, this is captured by the distributionally robust chance constraints:
\begin{align}
\inf_{\mathbb{P}\in\mathcal{P}}
\mathbb{P}\!\Bigl[
\sum_{i\in[N]}\ \sum_{(t_1,t_2)\in \eta^+(\mathcal{G}^{\rm temporal},t_1)}
y_{it_1t_2v}\,\bigl(s_{it_2}-s_{it_1}+\!\!\sum_{t'\in [t_1,t_2{-}1]}\!\! \zeta_{t'}^i\bigr)^+ \ \le\ Q
\Bigr] \ \ge\ 1-\varepsilon_1,\quad \forall t\in[T],\, v\in [V].
\label{eq:chance}
\end{align}

\begin{remark}[Emergency transportation under rare violations]
To accommodate the small probability $\varepsilon_1$ of violating \eqref{eq:chance}, we allow outsourced (emergency) transportation at a fixed per-unit cost to fulfill any excess load. For clarity, however, we do not explicitly include emergency transportation costs in the main model. Nevertheless, the model can be easily extended to incorporate such costs, as described in Appendix \ref{etc}. Section~\ref{5.2.2} reports a sensitivity analysis of realized emergency costs across different $\varepsilon_1$.$\hfill\blacksquare$
\end{remark}

Bringing these components together, the DR-CIRP can be formally written as
\begin{align}
\min_{x,\ y,\ s}\left\{C^{\rm vehicle}+C^{\rm trans}+C^{\rm DR\textit{-}inv}\big|{\rm Constraints}\ \eqref{eq:partition}-\eqref{eq:domain3},\ \eqref{st:sandy}-\eqref{eq:chance}\right\}.\label{dr-cirp}
\end{align}


\subsection{Special Cases}\label{s2.4}
The DR-CIRP formulation permits fully flexible routing and replenishment decisions, without prescribing how routes or replenishment plans must be organized. Several service policies studied in the CIRP literature can be obtained as tractable special cases of this general framework by introducing additional organizational assumptions. We highlight three representative cases. To this end, we first state several assumptions that will be used throughout.

\begin{assumption}[Stationary demand]\label{a1}
Retailer demand is stationary across time: $\zeta_t^i\equiv \zeta^i$ for all $i\in[N]$, $t\in[T]$.
\end{assumption}

\begin{assumption}[Fixed shipment interval]\label{a3}
All retailers served by the same vehicle share a common visit frequency: there exists $\kappa_v\in\{1,\ldots,T\}$, such that for any $t_1,t_2\in\mathcal{T}_v^\ast$, $t_2\%\kappa_v=t_1\%\kappa_v$. $\mathcal{T}_v^\ast$ denotes the set of periods in which vehicle $v$ is active.
\end{assumption}

\begin{assumption}[Route consistency]\label{a2}
Each vehicle executes the same route whenever it operates: $x_{ijt_1 v}=x_{ijt_2 v}=:x_{ijv}$ for all $(i,j)\in A^{\rm spatial}$, $t_1,t_2\in\mathcal{T}_v^\ast$, $v\in [V]$.
\end{assumption}

Different combinations of these assumptions yield the following canonical service policies:

\begin{itemize}
\item \textit{Fixed-Interval Policy.} Under Assumptions~\ref{a1}-\ref{a2}, vehicle $v$ follows a fixed tour that is repeated every $\kappa_v$ periods. Each retailer $i$ implements a periodic-review order-up-to policy $(\kappa_v,s_i)$. This policy has been analyzed, for example, by \citet{sonntag2023stochastic} in the setting of stochastic demand with known stationary distributions.

\item \textit{Consistent Policy.} Imposing only Assumption~\ref{a2} fixes the route of each vehicle while allowing visit intervals to vary across the horizon. This policy has been studied by \citet{hasturk2024stochastic} in the context of stochastic demand with known stationary distributions.

\item \textit{Flexible Policy.} With no additional restrictions, corresponding to the baseline DR-CIRP formulation, both routing and replenishment schedules are fully flexible. To the best of our knowledge, such flexible strategies have only been considered by \citet{malicki2021cyclic}, who assume normally distributed demand and propose heuristic large neighborhood search algorithms.
\end{itemize}

In this way, the DR-CIRP provides a unifying framework that encompasses flexible, consistent, and fixed-interval service policies. Our computational study will compare these policies within a common setting and highlight their relative performance.
\section{Deterministic Reformulation}\label{s3}
This section derives tractable deterministic counterparts of the DR-CIRP model in \eqref{dr-cirp}. Two components require reformulation: (i) the distributionally robust chance constraints that guarantee vehicle capacity feasibility, and (ii) the worst-case expected inventory cost \(C^{\rm DR\text{-}inv}\). We address these in turn.

\subsection{Distributionally Robust Chance Constraints}
Given the replenishment quantity defined in \eqref{eq:u}, the distributionally robust capacity feasibility condition \eqref{eq:chance} can be written as
\begin{align}
  \inf_{\mathbb{P}\in\mathcal{P}}\mathbb{P}\Bigl[\sum_{i\in[N]}\sum_{(t_1,t_2)\in \eta^+(\mathcal{G}^{\rm temporal},t_1)}y_{it_1t_2v}\bigl(s_{it_2}-s_{it_1}+\sum_{t^{\prime}\in[t_1,t_2-1]} \zeta^i_{t^{\prime}}\bigr)^+\leq Q\Bigr]\geq 1-\varepsilon_1 \quad \forall t_1\in[T], v\in[V].\label{st_cc}
\end{align}
The nonnegativity operator $(\cdot)^+$ ensures that only positive delivery amounts are enforced, preventing infeasibility when cumulative demand is low and inventory overshoots the order-up-to level. While operationally realistic, the non-smoothness of $(\cdot)^+$ complicates both analysis and computation.

To derive a tractable relaxation, we relax \eqref{st_cc} with two separate chance constraints. The first protects vehicle capacity,
\begin{align}
  \inf_{\mathbb{P}\in\mathcal{P}}\mathbb{P}\Bigl[\sum_{i\in[N]}\sum_{(t_1,t_2)\in \eta^+(\mathcal{G}^{\rm temporal},t_1)}y_{it_1t_2v}\bigl(s_{it_2}-s_{it_1}+\sum_{t^{\prime}\in[t_1,t_2-1]} \zeta^i_{t^{\prime}}\bigr) \leq Q\Bigr]\geq 1-\varepsilon_1, \quad \forall t_1\in[T],\ v\in[V], \label{st_cc1}
\end{align}
while the second bounds the probability of inventory overshoot,
\begin{align}
 \sup_{\mathbb{P}\in\mathcal{P}}\mathbb{P}\Bigl[\sum_{v\in[V]}\sum_{(t_1,t_2)\in \eta^+(\mathcal{G}^{\rm temporal},t_1)}y_{it_1t_2v}\bigl(s_{it_2}-s_{it_1}+\sum_{t^{\prime}\in[t_1,t_2-1]} \zeta^i_{t^{\prime}}\bigr)\leq 0\Bigr]\leq \varepsilon_2, \quad \forall i\in[N],\ t_1\in[T], \label{st_cc2}
\end{align}
where, since each retailer is served by at most one vehicle per period, the summation over \(v\in[V]\) in \eqref{st_cc2} effectively identifies the assigned vehicle, with all other terms vanishing due to \(y_{it_1t_2v}=0\). 

Inventory overshoot is a well-documented phenomenon in order-up-to policies \citep[see, e.g.,][]{chen2020dynamic}. While eliminating overshoot enhances control, it may reduce responsiveness. Moreover, overshooting under demand uncertainty remains analytically challenging. Following the argument in \citet{mohebbi2003supply} that moderate overshoot is acceptable, we adopt a soft-constraint formulation: Constraints~\eqref{st_cc2} limit the probability of overshoot to a user-defined threshold $\varepsilon_2$, where setting $\varepsilon_2 = 0$ enforces a hard constraint and yields a fully equivalent deterministic reformulation. Later numerical experiments demonstrate how the changes in $\epsilon_2$ affect the overshooting.

To reformulate \eqref{st_cc1} and \eqref{st_cc2}, we introduce two quantile-based risk measures that capture the upper- and lower-tail behavior of random variable $z$ following distribution $\mathbb{Q}$:
\begin{itemize}
    \item [($i$)] The upper tail is measured by the value-at-risk (VaR) at level $1-\varepsilon_1$:
    \begin{align}
      \mathbb{Q}\text{-}{\rm VaR}_{1-\varepsilon_1}[\mathcal{Z}] := \inf\bigl\{z\in\mathbb{R} : \mathbb{Q}[\mathcal{Z} \leq z] \geq 1 - \varepsilon_1\bigr\}.\label{st:444}
    \end{align}
    \item [($ii$)] The lower tail is measured by the left quantile (LQ) at level $\varepsilon_2$:
    \begin{align}
      \mathbb{Q}\text{-}{\rm LQ}_{\varepsilon_2}[\mathcal{Z}] := \sup\bigl\{z\in\mathbb{R} : \mathbb{Q}[\mathcal{Z} \leq z] \leq \varepsilon_2\bigr\}.\label{st:555}
    \end{align}
\end{itemize}
These definitions imply the following logical equivalences:
\begin{align}
  \mathbb{Q}[\mathcal{Z} \leq \tau] \geq 1-\varepsilon_1 \Leftrightarrow \mathbb{Q}\text{-}{\rm VaR}_{1-\varepsilon_1}[\mathcal{Z}] \leq \tau,\quad
  \mathbb{Q}[\mathcal{Z} \leq \tau] \leq \varepsilon_2 \Leftrightarrow \mathbb{Q}\text{-}{\rm LQ}_{\varepsilon_2}[\mathcal{Z}] \geq \tau.\nonumber
\end{align}

Under the ambiguity set $\mathcal{P}$ and the i.i.d. assumptions, the Chance Constraints \eqref{st_cc1} and \eqref{st_cc2} admit the following equivalent reformulations:
\begin{align}
    &\sum_{i\in[N]}\sum_{(t_1,t_2)\in \eta^+(\mathcal{G}^{\rm temporal},t_1)}y_{it_1t_2v}\sum_{t^{\prime}\in[t_1,t_2-1]} \sup_{\mathbb{P}\in\mathcal{P}}\mathbb{P}\text{-}{\rm VaR}_{1-\varepsilon_1}[\zeta^i_{t^{\prime}}]\leq Q-\sum_{i\in[N]}\sum_{(t_1,t_2)\in \eta^+(\mathcal{G}^{\rm temporal},t_1)}y_{it_1t_2v}(s_{it_2}-s_{it_1}), \nonumber\\
    &\sum_{v\in[V]}\sum_{(t_1,t_2)\in \eta^+(\mathcal{G}^{\rm temporal},t_1)}y_{it_1t_2v}\sum_{t^{\prime}\in[t_1,t_2-1]} \inf_{\mathbb{P}\in\mathcal{P}}\mathbb{P}\text{-}{\rm LQ}_{\varepsilon_2}[\zeta^i_{t^{\prime}}]\geq -\sum_{v\in[V]}\sum_{(t_1,t_2)\in \eta^+(\mathcal{G}^{\rm temporal},t_1)}y_{it_1t_2v}(s_{it_2}-s_{it_1}).\nonumber
\end{align}
These results follow from the additive structure of worst-case risk measures under the specified ambiguity set, as established in Theorem 3 of \citet{ghosal2020distributionally}. 

We further have the following proposition, which is a direct consequence of Proposition 2 of \citet{ghosal2020distributionally}:
\begin{proposition}[\citet{ghosal2020distributionally}]
\label{VaR_Prop}
Under the ambiguity set $\mathcal{P}$, the worst-case VaR and the worst-case LQ admit closed-form expressions:
\begin{align}
    &\sup_{\mathbb{P}\in\mathcal{P}}\mathbb{P}\text{-}{\rm VaR}_{1-\varepsilon_1}[ \zeta^i_{t}]=\mu_t^i+\min\bigl\{\bar{\zeta}_t^i-\mu_t^i,\frac{1-\varepsilon_1}{\varepsilon_1}(\mu_t^i-\underline{\zeta}_t^i),\frac{1}{2\varepsilon_1}\sigma_t^i\bigr\}=:U_t^i,\nonumber\\
  &\inf_{\mathbb{P}\in\mathcal{P}}\mathbb{P}\text{-}{\rm LQ}_{\varepsilon_2}[ \zeta^i_{t}]=\mu_t^i-\min\bigl\{\frac{1-\varepsilon_2}{\varepsilon_2}(\bar{\zeta}_t^i-\mu_t^i),\mu_t^i-\underline{\zeta}_t^i, \frac{1}{2\varepsilon_2}\sigma_t^i\bigr\}=:L_t^i.\nonumber
\end{align}
\end{proposition}

Substituting into \eqref{st:444} and \eqref{st:555}, we obtain deterministic constraints:
\begin{align}
    & \sum_{i\in[N]}\sum_{(t_1,t_2)\in \eta^+(\mathcal{G}^{\rm temporal},t_1)}y_{it_1t_2v}\bigl(s_{it_2}-s_{it_1}+\sum_{t^{\prime}\in[t_1,t_2-1]}U_{t^{\prime}}^i\bigr)\leq Q \quad \forall t_1\in [T],v\in[V],\label{st:111}\\
   & \sum_{v\in[V]}\sum_{(t_1,t_2)\in \eta^+(\mathcal{G}^{\rm temporal},t_1)}y_{it_1t_2v}\bigl(s_{it_2}-s_{it_1}+ \sum_{t^{\prime}\in[t_1,t_2-1]} L_{t^{\prime}}^i\bigr)\geq 0 \quad \forall i\in[N], t_1\in[T].\label{st:222}
\end{align}

\subsection{Worst-Case Expected Inventory Cost}
Recall that the worst-case expected inventory cost 
\begin{align}
C^{\rm DR-inv}=\frac{1}{T}\sum_{i\in[N]}\sum_{v\in [V]}\sum_{(t_1,t_2)\in A^{\rm temporal}} y_{it_1t_2v}f^{\rm inv}(i, t_1, t_2, s_{it_1})\nonumber,
\end{align}
where the interval-wise inventory cost is defined as
\begin{align}
    f^{\rm inv}(i, t_1, t_2, s_{it_1}) := \sup_{\mathbb{P}\in\mathcal{P}} \mathbb{E}_{\mathbb{P}}\!\bigl[\, h\!\!\sum_{t\in [t_1,t_2{-}1]} \!\!I_{it}(\boldsymbol{\zeta}^i)^+ + b\!\!\sum_{t\in [t_1,t_2{-}1]} \!\!I_{it}(\boldsymbol{\zeta}^i)^- \bigr].\nonumber
\end{align}

Assuming no overshoot occurs (recall that Chance Constraints \eqref{st_cc2} are imposed to limit the overshoot probability), the initial inventory of retailer $i$ at period $t_1$ is $s_{it_1}$. Since no replenishment takes place within the interval $[t_1, t_2 - 1]$, i.e., $\sum_{v\in[V]} u_{itv}(\boldsymbol{\zeta}^i) = 0$ for all $t\in[t_1, t_2 - 1]$, inventory evolves solely in response to demand. From Constraints~\eqref{eq:I}, the inventory level at any period $t \in [t_1, t_2 - 1]$ satisfies
\begin{align}
    I_{it}(\boldsymbol{\zeta}^i)=s_{it_1}-\sum_{t^{\prime}\in [t_1, t]} \zeta^i_{t^{\prime}}.\nonumber
\end{align}
Substituting into $f^{\rm inv}(i, t_1, t_2, s_{it_1})$ yields
\begin{align}
    f^{\rm inv}(i, t_1, t_2, s_{it_1}) := \sup_{\mathbb{P}\in\mathcal{P}}\mathbb{E}_{\mathbb{P}}\Bigl[h\sum_{t\in [t_1,t_2-1]}\bigl(s_{it_1}-\sum_{t^{\prime}\in [t_1,t]} \zeta_{t^{\prime}}^i\bigr)^++b\sum_{t\in [t_1,t_2-1]}\bigl(s_{it_1}-\sum_{t^{\prime}\in [t_1,t]} \zeta_{t^{\prime}}^i\bigr)^-\Bigr].\nonumber
\end{align}

We first establish a structural property:
\begin{proposition}\label{convex_prop}
The function $f^{\rm inv}(i, t_1, t_2, s_{it_1})$ is convex in $s_{it_1}$.
\end{proposition}

To compute $f^{\rm inv}(i, t_1, t_2, s_{it_1})$, one needs to solve a worst-case expectation problem over the ambiguity set. The following theorem provides a tractable characterization of this function. Under this characterization, the worst-case expectation inventory cost is denoted by $C^{\rm inv}$. 
\begin{theorem}\label{wd_theorem}
For retailer $i$ and interval $[t_1,t_2]$, given $s_{it_1}$, the worst-case inventory cost $f^{inv}(i, t_1, t_2, s_{it_1})$ equals the optimal value of the following linear optimization problem LOP:
\begin{align}
({\rm LOP})\quad  \max\ & \sum_{\hat{t}\in[t_1,t_2-1]}\Bigl(h\sum_{t\in[t_1,\hat{t})}\bigl(\pi_{\hat{t}}s_{it_1}-\sum_{t^{\prime}\in [t_1,t]}\lambda_{\hat{t}t^{\prime}}\bigr)+b\sum_{t\in[\hat{t},t_2-1]}\bigl(\sum_{t^{\prime}\in [t_1,t]}\lambda_{\hat{t}t^{\prime}}-\pi_{\hat{t}}s_{it_1}\bigr)\Bigr)+\nonumber\\
  &  h\sum_{t\in[t_1,t_2-1]}\bigl(\pi_{0}s_{it_1}-\sum_{t^{\prime}\in [t_1,t]}\lambda_{0t^{\prime}}\bigr)\nonumber\\
    s.t.\ &\sum_{\hat{t}\in[t_1,t_2-1]\cup \{0\}}\pi_{\hat{t}}=1,\nonumber\\
    &\sum_{\hat{t}\in[t_1,t_2-1]\cup \{0\}}\lambda_{\hat{t}t}=\mu_{t}^i\quad\forall t\in [t_1,t_2-1],\nonumber\\
    & \sum_{\hat{t}\in[t_1,t_2-1]\cup \{0\}} |\lambda_{\hat{t}t}-\pi_{\hat{t}}\mu_{t}^i|\leq \sigma_{t}^i\quad\forall t\in [t_1,t_2-1],\nonumber\\
    & \pi_{\hat{t}}\underline{\zeta}_{t}^i\leq \lambda_{\hat{t}t}\leq \pi_{\hat{t}}\bar{\zeta}_{t}^i\quad \forall t\in [t_1,t_2-1],\hat{t}\in[t_1,t_2-1]\cup \{0\},\nonumber\\
    &\sum_{t\in[t_1,\hat{t})}\lambda_{\hat{t}t}\leq \pi_{\hat{t}}s_{it_1}\leq \sum_{t\in[t_1,\hat{t}]}\lambda_{\hat{t}t} \quad\forall \hat{t}\in[t_1,t_2-1],\nonumber\\
    &\sum_{t\in[t_1,t_2-1]}\lambda_{0t}\leq \pi_{0}s_{it_1},\nonumber\\
    &\pi_{\hat{t}}\in\mathbb{R}_+, \boldsymbol{\lambda}_{\hat{t}}\in\mathbb{R}^{\hat{T}}\ \forall \hat{t}\in[t_1,t_2-1]\cup \{0\},\nonumber
\end{align}
where $\hat{T} = t_2 - t_1$ if $t_2>t_1$, and $\hat{T}=t_2+T-t_1$ otherwise, representing the length of the interval.
\end{theorem}

\begin{remark}[LOP Formulation Analysis]
The (LOP) formulation admits an insightful interpretation. If we regard $\pi_{\hat{t}}$ as a probability and $\boldsymbol{\lambda}_{\hat{t}}/\pi_{\hat{t}}$ as the realized demand associated with probability $\pi_{\hat{t}}$, then the objective can be understood as a weighted sum of inventory costs under different demand realizations. The constraints also have a clear meaning. The first constraint ensures that the probabilities sum to one, while the second to fourth constraints capture the restrictions specified in the ambiguity set. Finally, the last two constraints define the turning period $\hat{t}$, which we introduce next. \hfill$\blacksquare$
\end{remark}

Define the \emph{turning period} $\hat{t}$ as the earliest period when inventory depletes to or gets below zero, i.e., $\sum_{t\in[t_1,\hat{t})}\zeta_t^i < s_{it_1} \leq \sum_{t\in[t_1,\hat{t}]}\zeta_t^i$. If the inventory suffices throughout the interval, we set $\hat{t}=0$ by convention. Here, recall that $[T]=\{1,...,T\}$; thus, using $\hat{t}=0$ explicitly indicates this special case. The turning period concept plays a key role in interpreting the worst-case cost and its corresponding distribution:
\begin{corollary}\label{wd}
Let $\mathbb{P}^*$ be a discrete maximum-$(\hat{T}+1)$-point distribution defined as 
\begin{align}
    \mathbb{P}^*=\sum_{\hat{t}\in[t_1,t_2-1]\cup\{0\},\ \pi^{*}_{\hat{t}}\neq 0}\pi_{\hat{t}}^*\delta_{\frac{\boldsymbol{\lambda}_{\hat{t}}^*}{\pi_{\hat{t}}^*}}, \nonumber
\end{align}
where $\pi_{\hat{t}}^{*}$ and $\boldsymbol{\lambda}_{\hat{t}}^{*}$ are optimal solutions of problem (LOP), and $\delta_{\boldsymbol{\zeta}^{\prime}}$ denotes the Dirac distribution that places unit mass on the demand realization $\boldsymbol{\zeta}=\boldsymbol{\zeta}^{\prime}$. Then: 

    ($i$) $\mathbb{P}^*\in\mathcal{P}$;
    
     ($ii$)$\mathbb{E}_{\mathbb{P}^*}\bigl[h\sum_{t\in [t_1,t_2-1]}(s_{it_1}-\sum_{t^{\prime}\in [t_1,t]} \zeta_{t^{\prime}}^i)^++b\sum_{t\in [t_1,t_2-1]}(s_{it_1}-\sum_{t^{\prime}\in [t_1,t]} \zeta_{t^{\prime}}^i)^-\bigl]= f^{\rm inv}(i, t_1, t_2, s_{it_1})$. 
\end{corollary}

\begin{remark}[Sparsity Structure of Worst-Case Distributions]
The structure of the worst-case distribution $\mathbb{P}^*$ in Corollary~\ref{wd}, supported on at most $\hat{T}+1$ demand realizations, exhibits a sparsity property reminiscent of Carathéodory’s theorem in convex analysis. While the classical theorem pertains to linear optimization over distributions constrained by support and mean, our formulation extends this setting to include piecewise-linear cost terms and deviation constraints on the first moment. Nonetheless, the LP reformulation in Theorem~\ref{wd_theorem} ensures that the worst-case distribution can still be expressed as a convex combination of at most $\hat{T}+1$ extreme scenarios, thus preserving this structural sparsity.$\hfill\blacksquare$
\end{remark}

When the problem LOP admits alternative optimal solutions, we choose the one with the fewest active turning periods for better interpretability. We now examine two special cases. Illustrative examples of worst-case distributions for different order-up-to levels $s$ are provided in Section \ref{5.2.2}.

\begin{proposition}\label{p4}
If $\boldsymbol{\sigma}^i = \boldsymbol{0}$, the problem becomes deterministic and admits a unique turning period. If $\boldsymbol{\sigma}^i\geq \max\{\boldsymbol{\bar{\zeta}}^i-\boldsymbol{\mu}^i, \boldsymbol{\mu}^i-\boldsymbol{\underline{\zeta}}^i\}$ and the scaling condition $m(\boldsymbol{\bar{\zeta}}^i-\boldsymbol{\mu}^i)=n(\boldsymbol{\mu}^i-\boldsymbol{\underline{\zeta}}^i)$ holds for some $m,n\geq 0$, then the worst-case distribution reduces to a two-point distribution: 
\begin{align}
    \frac{m}{m+n}\delta_{\bar{\boldsymbol{\zeta}}^i}+\frac{n}{m+n}\delta_{\underline{\boldsymbol{\zeta}}^i}.\nonumber
\end{align}
\end{proposition}

Finally, we obtain the deterministic reformulation of DR-CIRP, dneoted Deter-CIRP:
\begin{align}
\min_{x,\ y,\ s}\left\{C^{\rm vehicle}+C^{\rm trans}+C^{\rm inv}\big|{\rm Constraints}\ \eqref{eq:partition}-\eqref{eq:domain3},\ \eqref{st:sandy},\ \eqref{eq:domain4}-\eqref{eq:visit_consistency},\ \eqref{st:111},\ \eqref{st:222}\right\}.\label{deter-cirp}
\end{align}

\section{Nested Branch-and-Price}\label{s4}
Although deterministic, the Deter-CIRP remains NP-hard. Nevertheless, its well-structured formulation on a spatio-temporal network offers a strong basis for decomposition. In particular, the Deter-CIRP admits natural decompositions along both spatial and temporal dimensions, as well as within each dimension. Building upon these properties, we develop a nested branch-and-price decomposition framework.

In this section, we elaborate on our nested decomposition framework. While the service policies described in Section~\ref{s2.4} affect certain formulations, their influence on the overall framework is limited. Therefore, we present the framework under the consistent policy, which entails greater technical challenges than the alternatives. Details for the fixed-interval and flexible policies are provided in Online Appendix~\ref{OA-policies}.

\subsection{Decomposition Reformulation}
The nested branch-and-price framework is illustrated in Figure~\ref{fig: nested_CG}. The restricted master problem (RMP) in the first level (1st-RMP) seeks to determine a partition of retailers into clusters, where each cluster is assigned to a single vehicle for service. Since the set of possible retailer clusters is not exhaustive, the pricing problem (PP) generates new clusters that may improve the current solution of the 1st-RMP. Nevertheless, even after this first-level decomposition, the first-level PP remains computationally challenging to solve directly. Thus, we design a second-level branch-and-price for 1st-PP. At this level, two pricing problems are defined: the routing pricing problem (2nd-PP1), which generates feasible vehicle routes, and the replenishment pricing problem (2nd-PP2), which determines replenishment plans. The corresponding second-level restricted master problem (2nd-RMP) integrates these two components to identify a combination of routing and replenishment plans, thereby defining a retailer cluster and generating a new column for the 1st-RMP. In the following, we introduce the reformulated problems one by one.

\begin{figure}[!htb]
\centering
  \includegraphics[width=0.8\linewidth]{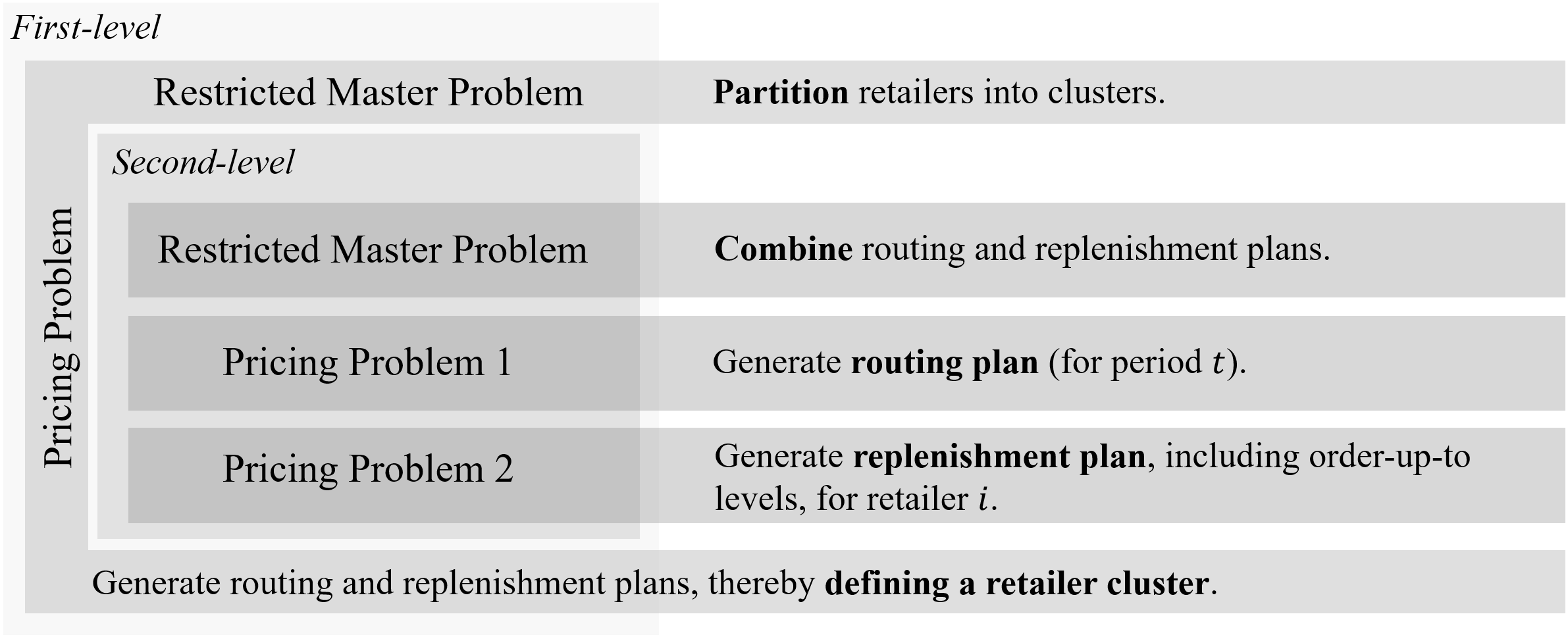}
  \caption{Illustration of the nested branch-and-price framework.}
  \label{fig: nested_CG}
\end{figure}

\subsubsection{First Level RMP: Retailers Partitioning.}\label{1st-rmp}
Define a pattern $\omega \in \Omega$ as a feasible combination of a cyclic routing plan and replenishment plans for a cluster of retailers. Let binary variable $q_\omega$ indicate the selection of pattern $\omega$, and $c_\omega$ its associated cost. The parameter $\alpha^i_\omega$ equals 1 if retailer $i$ is served under pattern $\omega$. The first-level restricted master problem is thus:
\begin{align}
{\rm (1st-RMP)}\ \min\ &\sum_{\omega\in \Omega}c_{\omega} q_{\omega}, \\
s.t. \ 
  & \sum_{\omega\in \Omega}\alpha_{\omega}^iq_{\omega} = 1 \quad \forall i\in[N],\label{master_dual}\\
  & q_{\omega} \in \{0,1\} \quad \forall \omega\in \Omega.
\end{align}

Let $\pi_i$ denote the dual variable corresponding to Constraints~\eqref{master_dual}. Instead of explicitly modeling the first-level pricing problem, we directly formulate the second-level restricted master problem (2nd-RMP), which integrates both routing and replenishment decisions for a retailer cluster.

\subsubsection{Second Level RMP: Retailer Cluster Defining.}\label{2-rmp}
Define a replenishment plan $l(i) \in L(i)$ as the replenishment decision for retailer $i \in [N]$, which specifies both the replenishment periods and the corresponding order-up-to levels. Let $\gamma_{l(i)}$ be a binary decision variable indicating whether replenishment plan $l(i)$ is selected for retailer $i$, and let $c_{l(i)}$ denote the total inventory cost incurred when retailer $i$ follows plan $l(i)$. We introduce parameter $\beta_{l(i)}^t$, which equals 1 if plan $l(i)$ involves replenishment during period $t$, and 0 otherwise. Similarly, let $s_{l(i)}^t$ represent the order-up-to level of retailer $i$ in period $t$ under plan $l(i)$. 
For the routing component, let $r \in R$ denote a feasible routing plan. Define $\lambda_r$ as a binary decision variable indicating whether routing plan $r$ is selected, and let $c_r$ be the transportation cost associated with plan $r$. The parameter $o_r^i$ equals 1 if retailer $i$ is visited under routing plan $r$, and 0 otherwise. Furthermore, let $x_r^t$ be a binary variable indicating whether routing plan $r$ is executed in period $t$. The formulation of the second-level restricted master problem (2nd-RMP) is presented as follows:
\begin{align}
{\rm (2nd-RMP)}\ \min\ &p+\frac{1}{T}\sum_{t\in[T]}\sum_{r\in R} c_{r}x_{r}^t+\sum_{i\in[N]}\sum_{l(i)\in L(i)}\bigl(\frac{1}{T}c_{l(i)}-\pi_i\bigr)\gamma_{l(i)}, \\
s.t. \ 
  & \sum_{r\in R}\lambda_{r} = 1, \label{iota}\\
  & \sum_{l(i)\in L(i)}\gamma_{l(i)} \leq 1 \quad\forall i\in [N], \\
  & \sum_{r\in R}x_{r}^to_{r}^i-\sum_{l(i)\in L(i)}\gamma_{l(i)}\beta_{l(i)}^t\geq 0\quad\forall i\in[N],t\in[T],\\
  & \sum_{i\in[N]}\sum_{l(i)\in L(i)}\gamma_{l(i)}\beta_{l(i)}^t\bigl(s_{l(i)}^{t}-s_{l(i)}^{t^-}+\sum_{t^{\prime}\in[t^-,t-1]}U_{t^{\prime}}^i\bigr)\leq Q \quad \forall t\in [T],\\ 
  & \lambda_r - x_{r}^t \geq 0 \quad\forall r\in R,  t\in[T], \label{alpha}\\
  & \lambda_{r} \in \{0,1\}\quad \forall r\in R,\\
  & x_r^t \in \{0,1\}\quad \forall t\in [T],r\in R,\\
  & \gamma_{l(i)} \in \{0,1\} \quad \forall i\in[N], l(i)\in L(i),
\end{align}
where $t^-$ and $t^+$ denote the preceding and succeeding replenishment periods of period $t$, respectively, under a cyclic replenishment policy. Let $\iota$, $\theta_i$, $\delta_{it}$, $\psi_t$, and $\alpha_{rt}$ denote the dual variables associated with Constraints~\eqref{iota} through~\eqref{alpha}, respectively.

\subsubsection{Second Level PP1: Routing Plan Generating.}
Note that Equation~\eqref{alpha} in the 2nd-RMP introduces column-dependent rows involving the variables $x_r^t$, which complicates the reduced cost calculation for newly generated routing plans, as it would typically require estimating the dual variables $\alpha_{rt}$ associated with these constraints. However, Theorem~\ref{rc_theorem} provides a tractable way to compute the reduced cost without explicitly evaluating the duals $\alpha_{rt}$.
\begin{theorem}\label{rc_theorem}
The dual variable $\alpha_{rt}$ corresponding to Constraints~\eqref{alpha} can be evaluated as:
\[
\alpha_{rt} = \min \bigl\{ \frac{1}{T}c_r - \sum_{i \in [N]} o_r^i \delta_{it},\; 0 \bigr\}.
\]
Consequently, the reduced cost of a new routing column $\lambda_r$ is:
\[
-\iota + \sum_{t \in [T]} \min \bigl\{ \frac{1}{T}c_r - \sum_{i \in [N]} o_r^i \delta_{it},\; 0 \bigr\}.
\]

Furthermore, if $\alpha_{rt} = \frac{1}{T}c_r - \sum_{i \in [N]} o_r^i \delta_{it}$, then the route $r$ is executed in period $t$, i.e., $\lambda_r = x_r^t$.
\end{theorem}

Based on Theorem~\ref{rc_theorem}, the pricing problem for routing plan generation, denoted as 2nd-PP1, can be formulated as:
\begin{align}
{\rm (2nd-PP1)}\ \min\ & -\iota+\sum_{t\in [T]}\min\Bigl\{\sum_{(i,j)\in A^{\rm spatial}}\bigl(\frac{1}{T}\rho c_{ij}-\delta_{it}\bigr)x_{ij}, 0\Bigr\},\nonumber\\
    s.t.\ 
  & \sum_{(i,j)\in A^{\rm spatial}} c_{ij} x_{ij} \leq L ,\nonumber\\
  & x_{ij}\in \mathscr{X} \quad \forall (i,j)\in A^{\rm spatial}\nonumber.
\end{align}

Theorem~\ref{rc_theorem} also indicates the relationship between the value of $\alpha_{rt}$ and route execution. This characterization facilitates a structured procedure for route generation by enumerating all possible subsets of periods during which a route may be active. For each candidate execution period set $\mathcal{\hat{T}} \subseteq [T]$, we solve a subproblem to construct a corresponding route. If the resulting route yields a negative reduced cost, it is added to the master problem. The subproblem associated with a fixed execution set $\mathcal{\hat{T}}$ is formulated as follows:
\begin{align}
\min\ & -\iota+\sum_{t\in \mathcal{\hat{T}}}\sum_{(i,j)\in A^{\rm spatial}}\bigl(\frac{1}{T}\rho c_{ij}-\delta_{it}\bigr)x_{ij},\nonumber\\
    s.t.\ 
  & x_{ij}\in \mathscr{X} \quad \forall (i,j)\in A^{\rm spatial}\nonumber,
\end{align}
which corresponds to an elementary shortest path problem with resource constraints (ESPPRC). This problem can be efficiently solved using a labeling algorithm \citep[see, e.g.,][]{feillet2004exact, costa2019exact}, as detailed in Online Appendix \ref{oa-labeling}.

\begin{remark}
The 2nd-PP1 problem is solved by enumerating all feasible combinations of execution periods. For each generated route $\lambda_r$, the corresponding execution schedule is determined during the route generation process and encoded using binary parameters $d_r^t$, indicating whether the route is executed in period $t$. Consequently, when adding a route to the master problem, the variables $x_r^t$ can be replaced by $d_r^t \lambda_r$, thereby reducing the number of binary decision variables and enhancing the computational efficiency of the master problem.$\hfill\blacksquare$
\end{remark}

\subsubsection{Second Level PP2: Replenishment Plan Generating.}\label{2-pp2}
The replenishment pricing problem, denoted as 2nd-PP2($i$), generates an optimal cyclic inventory policy for each retailer $i$, specifying both the replenishment periods and the associated order-up-to levels:
\begin{align}
{\rm (2nd-PP2}(i){\rm )}\ \min\ &\sum_{(t_1,t_2)\in A^{\rm temporal}}y_{it_1t_2}\bigl(\frac{1}{T}f^{inv}(i, t_1,t_2,s_{it_1})+\psi_{t_2}(s_{it_2}-s_{it_1})+\delta_{it_1}+\nonumber\\
&\psi_{t_2}\sum_{t\in[t_1,t_2-1]}U^i_{t}\bigr)+\theta_i-\pi_i, \\
s.t. \ 
    & y_{it_1t_2}\bigl(s_{it_2}-s_{it_1}+\sum_{t\in[t_1,t_2-1]}L_{t}^i\bigr)\geq 0 \quad \forall (t_1,t_2)\in A^{\rm temporal},\label{c31}\\ 
    &s_{it} \leq M \sum_{(t_1,t_2)\in\eta^+(A^{\rm temporal},t)}y_{it_1t_2}\quad \forall t\in[T],\\
    &s_{it} \in\mathbb{N} \quad \forall t\in[T],\\
    & y_{it_1t_2}\in \mathscr{Y}\quad\forall (t_1,t_2)\in A^{\rm temporal}.
\end{align}

To the best of our knowledge, most existing research in inventory management focuses on finite-horizon policies. Recently, \citet{chen2023robust} introduced the notion of cyclic policies in inventory systems; however, their framework remains confined to a finite planning horizon. In this study, we provide a cyclic inventory policy that repeats indefinitely for each retailer $i$, thus extending the literature on inventory policy design to the infinite-horizon setting. The algorithm for solving this problem is introduced in the subsequent subsection.

To accelerate the nested solution framework, we incorporate an early stopping criterion in the second-layer branch-and-price procedure. Specifically, the second-level branch-and-bound search terminates once a pattern with negative reduced cost is found after exploring at least $m$ nodes. This strategy improves computational efficiency in the first-level pricing step while preserving global optimality, as full enumeration is eventually conducted.

\subsection{Inventory Replenishment Strategy}
In this section, we present the solution algorithm for subproblem 2nd-PP2($i$). The main challenge in solving this formulation arises from Constraints \eqref{c31}, which prevent overshooting and impose limits on the order-up-to levels. However, overshooting occurs only with a small probability, and many of these constraints are redundant. Therefore, we treat them as lazy constraints, adding them only when necessary during the solution process.

\subsubsection{Solution Approach without Lazy Constraints.}
We first present the solution approach for 2nd-PP2($i$) without considering Constraints \eqref{c31}. In this case, the problem can be reformulated as
\begin{align}
{\rm (R-2nd-PP2}(i){\rm )}\ \min\ &\sum_{(t_1,t_2)\in A^{\rm temporal}}c_{t_1t_2}^{i*}y_{it_1t_2}+\theta_i-\pi_i, \nonumber
\end{align}
where $y_{it_1t_2}\in \mathscr{Y}$ and $c_{t_1t_2}^{i*}$ denotes the cost of replenishing retailer $i$ over the interval $(t_1,t_2)$, defined as
\begin{align}
    c_{t_1t_2}^{i*}=\min_{s_{it_1}}c_{t_1t_2}^{i}=\min_{s_{it_1}}\frac{1}{T}f^{\rm inv}(i, t_1, t_2, s_{it_1})+(\psi_{t_1}-\psi_{t_2})s_{it_1}+\delta_{it_1}+\psi_{t_2}\sum_{t\in [t_1,t_2-1]}U^i_t.\nonumber
\end{align}
In the following, we describe step by step how to compute the costs $c_{t_1t_2}^{i*}$ and how to solve the reformulated problem R-2nd-PP2($i$).

\textbf{Calculation of $c_{t_1t_2}^{i*}$.} Recall that Proposition \ref{convex_prop} establishes that the function $f^{\rm inv}(i, t_1, t_2, s_{it_1})$ is convex in $s_{it_1}$, implying that $c_{t_1t_2}^i$ is also convex in $s_{it_1}$. In addition, Theorem \ref{wd_theorem} shows that, for a given $s_{it_1}$, $f^{inv}(i, t_1, t_2, s_{it_1})$ equals the optimal value of a linear optimization problem. Therefore, for any fixed $s_{it_1}$, $c_{t_1t_2}^i$ can be computed by solving a corresponding linear program. Consequently, the optimal value $c_{t_1t_2}^{i*}$ can be efficiently obtained using a golden-section search. The detailed procedure is provided in Algorithm~\ref{goldenSection} in Appendix~\ref{others}.

It is worth noting that performing a golden-section search for every retailer and every interval at each update of the dual variables can be computationally expensive. To improve efficiency, we adopt two acceleration strategies.
First, all evaluated $s_{it_1}$ values and their corresponding inventory costs are stored during the golden-section process, avoiding redundant computations in subsequent iterations. Second, during the pricing stage, small changes in the dual variables typically have little impact on the optimal order-up-to level $s$. Therefore, in each iteration, we first examine the neighboring values of the previously optimal levels under zero dual variables before conducting a full re-optimization. The details of this procedure are summarized in Algorithm~\ref{updateCvw} in Appendix~\ref{others}.

\textbf{Solving R-2nd-PP2($i$).} Once the costs $c_{t_1t_2}^{i*}$ are updated, solving R-2nd-PP2($i$) reduces to a problem that is structurally similar to the Elementary Shortest Path Problem with Resource Constraints (ESPPRC). The key distinction is that there is no explicit starting or ending period, since the replenishment plan follows a cyclic pattern with no natural beginning or end. To address this, we define an artificial start and end period corresponding to the largest index in the cycle. With this definition, we solve the problem by enumerating all candidate cycle-ending periods $t^e \in [T]$. For each $t^e$, a labeling algorithm (described in Online Appendix~\ref{oa-labeling}) is applied to determine the minimum-cost replenishment plan that starts and ends at $t^e$ and includes only periods preceding it. Among all candidate cycles, the plan with the lowest total cost is selected as the final solution.

\subsubsection{Lazy Constraints.}
We now present the approach for handling lazy constraints. We first describe how to incorporate these constraints into the solution process introduced earlier, and then detail the procedure for solving 2nd-PP2($i$) with the incorporated lazy constraints.

\textbf{Lazy constraints incorporation.} 
It is straightforward to observe that the lazy constraint $y_{it_1t_2}(s_{it_2}-s_{it_1}+\sum_{t\in[t_1,t_2-1]}L_{t}^i)\geq 0$ for interval $(t_1,t_2)$ can be satisfied in two ways. First, by not performing replenishment during the interval, i.e., $y_{it_1t_2}=0$. Second, by enforcing replenishment in this interval, i.e., $y_{it_1t_2}=1$, in which case overshooting must be prevented by ensuring $s_{it_2}-s_{it_1}+\sum_{t\in[t_1,t_2-1]}L_{t}^i$. This indicates that whenever a new lazy constraint is triggered, it can be enforced in two alternative ways: either prohibiting replenishment in the corresponding interval or enforcing replenishment while imposing the associated restriction on the order-up-to levels. This branching logic motivates the incorporation of lazy constraints into the solution process via a branch-and-bound approach, where the state of each node includes the set of forbidden intervals $A_N$, the set of enforced intervals $A_Y$, and the restrictions on order-up-to levels associated with intervals in $A_Y$.

Specifically, the tree is initialized with a single root node, corresponding to the problem R-2nd-PP2($i$), where no lazy constraints are considered. At the beginning, the root node is selected for solving. During the branch-and-bound process, at each iteration, the node with the worst bound is selected. For the selected node, we solve 2nd-PP2($i$) without explicitly enforcing the triggered lazy constraints, but subject to the sets $A_Y$ and $A_N$. After solving, we check whether the solution violates Constraint~\eqref{c31}. If not, the solution is feasible, and the global upper bound can be updated. Otherwise, the lower bound is updated, and branching is performed based on the set of infeasible intervals:
\begin{align}
    A'=\bigl\{(t_1,t_2)\big|s_{it_2}-s_{it_1}+\sum_{t\in[t_1,t_2-1]}L_{t}^i<0\bigr\}.\nonumber
\end{align}
Branching is done by enumerating all possible combinations of enforcing or forbidding the intervals in $A'$, resulting in $2^{|A'|}$ child nodes. Each child node corresponds to updated sets $A'_{Y}$ and $A'_{N}$ such that $A'_{Y} \cup A'_{N} = A_{Y} \cup A_{N} \cup A'$. This process is repeated recursively until all leaf nodes yield solutions that satisfy Constraint~\eqref{c31}, thereby ensuring global feasibility. An example of this branching procedure is illustrated in Figure \ref{fig: branching}.

\begin{figure}[!htb]
\centering
  \includegraphics[width=0.5\linewidth]{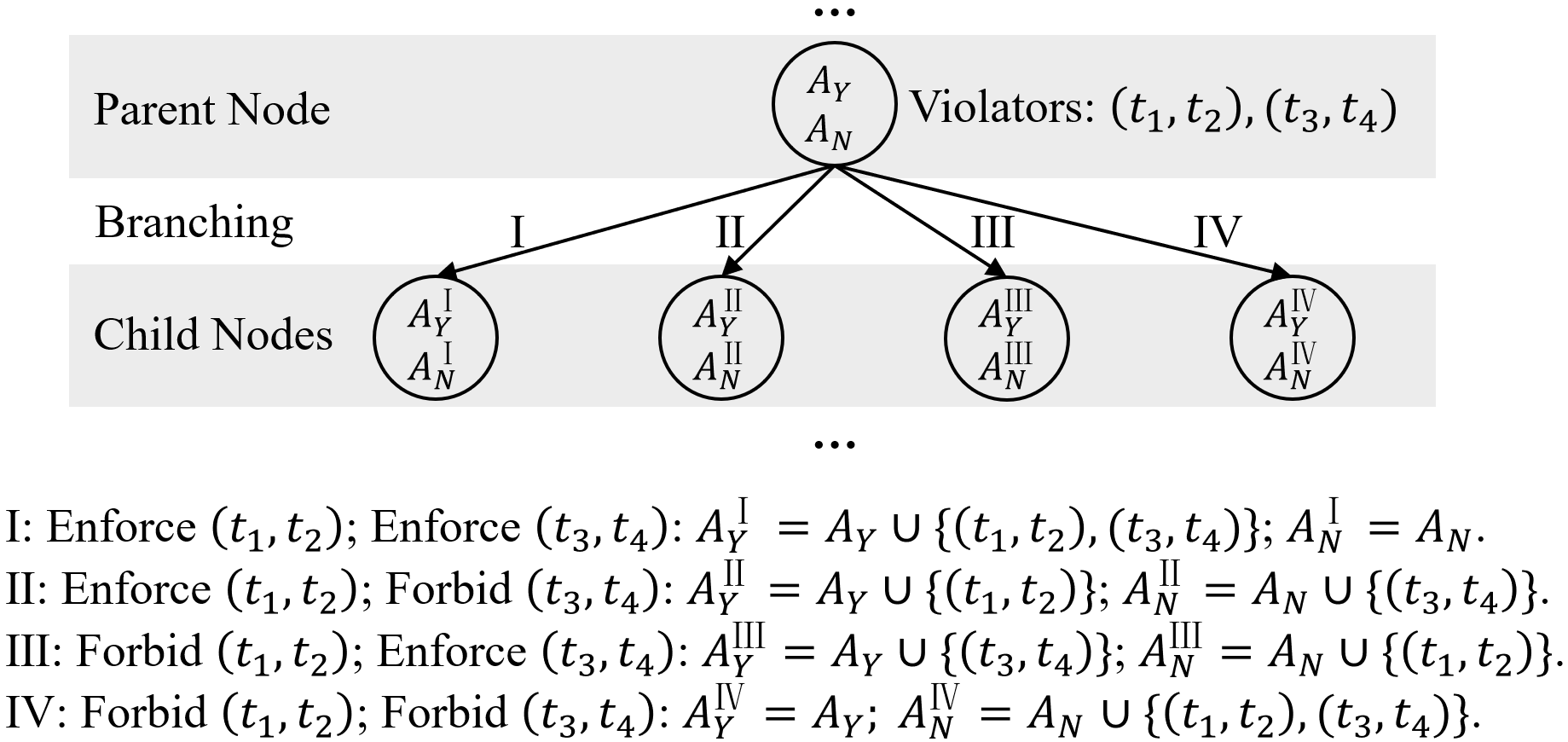}
  \caption{Illustration of a node with two infeasible intervals, leading to four child nodes, each with updated sets according to the branching conditions.}
  \label{fig: branching}
\end{figure}


\textbf{Solving 2nd-PP2($i$) with incorporated lazy constraints.}
We now describe how to solve 2nd-PP2($i$) under given states $A_{Y}$ and $A_{N}$. Eariler, we introduced how to solve R-2nd-PP2($i$); with the additional restrictions imposed by the sets $A_{Y}$ and $A_{N}$, the solution process remains similar, but the cost $c_{t_1t_2}^{i*}$ of certain intervals must be adjusted. 

For any interval $(t_1,t_2) \in A_{N}$, replenishment is prohibited. To enforce this, we set $c_{t_1t_2}^{i*} = +\infty$, effectively excluding the interval from consideration.

For any interval $(t_1,t_2) \in A_{Y}$, replenishment must be performed, and the associated order-up-to level constraints must be satisfied. These restrictions influence not only the interval itself but also other related intervals. Consequently, the cost $c_{t_1t_2}^{i*}$ must be recalculated for all intervals $(t_1',t_2') \in A^{\rm temporal}$. The intervals are classified into three categories, with each category requiring a distinct adjustment procedure:

    \paragraph{Category 1: Extensions of required intervals.} These are intervals $(t_1^{\prime}, t_2^{\prime})$ such that $t_1^{\prime} = t_1$ and $(t_1^{\prime}, t_2^{\prime})$ strictly contains $(t_1, t_2)$ in cyclic order. Although Constraints~\eqref{c31} impose a relationship between $s_{it_1}$ and $s_{it_2}$, $s_{it_2} - s_{it_1} + \sum_{t \in [t_1, t_2 - 1]} L_t^i \geq 0$, the actual values of $s_{it_1}$ and $s_{it_2}$ are influenced by the timing of subsequent replenishment, i.e., the endpoint $t_2^{\prime}$ of the next interval $(t_2, t_2^{\prime})$. Consequently, it is necessary to evaluate the intervals $(t_1, t_2)$ and $(t_2, t_2^{\prime})$ jointly. A similar dependency arises when the required intervals form a replenishment chain. In either case, whether dealing with a partial chain or a complete cycle (e.g., $(t_1, t_2), (t_2, t_3), \ldots, (t_k, t_1)$), the optimal order-up-to levels must be computed simultaneously for all involved periods. In the cyclic case, an additional consistency condition ensures that the starting and ending order-up-to levels are aligned.

    To solve this problem efficiently, we design a dynamic programming algorithm. Each stage corresponds to a time period, and each state represents a feasible order-up-to level. State transitions are governed by Constraints~\eqref{c31}, and suffix minima are employed to accelerate the forward recursion. After completing the recursion, we obtain the optimal sequence of order-up-to levels via backward tracing.

    \paragraph{Category 2: Conflicting intervals.} These are intervals that conflict with the required structural properties, such as $(t_1^{\prime}, t_2^{\prime})$ partially overlapping with $(t_1, t_2)$, and are therefore considered structurally incompatible. To conservatively exclude such intervals from consideration, we assign them an infinite cost, i.e., set $c^{i*}_{t_1't_2'} = +\infty$.

    \paragraph{Category 3: Independent intervals.} Intervals not structurally related to $A^{Y}$ are treated normally and evaluated using the standard baseline method.


\subsection{Branching Strategies}
Efficient branching strategies are critical to accelerating the convergence of the branch-and-bound process. We develop a two-layer branching scheme that aligns with the nested structure of our branch-and-price framework.

In the first-layer branch-and-price framework, let $q_{\omega}^*$ denote the LP relaxation value of pattern $q_{\omega}$ in the restricted master problem (1st-RMP). Under the consistent policy, the routing plan associated with a pattern remains identical across all periods. When the solution is fractional, we identify the pattern for which $q_{\omega}^*$ is closest to 0.6 and randomly select an arc from its corresponding route. Two child nodes are then created: one enforces the use of the selected arc, while the other forbids it.

In the second-layer branch-and-price framework, we branch on replenishment decisions through three sequential refinement steps:

    \paragraph{Retailer-level branching.} Let $\gamma_{l(i)}^*$ denote the LP relaxation value of replenishment plan $\gamma_{l(i)}$ for retailer $i$ in the 2nd-RMP. If $\sum_{l(i) \in L(i)} \gamma_{l(i)}^*$ is fractional for any retailer, we select the one closest to 0.6. Two child nodes are created: one enforces serving the retailer (i.e., $\sum_{l(i)} \gamma_{l(i)} = 1$), and the other forbids it (i.e., $\sum_{l(i)} \gamma_{l(i)} = 0$).
    \paragraph{Interval-level branching.} If all retailer-level decisions are integral, we proceed to branch on replenishment intervals. We identify a fractional replenishment plan and select one visiting interval from it. Two child nodes are created: one enforces the execution of this interval, while the other prohibits it. The corresponding branching decisions are incorporated into the sets $A^{\rm temporal}_Y$ and $A^{\rm temporal}_N$, respectively.
    \paragraph{Order-up-to-level branching.} If all retailer and interval decisions are integral, but a retailer still has multiple replenishment plans with the same visiting periods but different order-up-to levels, we branch on the order-up-to level $s_{it}$. We identify such a conflict, select one of the replenishment plans, and choose a retailer whose $s_{it}$ differs in other plans. Two child nodes are generated: one enforces the retailer to adopt the specified $s_{it}$, and the other forbids it.

\subsection{Initial Solution Construction}
We construct an initial feasible solution in which each retailer is visited in every period of the planning horizon. The order-up-to levels for each retailer are uniformly set across all periods to the upper bound of their respective demand support, as specified in the ambiguity set. This choice guarantees inventory feasibility while mitigating the risk of stockouts and excessive overstocking.

To generate consistent routing plans, i.e., identical vehicle routes across periods, we solve a capacitated vehicle routing problem (CVRP), where each retailer's demand is set to their maximum possible value over the planning horizon. The resulting routes define distinct retailer clusters, each forming a candidate pattern for the subsequent optimization.

This initialization ensures feasibility of both inventory and routing decisions and provides a high-coverage baseline that can be refined in later stages of the solution process.

\section{Computational Experiments}\label{s5}

Our solution approach is implemented in Python 3.12.2. The branch-and-bound framework is developed using the PyBnB library, and the pricing subproblems are implemented in C++ and integrated as a callable Python module. All linear and mixed-integer programming problems are solved using the Gurobi Optimizer (version 12.0.1). Computational experiments are conducted on a server running Ubuntu 22.04.4 LTS, equipped with an Intel(R) Core(TM) i9-14900K processor (24 cores, 32 threads, base frequency 1.60GHz, turbo up to 6.00GHz) and 125 GB RAM. All instances are solved using a single computational thread to ensure consistency.

\subsection{Synthetic Datasets.}
We construct synthetic datasets to enable controlled performance comparisons and sensitivity analyses. To generate the data, we first create ground-truth moment information for each retailer and each cycle period. Specifically, the demand mean $\mu_i^t$ is sampled from the uniform distribution $U(6, 14)$, and the mean absolute deviation $\sigma_i^t$ is sampled from $U(1, 3)$. We consider five cycle configurations: one stationary case and four cyclic cases with $T \in \{4, 5, 6, 7\}$. In the cyclic settings, both the planning horizon and the underlying demand cycle span $T$ periods. In contrast, under the stationary setting, the demand is invariant across all periods (i.e., the demand cycle length is 1), while $T$ denotes the length of the planning cycle for decision-making purposes. Hence, in the stationary case, the value of $T$ may refer to either the demand cycle (always 1) or the planning cycle (varying with instance settings). For clarity and conciseness, we do not explicitly distinguish these interpretations when the context is self-explanatory. For each configuration, $(\mu_i^t, \sigma_i^t)$ values are independently generated for every retailer $i$ and each period $t$ within the planning cycle.

For each retailer $i$ and period $t \in [T]$, 100,000 sample demands are generated. Following \citet{cui2023inventory}, each demand is drawn with equal probability from either a truncated normal distribution $N(\mu_i^t, \sigma_i^t)$, truncated to $\bigl[\max\{0, \mu_i^t - 3\sigma_i^t\}, \mu_i^t + 3\sigma_i^t\bigr]$, or a uniform distribution $U(\mu_i^t - 2\sigma_i^t, \mu_i^t + 2\sigma_i^t)$. This results in a multimodal mixture distribution that introduces nontrivial uncertainty. Based on the generated samples, we estimate the parameters of the ambiguity set for each retailer in each period. For each cycle type, multiple instances are created by varying the number of retailers and selecting different subsets.

To assess out-of-sample performance, we simulate demand over 100 test periods using the same ground-truth $(\mu_i^t, \sigma_i^t)$ parameters. These simulations respect the cyclic structure of each instance. Retailer locations are randomly sampled with latitude in $[50.7, 53.4]$ and longitude in $[3.5, 7.1]$.

\subsection{Numerical Experiments on Synthetic Data}
This subsection evaluates the performance of our proposed approach under synthetic data across various settings. We compare the three service policies in terms of solution quality and computational time, under different cycle configurations and retailer sizes. We then conduct a series of sensitivity analyses to examine the effects of the chance-constraint parameters ($\varepsilon_1$ and $\varepsilon_2$), unit cost variations, and the choice of the order-up-to level $s_{it}$ on the distribution of worst-case expected inventory costs. Finally, we compare the proposed decomposition method with a direct solution approach using Gurobi on the compact formulation (CP).

Since the fixed-interval policy is only meaningful under stationary demand, we adapt it for non-stationary instances ($T > 1$) by treating demand as stationary within each cycle. Specifically, we compute the mean and mean absolute deviation over the entire cycle to define the parameters used by the fixed-interval policy.

Unless otherwise specified, the default parameter settings are as follows: unit holding cost $h = 1$, unit backorder cost $b = 4$, unit transportation cost $\rho = 0.25$, fixed vehicle dispatch cost $p = 10$, and chance-constraint parameters $\varepsilon_1 = 0.3$ and $\varepsilon_2 = 0.1$. Each vehicle has a capacity of 70 units, and no constraint is imposed on the maximum driving distance.

\subsubsection{Performance Evaluation.}
In the performance evaluation, we impose two time limits to manage computational cost. First, the total runtime for solving each instance is capped at one hour. Second, each call to the second-level branch-and-price solver is limited to five minutes. Recall that our framework employs an early-stopping criterion: the second-level branch-and-bound search terminates as soon as a column with negative reduced cost is found. If no such column is identified within five minutes, we assume no improving column exists and terminate the pricing process at that first-level branch-and-bound node. As a result, the corresponding node may remain unsolved to proven optimality. Nevertheless, our approach consistently yields high-quality solutions. In practice, the second-level solver often identifies the optimal integer solution quickly, while certifying its optimality may take significantly longer. Thus, even when formal optimality is not attained within the time limit, the resulting solution is frequently optimal or near-optimal.

\textbf{Solution performance.}
Table~\ref{tab:solution} summarizes the performance of different service policies, reporting averages over three instances generated under the same configuration but with different retailer samples. The table presents results for the cyclic case with $T=7$; complete results for $T \in \{4,5,6,7\}$ are available in Table~\ref{tab:oa-solution} in Online Appendix~E. We defer analysis of the stationary case ($T=1$) to a later subsection. In the table, ``T.O.\%" denotes the percentage of first-level branch-and-bound nodes that were not solved to proven optimality due to the five-minute time limit. ``Avg I." denotes the average replenishment interval, and ``S.L." indicates the average service level. ``O.\%" represents the overshooting rate, and ``Avg. O." reports the average overshooting amount, conditional on overshooting occurring. Similarly, ``E.T.\%" denotes the rate of emergency transportation, and ``Avg. E.T." gives the average emergency transportation quantity when such events occur. Note that the statistics for ``S.L.", ``Vehicle Util", ``O.\%", ``Avg O.", ``E.T.\%", and ``Avg E.T." are computed based on the simulation outcomes.

\begin{table}[!htb]
\centering
\caption{Solution performance of different service policies.}
\label{tab:solution}
\scriptsize
\begin{tabular}{c c r r r r rr r r r r r}
\toprule
\#Retailer          & Policy    &Time (s)   &T.O.\%  & Cost & \#Cluster & Avg I. & S.L. & Vehicle Util & O.\% & Avg O. & E.T.\% & Avg E.T. \\
\midrule
\multirow{3}{*}{5}  & Fixed-Interval&4 & 0\%  & 217  & 3.0       & 2.7              & 92\%          & 56\%                & 0.0\%                & 0.0               & 0.0\%                               & 0.0                              \\
                    & Consistent &909 & 100\%    & 188  & 2.0       & 1.8              & 94\%          & 64\%                & 0.0\%                & 0.0               & 0.0\%                               & 0.0                              \\
                    & Flexible  &1074  & 100\%    & 175  & 1.0       & 2.1              & 94\%          & 70\%                & 0.0\%                & 0.0               & 0.0\%                               & 0.0                              \\
\midrule
\multirow{3}{*}{7}  & Fixed-Interval&5 & 0\% & 273  & 2.3       & 1.6              & 89\%          & 59\%                & 0.0\%                & 0.0               & 0.0\%                               & 0.0                              \\
                    & Consistent &3368  & 88\%   & 254  & 2.0       & 1.4              & 94\%          & 64\%                & 0.6\%                & 1.1               & 0.0\%                               & 0.0                              \\
                    & Flexible &731  & 100\%     & 238  & 2.0       & 1.8              & 94\%          & 60\%                & 0.0\%                & 0.0               & 0.0\%                               & 0.0                              \\
\midrule
\multirow{3}{*}{10} & Fixed-Interval&5 & 0\% & 369  & 4.0       & 1.8              & 90\%          & 55\%                & 0.0\%                & 0.0               & 0.0\%                               & 0.0                              \\
                    & Consistent &3623  & 0\%   & 476  & 4.0       & 1.0              & 100\%         & 35\%                & 1.8\%                & 1.9               & 0.0\%                               & 0.0                              \\
                    & Flexible  &3622  & 77\%    & 339  & 2.7       & 1.7              & 95\%          & 55\%                & 0.2\%                & 1.2               & 0.0\%                               & 0.0                              \\
\midrule
\multirow{3}{*}{12} & Fixed-Interval&7 & 0\% & 419  & 4.0       & 1.5              & 89\%          & 56\%                & 0.0\%                & 0.0               & 0.0\%                               & 0.0                              \\
                    & Consistent  &3626 & 0\%   & 515  & 4.0       & 1.0              & 100\%         & 42\%                & 1.5\%                & 1.9               & 0.0\%                               & 0.0                              \\
                    & Flexible   &3443  & 75\%   & 371  & 3.0       & 1.7              & 94\%          & 62\%                & 0.0\%                & 0.3               & 0.0\%                               & 0.0                              \\
\midrule
\multirow{3}{*}{15} & Fixed-Interval&7 & 0\% & 491  & 4.0       & 1.3              & 89\%          & 61\%                & 0.0\%                & 0.0               & 0.0\%                               & 0.0                              \\
                    & Consistent &3628 & 0\%    & 599  & 5.0       & 1.0              & 100\%         & 42\%                & 1.5\%                & 1.9               & 0.0\%                               & 0.0                              \\
                    & Flexible  &3636  & 0\%    & 475  & 3.7       & 1.5              & 96\%          & 60\%                & 1.1\%                & 2.6               & 0.0\%                               & 0.0    \\
\bottomrule
\end{tabular}
\end{table}

Several key observations emerge from the computational results. The flexible policy consistently achieves the lowest total cost and requires fewer clusters than the other two policies. This advantage stems from its ability to jointly optimize routing and inventory decisions over flexible intervals and vehicle routes, enabling more efficient vehicle utilization and more responsive replenishment decisions. 

When the number of retailers is small, the consistent policy outperforms the fixed-interval policy in both total cost and cluster usage. However, as the problem scale increases, the limited computational budget (one hour) prevents the consistent policy from reaching optimality, resulting in higher costs and increased cluster counts relative to the fixed-interval policy. Although the flexible policy is also affected by time limits, its structural adaptability helps mitigate the impact of incomplete optimization. Nonetheless, we anticipate that for sufficiently large-scale instances, the fixed-interval policy may become more favorable due to its superior computational traceability.

In terms of service level, the fixed-interval policy performs the worst among the three. This can be attributed, on the one hand, to its reliance on aggregated demand information within the cycle, which overlooks intra-cycle variability; and on the other hand, to its rigid visiting schedule, which limits responsiveness to actual demand realizations and thereby reduces overall service quality.

Importantly, the chance constraints prove effective in managing risk. Across all tested instances, both the overshooting rate and emergency transportation rate remain low. Notably, the emergency transportation rate is consistently zero, indicating high reliability of the planned schedules under uncertainty.

\begin{figure*}[!htb]
    \centering
    \subfigure[Fixed-Interval Policy]{\includegraphics[width=0.8\linewidth]{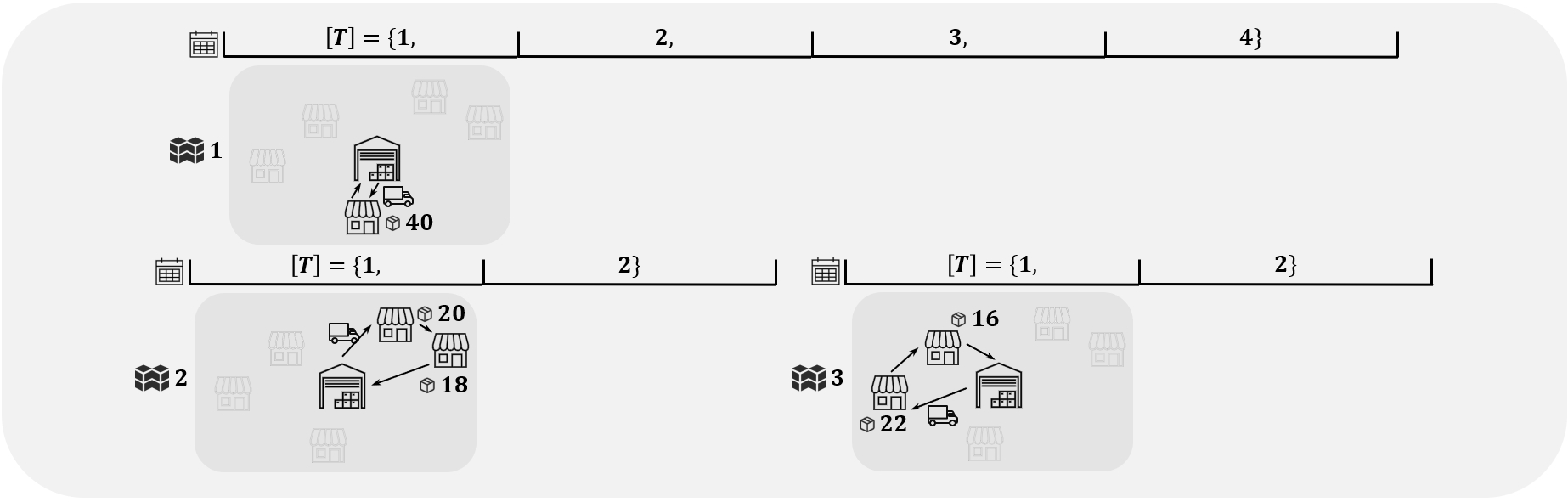}}
    \subfigure[Consistent Policy]{\includegraphics[width=0.8\linewidth]{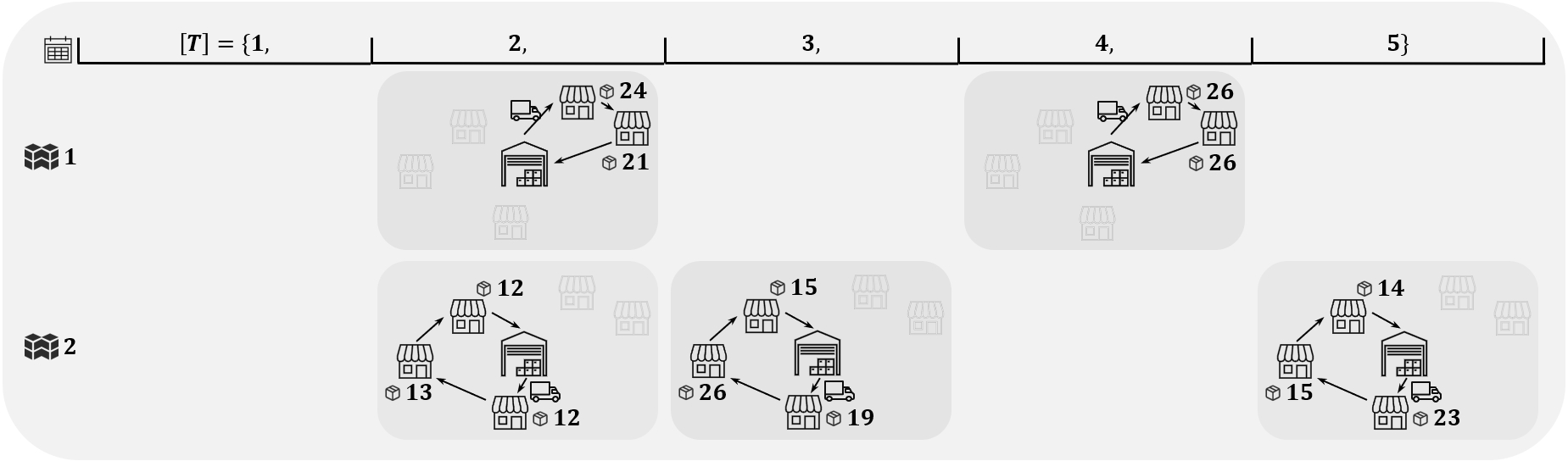}}
    \subfigure[Flexible Policy]{\includegraphics[width=0.8\linewidth]{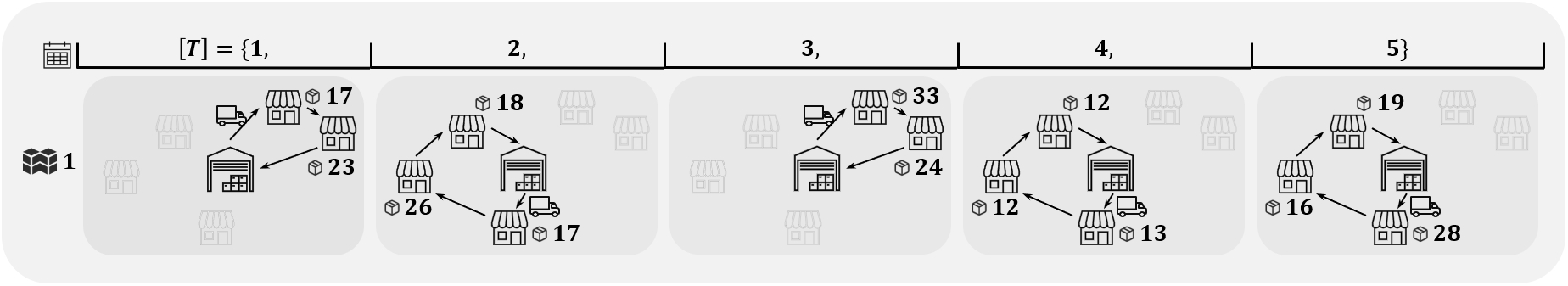}}
    \caption{Solution visualization of the three service policies.}
    \label{fig-policy}
\end{figure*}

To further illustrate the operational differences among the three policies, Figure~\ref{fig-policy} presents a representative solution from a randomly selected instance involving five retailers over a five-period cycle. Each cluster is shown separately, and the number next to each retailer denotes its corresponding order-up-to level. Notably, no overshooting or emergency transportation occurs in this instance. The visualization underscores the superior efficiency of the flexible policy in jointly coordinating inventory replenishment and vehicle routing decisions.

\textbf{Computation performance.}
We also report the computational performance of the proposed approach. Table~\ref{tab:compu} presents detailed solving statistics for the three service policies. In the table, ``PP T." denotes the total time spent solving the pricing problems. ``Route" indicates the total number of routes generated, and ``Route T." reports the cumulative time spent solving the routing subproblem. ``Replnsh" refers to the total number of replenishment plans generated, and ``Replnsh T." corresponds to the time required to solve the replenishment subproblem. Additional results for other cycle configurations are provided in Table~\ref{tab:oa-compu} in Online Appendix~E.

From Table~\ref{tab:compu}, we observe that when the number of retailers is small, the one-hour time limit is rarely reached, whereas the five-minute time limit for the second-level branch-and-price is frequently triggered. As previously discussed, this secondary limit generally does not significantly compromise solution quality, since near-optimal columns are often found quickly even if optimality cannot be certified within the time bound. As the number of retailers increases, total solving time rises correspondingly, and eventually reaches the one-hour cap. When this occurs, further exploration of the first-level branch-and-bound tree is curtailed, potentially affecting solution quality for larger instances.

We also note a sharp increase in the time required to solve the second-level routing problem as the number of retailers grows, under the consistent and flexible policies. This trend arises because the routing subproblem corresponds to solving an ESPPRC, which is NP-hard and grows exponentially in complexity with the number of retailers. In contrast, the replenishment subproblem, which determines inventory control policies, scales linearly with the number of retailers, and its solution time increases more gradually.

\begin{table}[!htb]
\centering
\caption{Computation performance of different service policies.}
\label{tab:compu}
\scriptsize
\begin{tabular}{c c r r r r r r r r r r r}
\toprule
\multirow{2}{*}{\#Retailer} & \multirow{2}{*}{Policy} & \multirow{2}{*}{Cost} & \multirow{2}{*}{Time (s)} & \multirow{2}{*}{T.O.\%} & \multicolumn{3}{c}{First-Level} & \multicolumn{5}{c}{Second-Level}                         \\
\cmidrule(r){6-8}
\cmidrule(r){9-13}
                            &                         &                       &                           &                         & \#Node   & \#Column  & PP T.  & \#Node & \#Route & Route T. & \#Replnsh & Replnsh T. \\
                            \midrule
\multirow{3}{*}{5}          & Fixed-Interval          & 217                   & 4                         & 0\%                     & 7        & 41        & 0        & -      & -       & -          & -         & -            \\
                            & Consistent              & 188                   & 909                       & 100\%                   & 1        & 19        & 895      & 2168   & 18154   & 250        & 11908     & 429          \\
                            & Flexible                & 175                   & 1074                      & 100\%                   & 1        & 21        & 1061     & 4680   & 4498    & 14         & 12511     & 670          \\
                            \midrule
\multirow{3}{*}{7}          & Fixed-Interval          & 273                   & 5                         & 0\%                     & 9        & 86        & 0        & -      & -       & -          & -         & -            \\
                            & Consistent              & 254                   & 3368                      & 88\%                    & 5        & 40        & 3351     & 3951   & 73299   & 1728       & 22622     & 802          \\
                            & Flexible                & 238                   & 731                       & 100\%                   & 1        & 18        & 714      & 1019   & 7079    & 11         & 4749      & 250          \\
                            \midrule
\multirow{3}{*}{10}         & Fixed-Interval          & 369                   & 5                         & 0\%                     & 1        & 84        & 0        & -      & -       & -          & -         & -            \\
                            & Consistent              & 476                   & 3623                      & 0\%                     & 1        & 8         & 3600     & 57     & 12959   & 3415       & 809       & 150          \\
                            & Flexible                & 339                   & 3622                      & 77\%                    & 4        & 77        & 3600     & 4383   & 77678   & 155        & 20673     & 808          \\
                            \midrule
\multirow{3}{*}{12}         & Fixed-Interval          & 419                   & 7                         & 0\%                     & 35       & 322       & 0        & -      & -       & -          & -         & -            \\
                            & Consistent              & 515                   & 3626                      & 0\%                     & 1        & 11        & 3600     & 65     & 14370   & 3374       & 1008      & 187          \\
                            & Flexible                & 371                   & 3443                      & 75\%                    & 2        & 60        & 3413     & 694    & 108437  & 427        & 12296     & 525          \\
                            \midrule
\multirow{3}{*}{15}         & Fixed-Interval          & 491                   & 7                         & 0\%                     & 25       & 453       & 1        & -      & -       & -          & -         & -            \\
                            & Consistent              & 599                   & 3628                      & 0\%                     & 1        & 9         & 3600     & 45     & 11992   & 3372       & 810       & 196          \\
                            & Flexible                & 475                   & 3636                      & 0\%                     & 1        & 41        & 3604     & 290    & 90405   & 2162       & 7612      & 369     \\
                            \bottomrule
\end{tabular}
\end{table}

\textbf{Stationary demand performance.}
Table~\ref{tab:sta-solution} reports the solution performance of the three service policies under stationary demand, where the decision-making cycle is set to $T=7$. For results under alternative planning cycle lengths ($T \in \{4,5,6,7\}$), we refer the reader to Table~\ref{tab:oa-sta-solution} in Online Appendix~E.

\begin{table}[!htb]
\centering
\caption{Solution performance of different service policies under stationary demand.}
\label{tab:sta-solution}
\scriptsize
\begin{tabular}{c c r r r r rr r r r r r}
\toprule
\#Retailer          & Policy         & Time (s) & T.O.\% & Cost & \#Cluster & Avg I. & S.L. & Vehicle Util & O.\%  & Avg O. & E.T.\% & Avg E.T. \\
\midrule
\multirow{3}{*}{5}  & Fixed-Interval & 4        & 0\%    & 191  & 2.3       & 2.2          & 82\%          & 67\%         & 0.0\% & 0.0    & 3.2\%  & 3.4      \\
                    & Consistent     & 728      & 100\%  & 197  & 2.0       & 1.8          & 83\%          & 59\%         & 0.0\% & 0.0    & 0.0\%  & 0.0      \\
                    & Flexible       & 1303     & 87\%   & 181  & 1.3       & 1.9          & 92\%          & 61\%         & 0.0\% & 0.0    & 0.0\%  & 0.0      \\
                    \midrule
\multirow{3}{*}{7}  & Fixed-Interval & 5        & 0\%    & 245  & 2.7       & 2.2          & 81\%          & 71\%         & 0.0\% & 0.0    & 1.2\%  & 3.8      \\
                    & Consistent     & 1433     & 100\%  & 244  & 2.0       & 1.7          & 86\%          & 68\%         & 0.3\% & 0.7    & 2.5\%  & 3.1      \\
                    & Flexible       & 1803     & 72\%   & 233  & 2.0       & 1.8          & 94\%          & 59\%         & 0.0\% & 0.0    & 0.0\%  & 0.0      \\
                    \midrule
\multirow{3}{*}{10} & Fixed-Interval & 7        & 0\%    & 335  & 3.7       & 2.0          & 82\%          & 69\%         & 0.0\% & 0.0    & 1.7\%  & 3.2      \\
                    & Consistent     & 3620     & 0\%    & 386  & 3.0       & 1.0          & 98\%          & 46\%         & 0.0\% & 0.0    & 0.0\%  & 0.0      \\
                    & Flexible       & 3623     & 67\%   & 324  & 2.0       & 2.0          & 93\%          & 64\%         & 0.0\% & 0.0    & 0.0\%  & 0.0      \\
                    \midrule
\multirow{3}{*}{12} & Fixed-Interval & 8        & 0\%    & 366  & 3.0       & 1.3          & 80\%          & 67\%         & 0.0\% & 0.0    & 0.0\%  & 0.0      \\
                    & Consistent     & 3623     & 0\%    & 433  & 3.0       & 1.0          & 98\%          & 56\%         & 0.0\% & 0.0    & 0.0\%  & 0.0      \\
                    & Flexible       & 3633     & 39\%   & 390  & 3.0       & 1.7          & 93\%          & 57\%         & 2.4\% & 2.0    & 0.0\%  & 0.0      \\
                    \midrule
\multirow{3}{*}{15} & Fixed-Interval & 9        & 0\%    & 424  & 3.3       & 1.3          & 82\%          & 71\%         & 0.0\% & 0.0    & 1.2\%  & 1.8      \\
                    & Consistent     & 3625     & 0\%    & 499  & 4.0       & 1.0          & 98\%          & 53\%         & 0.0\% & 0.0    & 0.0\%  & 0.0      \\
                    & Flexible       & 3627     & 0\%    & 459  & 3.3       & 1.3          & 96\%          & 61\%         & 0.7\% & 2.0    & 0.0\%  & 0.0     \\
                    \bottomrule
\end{tabular}
\end{table}

From Table~\ref{tab:sta-solution}, we observe that under stationary demand, all policies share the same underlying demand information. Consequently, the performance of the fixed-interval policy improves relative to the non-stationary case. In particular, when the number of retailers is small, the fixed-interval policy achieves comparable performance to the consistent policy, unlike in the non-stationary setting, where it consistently underperforms. 

Interestingly, while in the non-stationary case, a lower total cost is typically associated with a smaller number of clusters, this pattern does not hold under stationary demand. When the fixed-interval policy outperforms in total cost, it does not necessarily use fewer clusters, suggesting that its cost advantage primarily arises from more efficient inventory control rather than superior vehicle utilization. This observation aligns with the finding of \citet{hasturk2024stochastic}, which shows that the base-stock policy is optimal when inter-delivery intervals are fixed. This further supports the notion that the fixed-interval policy can be highly efficient in inventory control under stationary demand.

Despite having the same demand information as the other two policies, the fixed-interval policy still delivers the lowest service level. This underscores the limitations imposed by its rigid delivery schedule, which reduces responsiveness to actual demand realizations. We also observe that the fixed-interval policy incurs emergency transportation more frequently and exhibits higher vehicle utilization. This suggests that, under the same risk parameters, it entails higher operational risk than the other two policies. Practically, this implies that fixed-interval policies may require more conservative risk control, e.g., by selecting a smaller value of $\varepsilon_1$. Other observations remain consistent with those obtained in the non-stationary case. 

Computation results are summarized in Table~\ref{tab:sta-copu}, which exhibits trends consistent with those observed under non-stationary demand. Complete results for varying cycle lengths are provided in Table E4 of Online Appendix E.

\begin{table}[!htb]
\centering
\caption{Computation performance of different service policies under stationary demand.}
\label{tab:sta-copu}
\scriptsize
\begin{tabular}{c c r r r r r r r r r r r}
\toprule
\multirow{2}{*}{\#Retailer} & \multirow{2}{*}{Policy} & \multirow{2}{*}{Cost} & \multirow{2}{*}{Time (s)} & \multirow{2}{*}{T.O.\%} & \multicolumn{3}{c}{First-Level} & \multicolumn{5}{c}{Second-Level}                         \\
\cmidrule(r){6-8}
\cmidrule(r){9-13}
                            &                         &                       &                           &                         & \#Node   & \#Column  & PP T.  & \#Node & \#Route & Route T. & \#Replnsh & Replnsh T. \\
                            \midrule
\multirow{3}{*}{5}  & Fixed-Interval & 191  & 4        & 0\%    & 6           & 44       & 0       & 0            & 0       & 0          & 0         & 0            \\
                    & Consistent     & 197  & 728      & 100\%  & 1           & 17       & 715     & 821          & 11124   & 132        & 3988      & 477          \\
                    & Flexible       & 181  & 1303     & 87\%   & 2           & 20       & 1290    & 3407         & 3164    & 12         & 8729      & 975          \\
                            \midrule
\multirow{3}{*}{7}  & Fixed-Interval & 245  & 5        & 0\%    & 46          & 269      & 0       & 0            & 0       & 0          & 0         & 0            \\
                    & Consistent     & 244  & 1433     & 100\%  & 1           & 24       & 1420    & 1315         & 33011   & 591        & 7471      & 594          \\
                    & Flexible       & 233  & 1803     & 72\%   & 5           & 40       & 1790    & 1968         & 13799   & 21         & 8938      & 970          \\
                            \midrule
\multirow{3}{*}{10} & Fixed-Interval & 335  & 7        & 0\%    & 111         & 546      & 1       & 0            & 0       & 0          & 0         & 0            \\
                    & Consistent     & 386  & 3620     & 0\%    & 1           & 19       & 3600    & 117          & 29912   & 3334       & 1785      & 183          \\
                    & Flexible       & 324  & 3623     & 67\%   & 3           & 67       & 3603    & 3364         & 70651   & 122        & 15213     & 1002         \\
                            \midrule
\multirow{3}{*}{12} & Fixed-Interval & 366  & 8        & 0\%    & 81          & 833      & 2       & 0            & 0       & 0          & 0         & 0            \\
                    & Consistent     & 433  & 3623     & 0\%    & 1           & 16       & 3600    & 97           & 20531   & 3372       & 1221      & 184          \\
                    & Flexible       & 390  & 3633     & 39\%   & 2           & 57       & 3609    & 549          & 99366   & 335        & 9558      & 653          \\
                            \midrule
\multirow{3}{*}{15} & Fixed-Interval & 424  & 9        & 0\%    & 85          & 1002     & 3       & 0            & 0       & 0          & 0         & 0            \\
                    & Consistent     & 499  & 3625     & 0\%    & 1           & 13       & 3600    & 71           & 15582   & 3393       & 1003      & 177          \\
                    & Flexible       & 459  & 3627     & 0\%    & 1           & 40       & 3603    & 284          & 92007   & 2203       & 6682      & 378      \\
                    \bottomrule
\end{tabular}
\end{table}

\subsubsection{Sensitivity Analysis}\label{5.2.2}
This subsection presents a comprehensive sensitivity analysis to evaluate the influence of key parameters in our framework. Specifically, we examine the effects of (i) the tolerance levels of chance constraints and (ii) variations in unit cost parameters. The analysis is based on the average results of three synthetic instances, each involving five retailers and a planning cycle of five periods, solved under the flexible service policy. In addition, we provide a visual and explanatory demonstration of how different values of the order-up-to level $s_{it}$ affect the structure of the worst-case distribution of inventory costs.

\textbf{Chance constraints tolerance.} The impact of the chance constraint parameters $\varepsilon_1$ and $\varepsilon_2$ is illustrated in Figure~\ref{fig:epsilon}, which displays normalized values (scaled to [0,1]) of total cost, solving time, vehicle utilization, and service level (applicable to $\varepsilon_2$) as the tolerance levels increase. The cumulative quantities of emergency transportation and inventory overshooting are also shown on the secondary axis.

As $\varepsilon_1$ increases (relaxing the tolerance on vehicle capacity violations), both total cost and solving time decrease, indicating improved routing efficiency and faster computation. Vehicle utilization also rises, as less conservative solutions are permitted: with tighter constraints, vehicles must reserve additional capacity to hedge against uncertainty, leading to underutilization. Loosening this constraint enables more fully loaded routes. However, this flexibility comes at the cost of increased emergency transportation, reflecting the higher operational risk under more permissive $\varepsilon_1$ values.

For $\varepsilon_2$, which controls the tolerance on inventory overshooting, we observe a clear increase in cumulative overshooting as the parameter grows. In contrast, other performance indicators, total cost, solving time, vehicle utilization, and service level, remain relatively unaffected. This suggests that $\varepsilon_2$ primarily influences local inventory behaviors without substantially impacting broader system-level efficiency.

\begin{figure*}[!htb]
    \centering
    \subfigure[$\varepsilon_1$]{\includegraphics[width=0.31\linewidth]{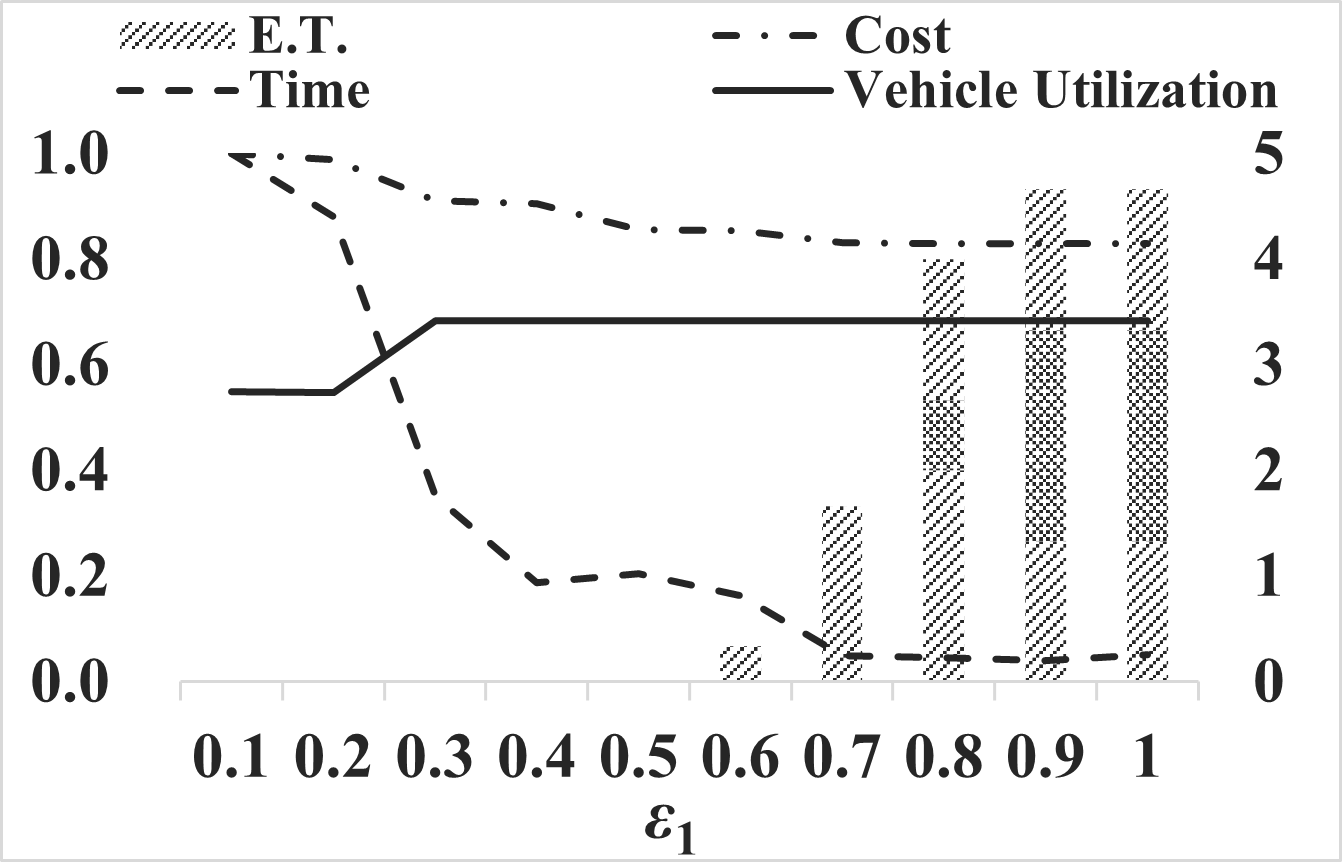}}
    \subfigure[$\varepsilon_2$]{\includegraphics[width=0.31\linewidth]{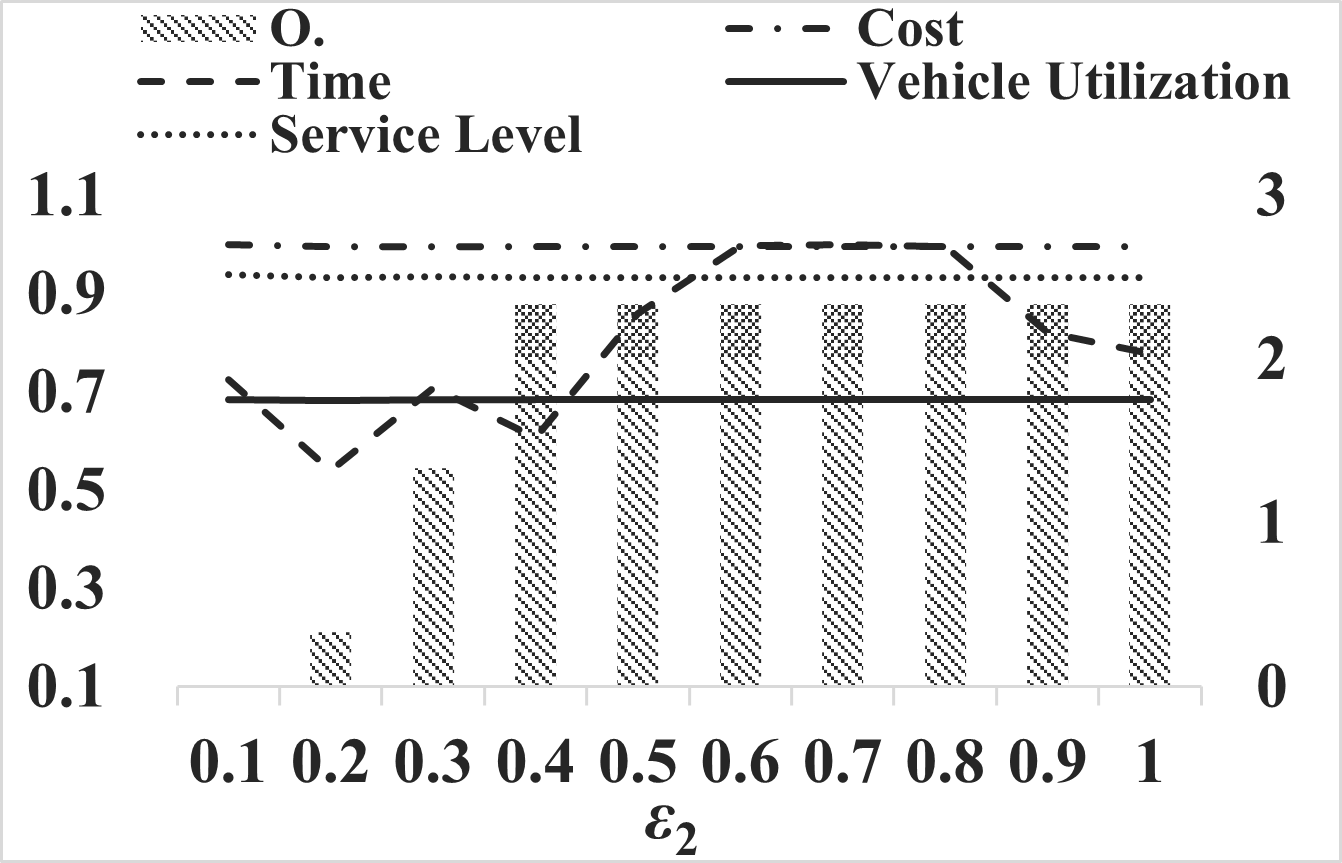}}
    \caption{Influence of chance constraint parameters.}
    \label{fig:epsilon}
\end{figure*}

\textbf{Unit Cost Parameters.}
We further analyze the impact of unit cost parameters on solution performance, as illustrated in Figure~\ref{fig:unitcost}. The cumulative costs are shown on the left axis, while the average replenishment interval and service level are normalized and displayed on the right axis.

We first observe that increasing the unit holding cost leads to a reduction in the average replenishment interval. This is intuitive: higher holding costs discourage excessive inventory, prompting more frequent replenishments to minimize inventory exposure. However, the impact on service level is more complex. The service level reflects trade-offs among transportation cost, holding cost, and backorder cost. On one hand, a higher holding cost makes backorders relatively less expensive, potentially lowering the service level. On the other hand, more frequent replenishment driven by increased holding cost can improve service quality. Because routing decisions are discrete by nature, marginal changes in holding cost may have no immediate effect until a threshold is crossed. At such points, a change in the transportation plan may occur, allowing for more efficient vehicle utilization and shipment consolidation, which can temporarily boost service level. This non-linear behavior is visible in the figure: each time the holding cost increase triggers a routing adjustment (indicated by a change in transportation cost), the service level exhibits a noticeable jump. Conversely, when the routing plan remains fixed, continued increases in holding cost may gradually erode the service level.

The effect of the unit backorder cost is more direct. As the backorder penalty increases, the system prioritizes avoiding stockouts, leading to shorter replenishment intervals and improved service levels. Moreover, as the backorder cost becomes increasingly dominant relative to the holding cost, this effect is amplified, resulting in a monotonic rise in service quality.

The influence of transportation cost, however, appears counterintuitive. As expected, higher transportation costs lead to longer replenishment intervals, as the system seeks to reduce delivery frequency. Yet, under the base cost configuration, we observe an increase in service level. This can be attributed to the relatively high backorder cost: when transportation becomes expensive, the system compensates by maximizing load efficiency and delivering larger quantities per trip, thereby reducing the risk of stockouts and improving service reliability. To validate this interpretation, we adjust the cost structure by lowering the backorder cost to match the holding cost. The resulting trend, shown in Figure~\ref{fig:unitcost}(d), reveals that service level indeed declines as transportation cost rises, confirming that the previously observed improvement was due to the relatively high penalty on backorders.

\begin{figure*}[!htb]
    \centering
    \subfigure[Unit holding cost]{\includegraphics[width=0.24\linewidth]{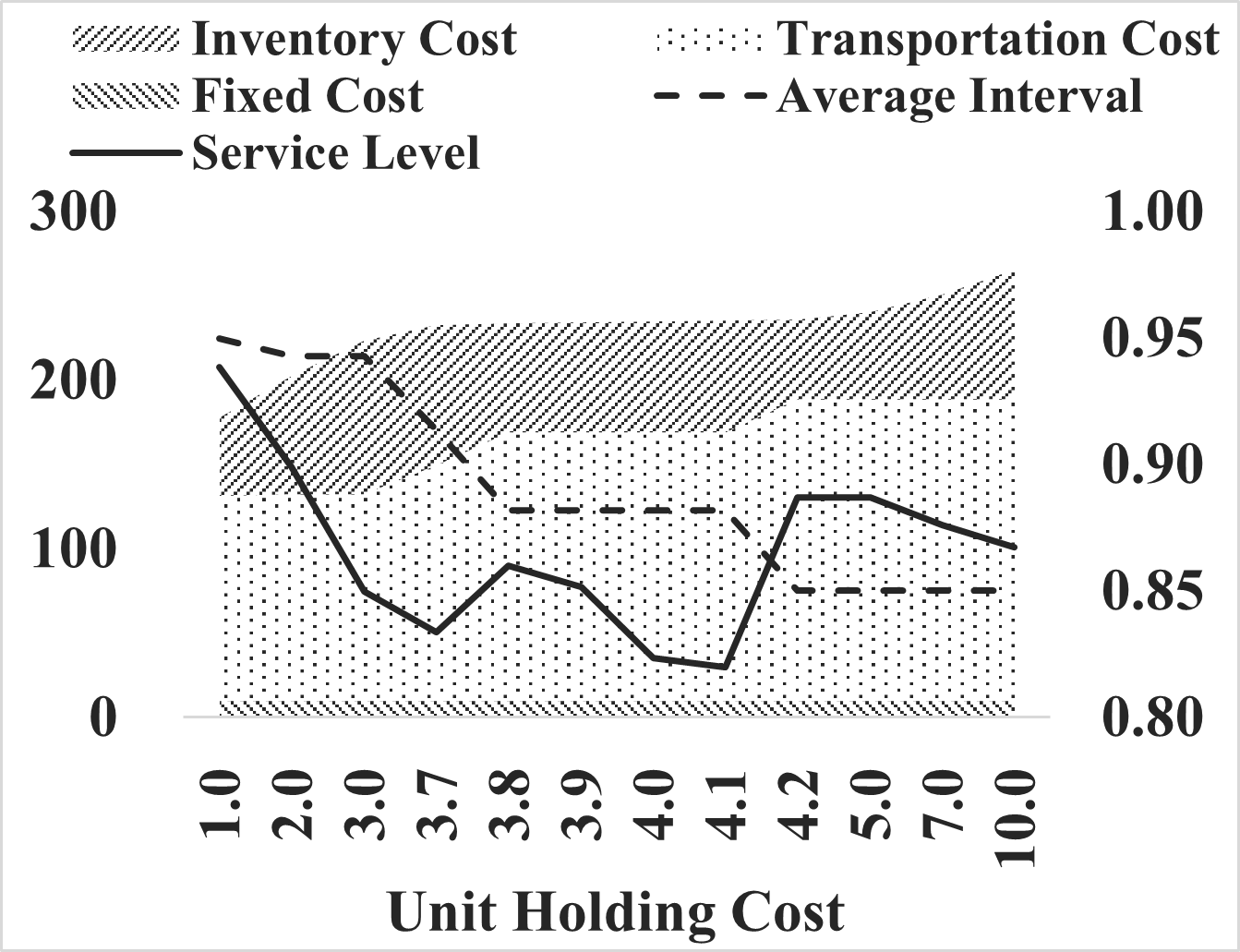}}
    \subfigure[Unit back-order cost]{\includegraphics[width=0.24\linewidth]{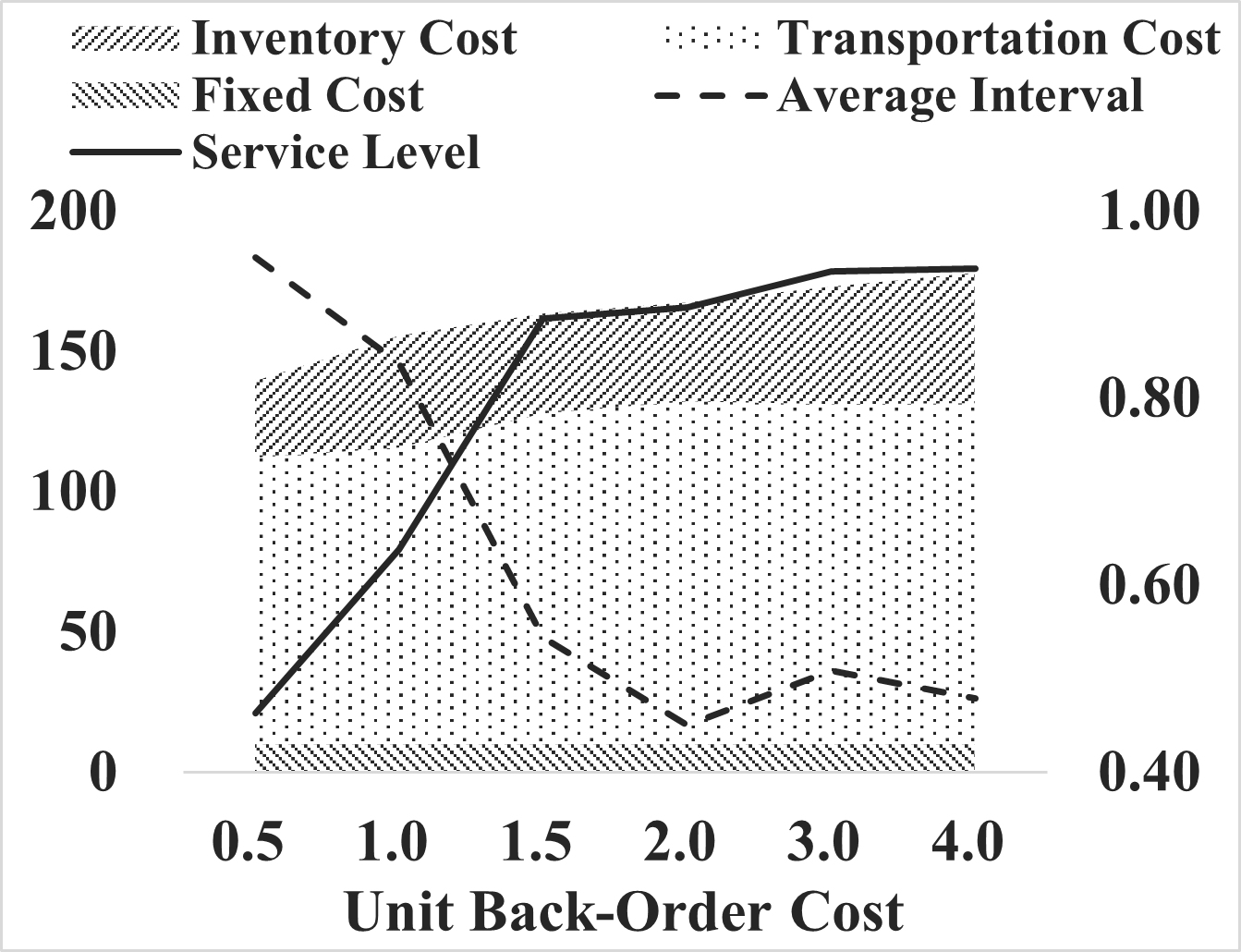}}
    \subfigure[Unit transportation cost]{\includegraphics[width=0.24\linewidth]{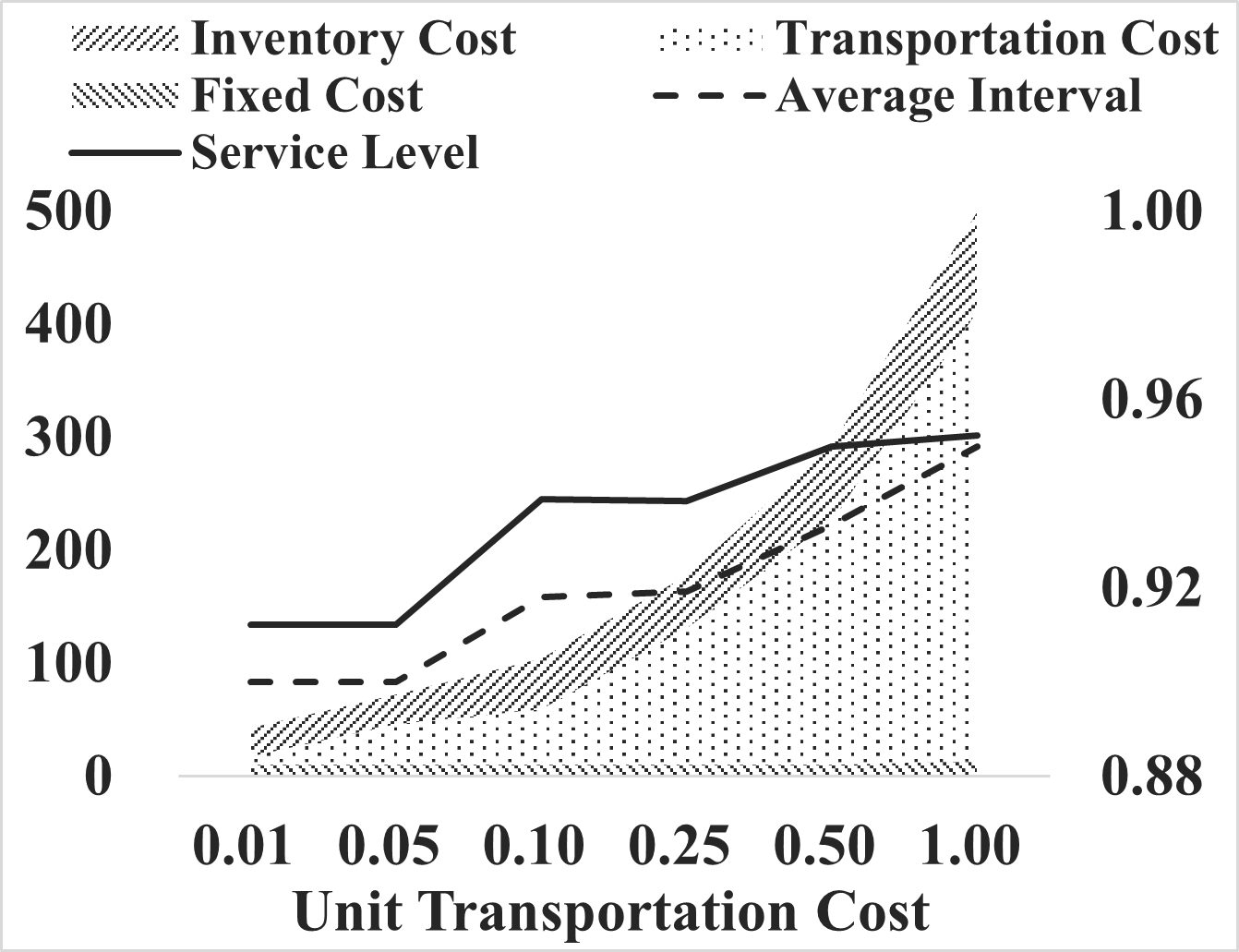}}
    \subfigure[Unit transportation cost (added)]{\includegraphics[width=0.24\linewidth]{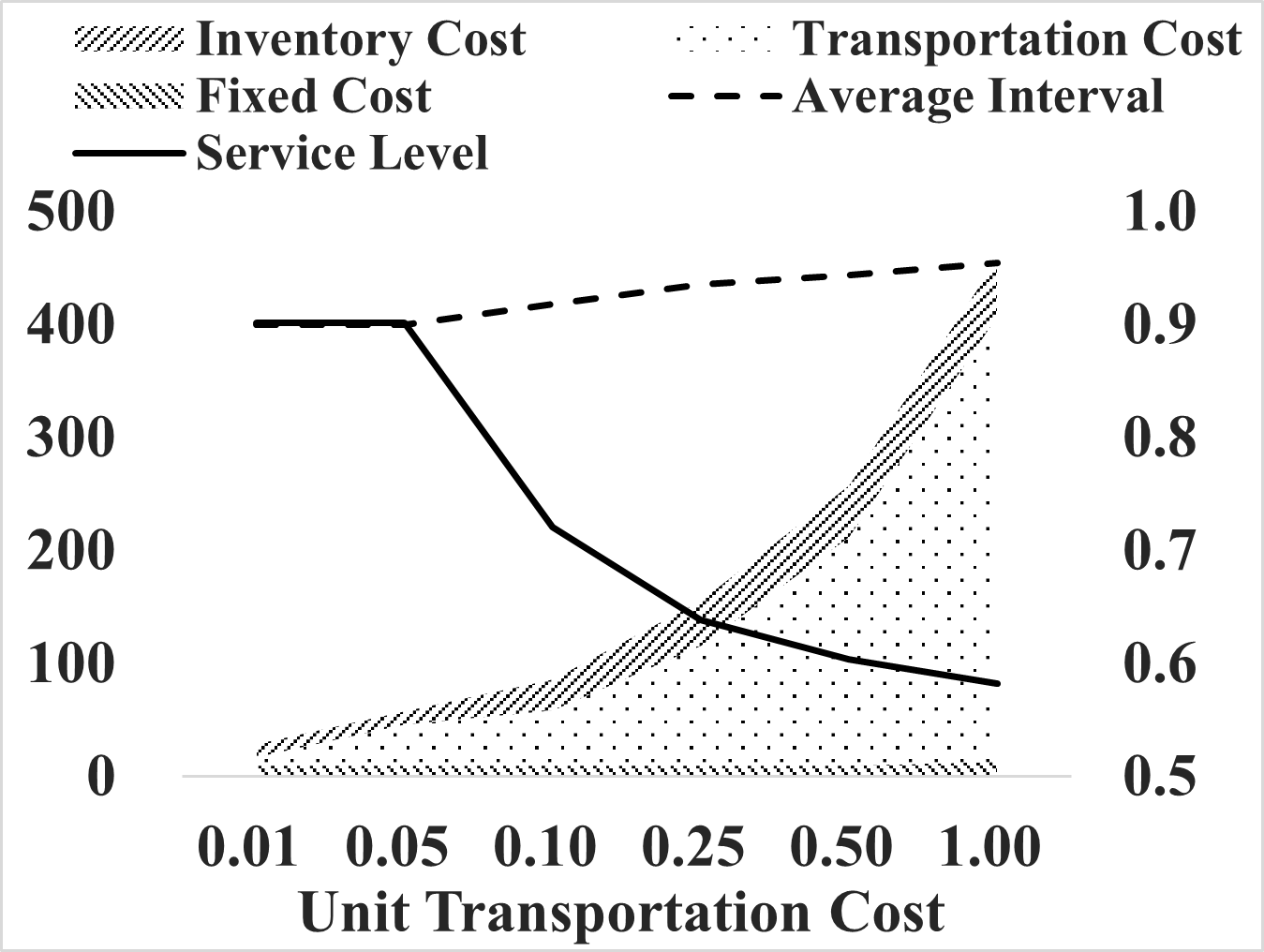}}
    \caption{Influence of unit cost parameters.}
    \label{fig:unitcost}
\end{figure*}

\textbf{Worst-case distribution.}
We additionally illustrate how the worst-case distribution of expected inventory cost evolves as the order-up-to level increases, to provide visual insight into the worst-case distribution introduced in Section~\ref{s3}. We consider a representative retailer $i$ over a replenishment cycle covering periods $1$ to $4$. A turning period value of $5$ indicates that the available inventory is sufficient to satisfy demand in all four periods, implying no risk of stockouts. Accordingly, the worst-case distribution is a discrete distribution with at most five support points. The ambiguity set parameters are specified as follows:
 $\boldsymbol{\mu}^i = \{6.0, 11.0, 13.0, 12.0\}$,
 $\boldsymbol{\underline{\zeta}}^i = \{0.0, 5.0, 4.0, 9.0\}$,
 $\boldsymbol{\bar{\zeta}}^i = \{15.0, 17.0, 22.0, 15.0\}$, and
 $\boldsymbol{\sigma}^i = \{2.7, 1.8, 2.7, 0.9\}$. 
 
The results are presented in Figure~\ref{fig:wd}, which shows how the probability mass is distributed across different turning periods under varying order-up-to levels $s$. Each bar corresponds to the probability of the worst-case demand path resulting in a specific turning period. The legend indicates the representative demand realizations associated with each case. As $s$ increases, the turning period tends to shift later in the planning horizon, reflecting a reduced risk of stockouts due to higher inventory levels. This visual representation captures how the worst-case distribution adapts dynamically to different replenishment decisions.

\begin{figure*}[!htb]
    \centering
    \subfigure[$s$=10]{\includegraphics[width=0.25\linewidth]{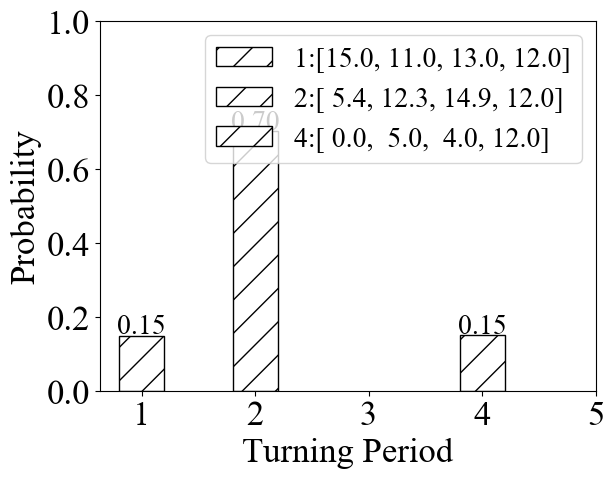}}
    \subfigure[$s$=20]{\includegraphics[width=0.25\linewidth]{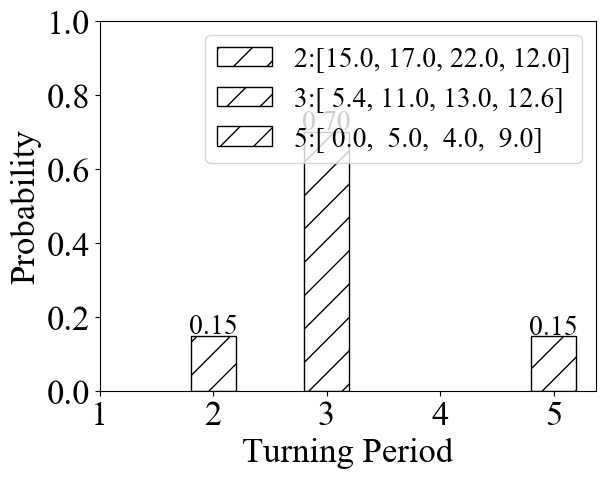}}
    \subfigure[$s$=30]{\includegraphics[width=0.25\linewidth]{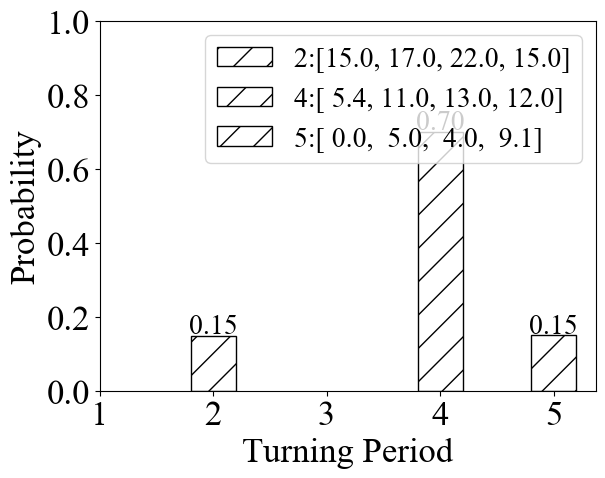}}
    \subfigure[$s$=40]{\includegraphics[width=0.25\linewidth]{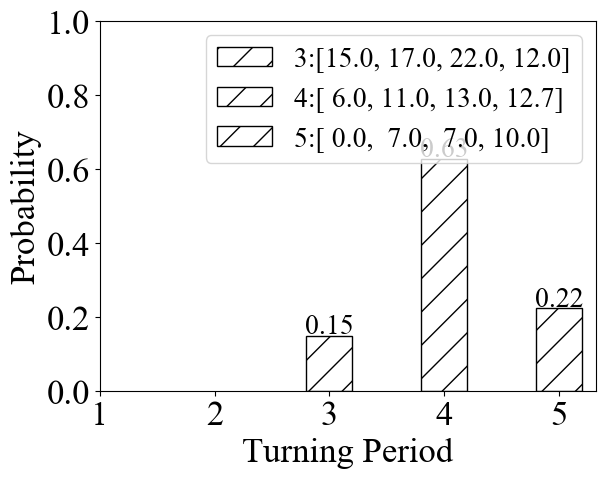}}
    \subfigure[$s$=50]{\includegraphics[width=0.25\linewidth]{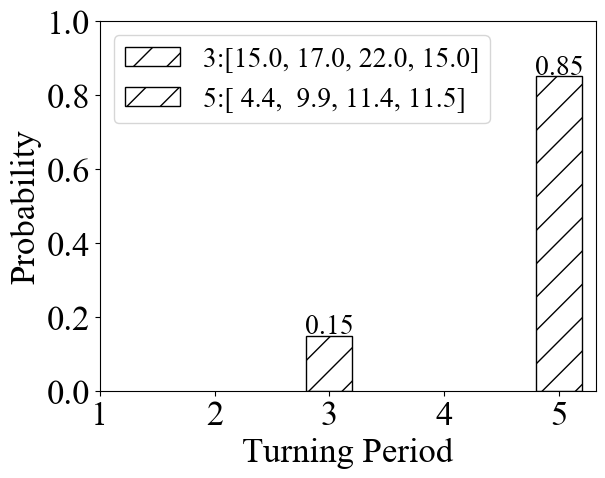}}
    \subfigure[$s$=60]{\includegraphics[width=0.25\linewidth]{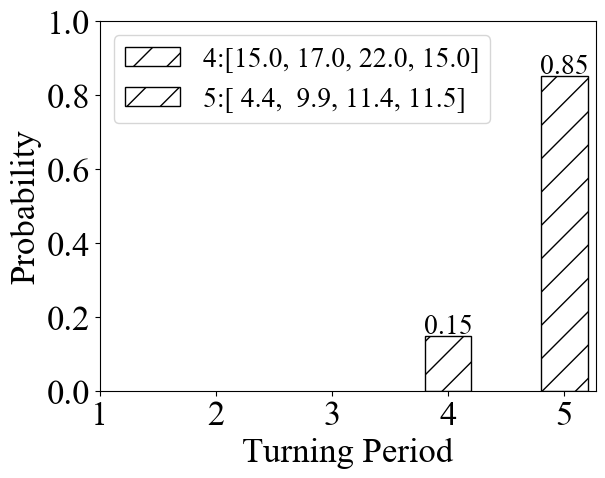}}
    \caption{Worst-case distributions under increasing order-up-to levels.}
    \label{fig:wd}
\end{figure*}

\subsubsection{Comparation with Direct Solver.}
In this subsection, we compare the performance of the proposed nested branch-and-price framework with a direct solution approach using Gurobi on a quadratically constrained quadratic programming (QCQP) formulation, as detailed in Online Appendix~\ref{OA-QCQP}. This QCQP model is a Gurobi-compatible reformulation derived from the compact formulation (CP). Table~\ref{tab:grb} reports the average results across three instances for each configuration.

\begin{table}[!htb]
\centering
\caption{Performance comparison between the proposed nested branch-and-price and Gurobi.}
\label{tab:grb}
\scriptsize
\begin{tabular}{c c r r r r r| c c r r r r r}
\toprule
\multirow{2}{*}{\#Period} & \multirow{2}{*}{\#Retailer} & \multicolumn{2}{c}{N-BnP} & \multicolumn{2}{c}{GRB} & \multirow{2}{*}{Cost Gap}&\multirow{2}{*}{\#Period} & \multirow{2}{*}{\#Retailer} & \multicolumn{2}{c}{N-BnP} & \multicolumn{2}{c}{GRB} & \multirow{2}{*}{Cost Gap} \\
                          \cmidrule(r){3-4}
                          \cmidrule(r){5-6}
                          \cmidrule(r){10-11}
                          \cmidrule(r){12-13}
                          &                             & Cost        & Time        & Cost       & Time       &  &   &                             & Cost        & Time        & Cost       & Time       &                           \\
                          \midrule
\multirow{7}{*}{4}                         & 3                           & 114         & 12          & 157        & T.O.       & 28\%           &\multirow{7}{*}{5}                          & 3                           & 113         & 17          & 154        & T.O.       & 27\%              \\
                                                   \cmidrule(r){2-7}
                                                   \cmidrule(r){9-14}
 & 4                           & 149         & 21          & 176        & T.O.       & 16\%                 && 4                           & 142         & 34          & 191        & T.O.       & 25\%       \\
                                                   \cmidrule(r){2-7}
                                                   \cmidrule(r){9-14}
 & 5                           & 178         & 271         & 252        & T.O.       & 29\%                 && 5                           & 171         & 631         & 269        & T.O.       & 37\%           \\
                                                   \cmidrule(r){2-7}
                                                   \cmidrule(r){9-14}
 & 7                           & 227         & 1857        & 388        & T.O.       & 41\%                  && 7                           & 238         & 2728        & 297        & T.O.       & 20\%          \\
                                                    \cmidrule(r){2-7}
                                                    \cmidrule(r){9-14}
& 10                          & 315         & 2502        & 496        & T.O.       & 36\%         & & 10                          & 328         & T.O.        & inf        & T.O.       & -                 \\
\midrule
\multirow{7}{*}{6}                          & 3                           & 116         & 54          & 160        & T.O.       & 28\%               &\multirow{7}{*}{7}                          & 3                           & 112         & 46          & 149        & T.O.       & 25\%         \\
                                                                            \cmidrule(r){2-7}
                                                                            \cmidrule(r){9-14}
 & 4                           & 148         & 161         & 234        & T.O.       & 37\%                  && 4                           & 142         & 181         & 221        & T.O.       & 36\%      \\
                                                                           \cmidrule(r){2-7}
                                                                           \cmidrule(r){9-14}
  & 5                           & 175         & 914         & 496        & T.O.       & 65\%           & & 5                           & 175         & 1074        & inf        & T.O.       & -               \\
                                                                            \cmidrule(r){2-7}
                                                                            \cmidrule(r){9-14}
 & 7                           & 234         & 2247        & inf        & T.O.       & -             & & 7                           & 238         & 731         & inf        & T.O.       & -             \\
                                                                            \cmidrule(r){2-7}
                                                                            \cmidrule(r){9-14}
 & 10                          & 320         & 2513        & inf        & T.O.       & -          & & 10                          & 339         & T.O.        & inf        & T.O.       & -                \\
                          \bottomrule
\end{tabular}
\end{table}

The results highlight the computational difficulty of solving the QCQP formulation directly: in several cases, Gurobi fails to identify even a feasible solution within the one-hour time limit. In contrast, the nested branch-and-price framework closes, on average, 32.1\% of the cost gap relative to Gurobi’s best solution, with the maximum observed gap being 65\%. Moreover, it achieves this with significantly shorter computation times. These findings demonstrate the practical efficiency and scalability of the proposed decomposition approach.

\subsection{Case Study on Real-life Data}
The real-life case study is based on aftermarket spare parts demand data from SAIC Volkswagen Automotive Co., Ltd. The inventory routing system under study manages the transportation and replenishment of spare parts from SAIC Volkswagen, as the original equipment manufacturer, to downstream retailers and repair shops (hereafter collectively referred to as retailers). Currently, inventory management is handled independently by each retailer, and SAIC Volkswagen does not directly observe the demand they face. As a result, retailer demand is inferred from historical daily order records. Although replenishment and distribution are presently managed as separate functions, SAIC Volkswagen has proactively established a real-time inventory information-sharing system with its retailers. This ongoing digital integration effort positions our proposed framework as a timely and practical reference for advancing SAIC Volkswagen’s transition toward a vendor-managed inventory paradigm.

\begin{figure*}[!htb]
    \centering
    \subfigure[Nationwide retailers]{\includegraphics[width=0.45\linewidth]{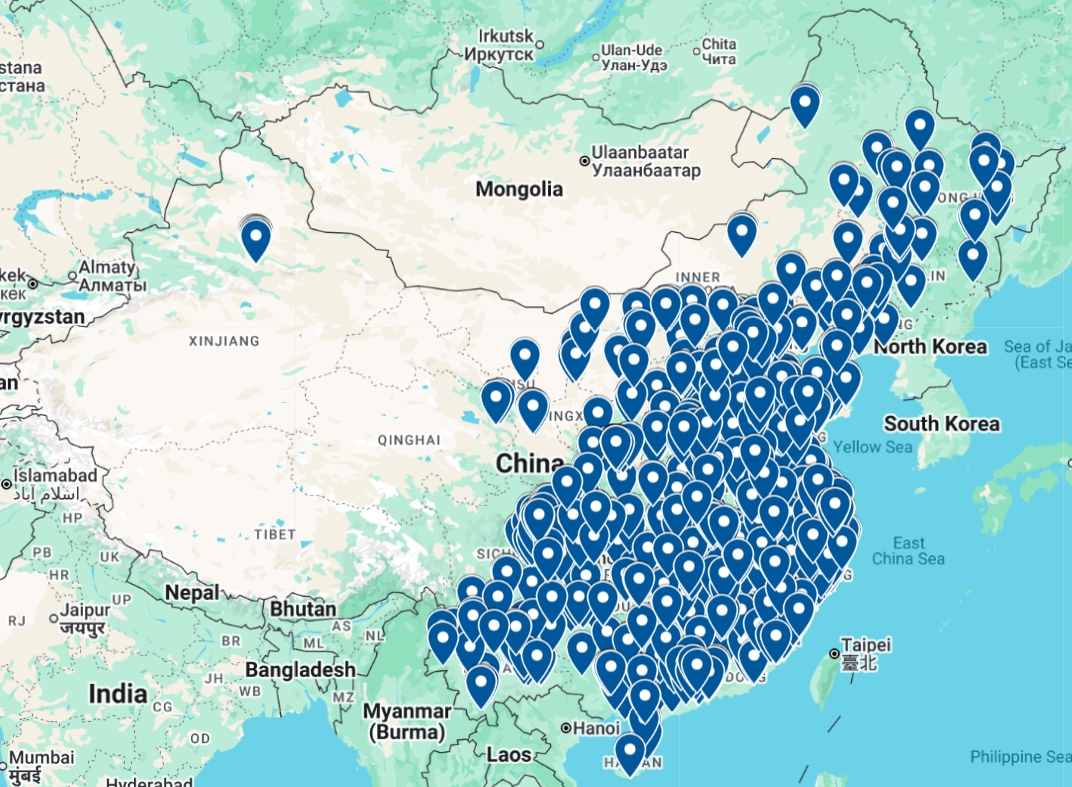}}
    \subfigure[Retailers in Shanghai]{\includegraphics[width=0.32\linewidth]{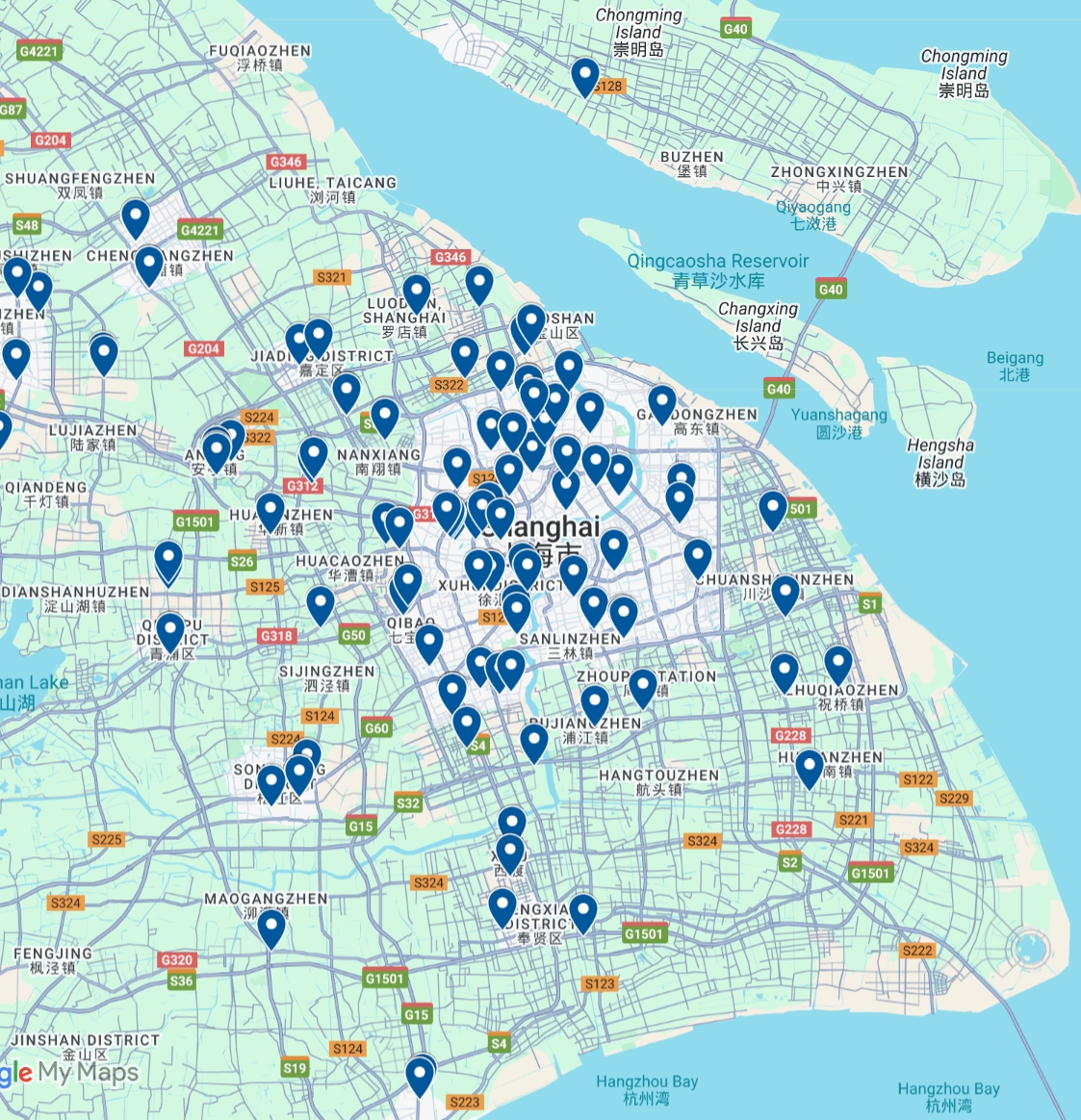}}
    \caption{Geographic distribution of all nationwide retailers and those located in Shanghai.}
    \label{fig:example}
\end{figure*}

SAIC Volkswagen provides data on 2,221 retailers across China, which are served by five regional warehouses. One of these warehouses is located in Shanghai and is responsible for fulfilling orders from retailers within Shanghai and nearby cities. In this study, we focus on the Shanghai region, which includes 112 retailers. The geographic distribution of all nationwide retailers, as well as those located in Shanghai, is illustrated in Figure~\ref{fig:example}. For the 112 Shanghai-based retailers, SAIC Volkswagen has established 10 fixed routing plans. Each day, operational delivery routes are constructed by following the structure of these fixed plans, with minor adjustments made to skip retailers who do not place an order on that day. To construct our experimental instances, we group geographically adjacent routes into clusters, with each cluster forming a distinct instance. This design allows for a direct, instance-level comparison between our optimized routing and inventory solutions and the corresponding operational plans. In this paper, we analyze two such instances: one consisting of 9 retailers located in central Shanghai, formed by combining 2 fixed routes; and another consisting of 12 retailers in suburban areas, derived from 3 fixed routes. The retailer locations and corresponding fixed routes of the second instance (includes 12 retailers) are depicted in Figure~\ref{fig:route}. It is essential to note that some retailers place orders only a few times a month. These infrequent-order retailers are excluded from our instances to ensure the representativeness and stability of the data. Accordingly, they are also excluded from the real-world cost evaluations to maintain consistency.

\begin{figure*}[!htb]
    \centering
    \subfigure[Fixed route 1]{\includegraphics[width=0.32\linewidth]{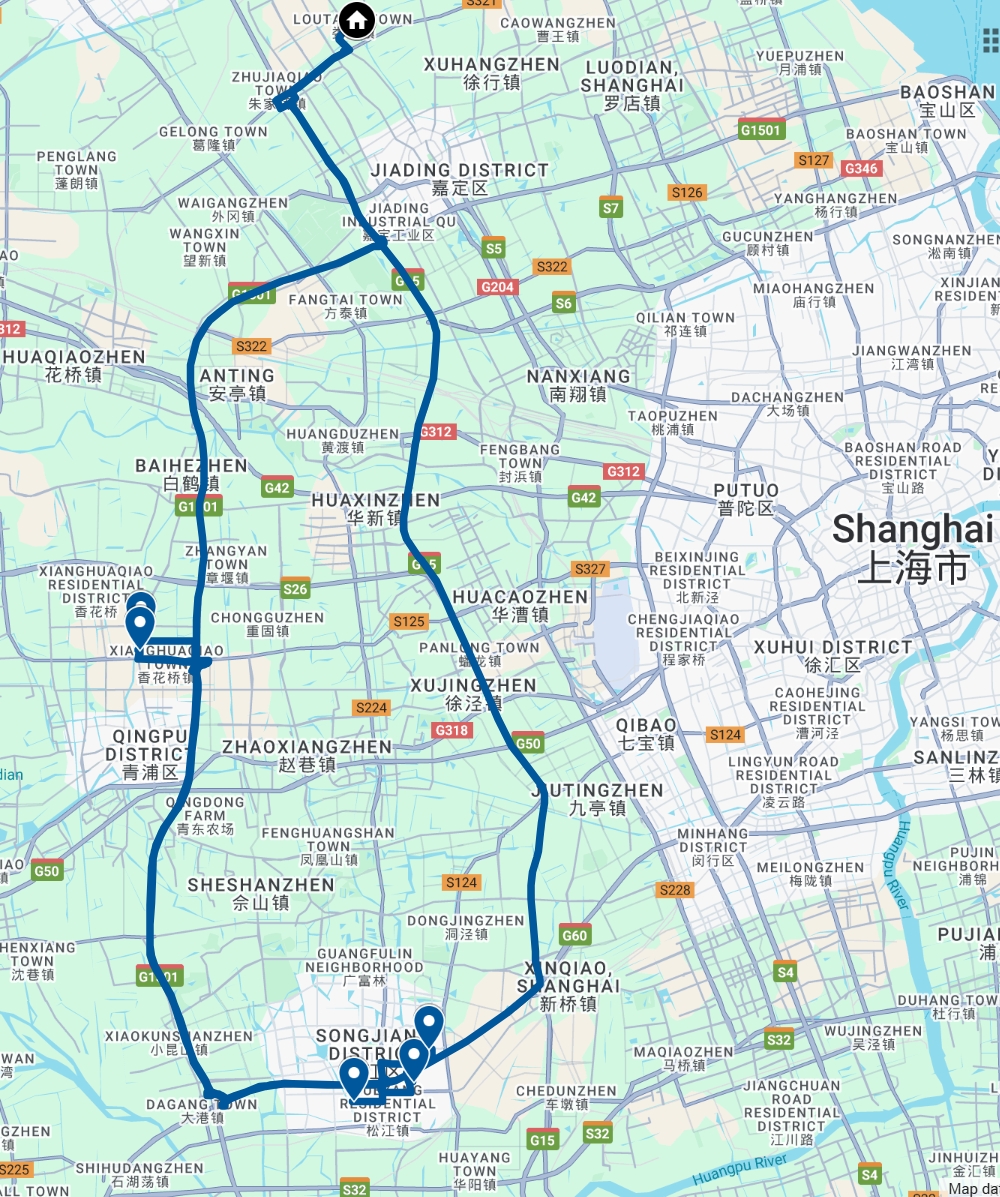}}
    \subfigure[Fixed route 2]{\includegraphics[width=0.32\linewidth]{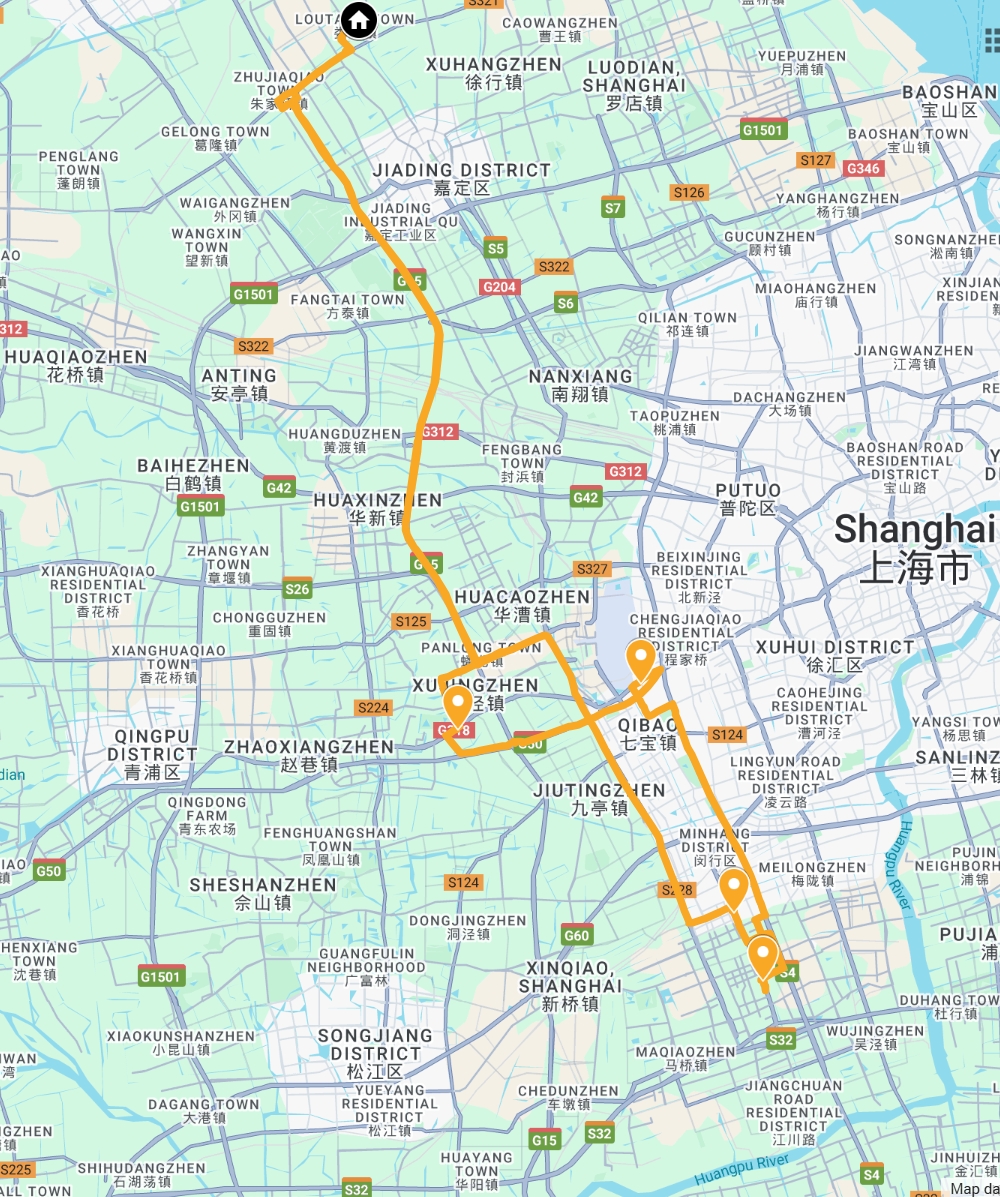}}
    \subfigure[Fixed route 3]{\includegraphics[width=0.32\linewidth]{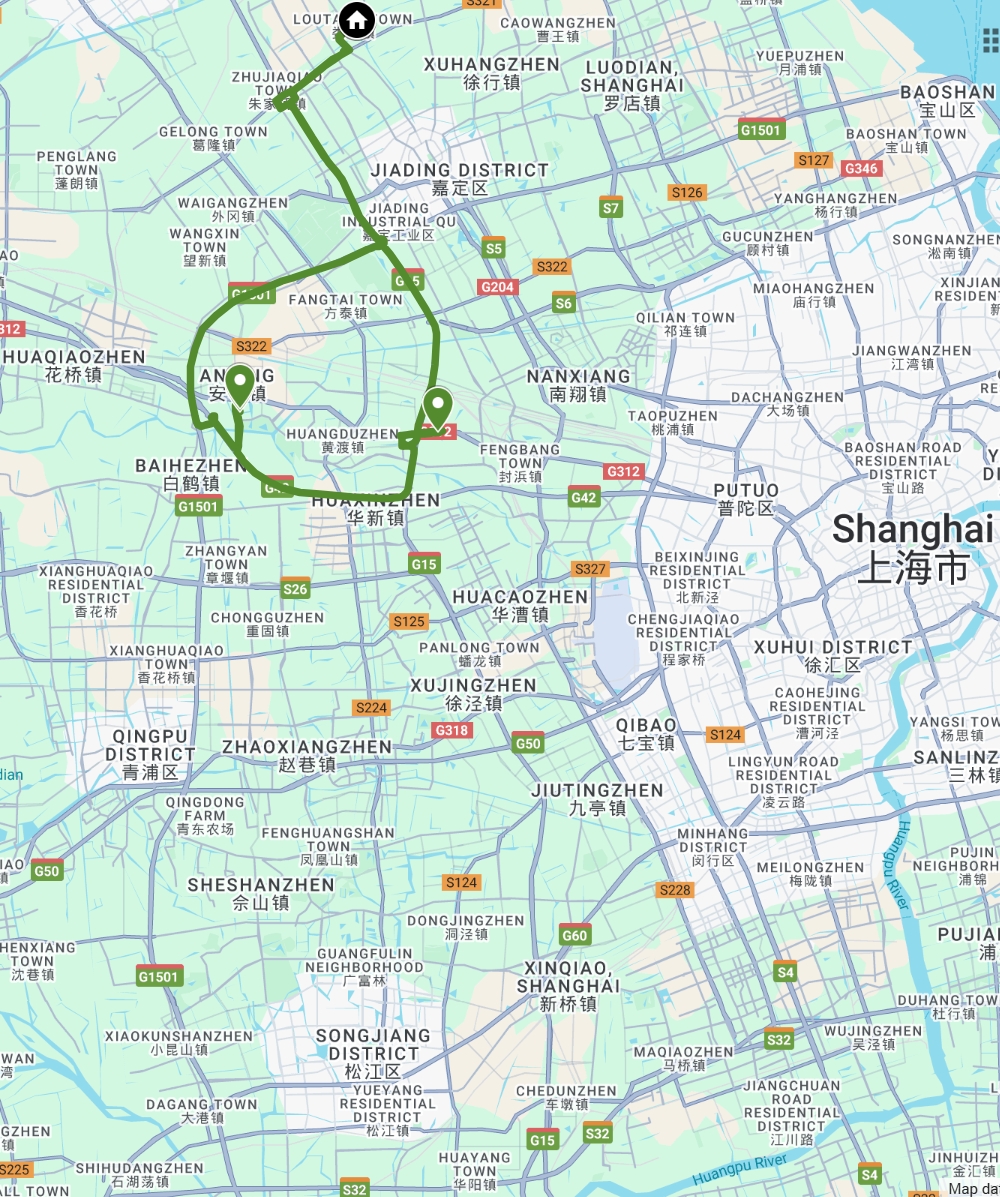}}
    \caption{Retailer locations and corresponding fixed delivery routes for the second experimental instance.}
    \label{fig:route}
\end{figure*}

In the order records provided by SAIC Volkswagen, each retailer places daily orders involving multiple types of spare parts. Following prior work \citep[e.g.,][]{gaur2004periodic}, we aggregate the demand across different spare parts into a unified measure for modeling purposes. However, because spare parts vary significantly in size and weight, e.g., screws versus car tires, a naive aggregation could lead to distortions, where extremely large or small items dominate the overall measure. To address this, we first pre-selected a subset of representative spare parts by filtering out those with extreme physical characteristics. Specifically, we excluded spare parts with volume less than 0.01 cubic meters or greater than 1 cubic meter, and with weight less than 0.1 kilograms or greater than 10 kilograms. In addition, parts with missing or zero values for volume or weight, due to data entry omissions, were also removed. To jointly account for volume and weight, we computed the volumetric weight (also known as dimensional weight) for each part, which is widely used in logistics to capture the density-related cost of shipping. Volumetric weight is calculated as the volume in cubic centimeters divided by 5000, yielding a comparable weight in kilograms. We further excluded parts with volumetric weight below 0.1 kilograms or above 10 kilograms. For each part, we took the maximum of its actual weight and its volumetric weight as the effective weight, and used this as the basis for aggregating spare parts into a single demand unit. After these preprocessing steps, 18,302 distinct spare parts remained in the dataset. For each day and each retailer, we aggregated the effective weights of all ordered spare parts to obtain the total daily demand. These aggregate values exhibited significant variability: the maximum was 863.73 kg, while the minimum was 0. To normalize the data for modeling, we scaled the daily aggregated demand values to the range [0, 60], and rounded them to the nearest integer to obtain a proxy for the retailer-level daily demand. However, we observed a heavy-tailed distribution: only 0.65\% of the demand values exceeded 20, indicating the presence of outliers that could severely affect the stationarity and robustness of the dataset. To mitigate this, we truncated the scaled demand values at 20, using it as the upper cap for all observations.

We utilize demand data from the entire year of 2020, excluding national holidays, resulting in 287 effective observation days. The first 259 days are used as training data to estimate the ambiguity set, while the remaining 28 days are reserved for out-of-sample performance evaluation. Figure~\ref{fig:real-data} illustrates the demand patterns of three representative retailers over a seven-period cycle corresponding to a weekly schedule. As shown in the figure, demand is typically near zero during periods 6 and 7 (i.e., weekends), while demand in other periods remains generally low, typically below 10 units, with occasional spikes reaching up to 20 units. Additionally, significant heterogeneity exists across retailers; for instance, the retailer in subfigure (c) exhibits consistently lower demand levels compared to those in subfigures (a) and (b). These patterns highlight both the temporal and spatial variability of demand, as well as the considerable distributional ambiguity present in real-world operations.

\begin{figure*}[!htb]
    \centering
    \subfigure[]{\includegraphics[width=0.31\linewidth]{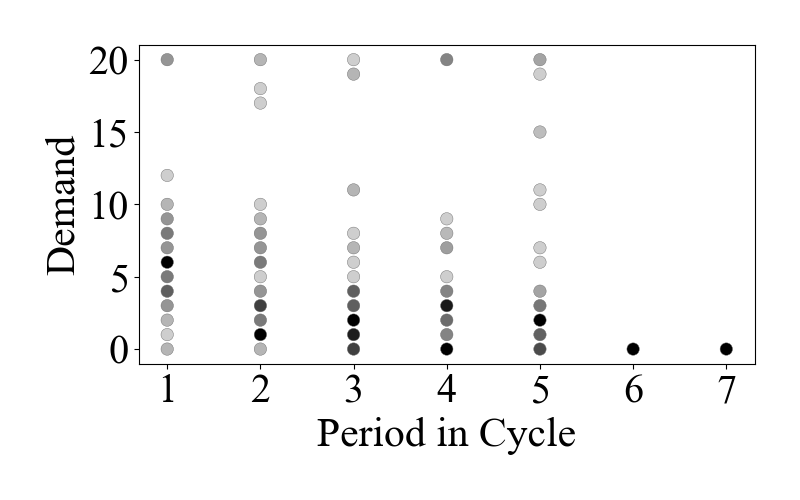}}
    \subfigure[]{\includegraphics[width=0.31\linewidth]{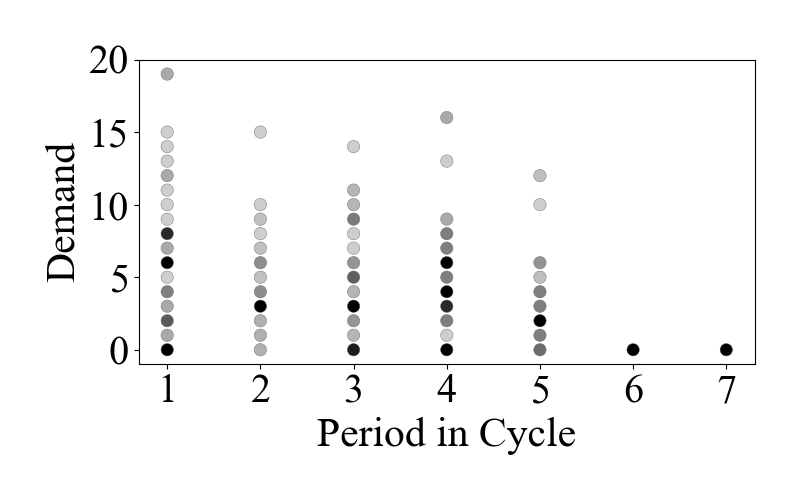}}
    \subfigure[]{\includegraphics[width=0.31\linewidth]{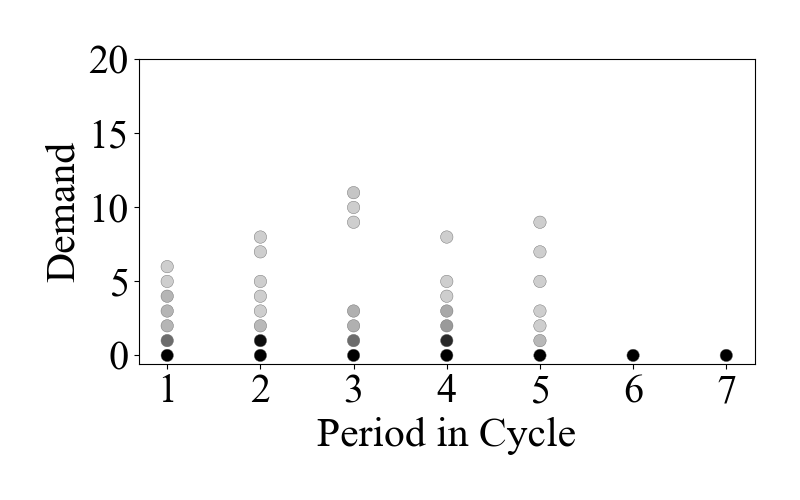}}
    \caption{Representative real-life demand patterns over a weekly cycle ($T=7$).}
    \label{fig:real-data}
\end{figure*}

Although the demand data exhibit a natural weekly cycle ($T=7$), we evaluate system performance under a range of cycle lengths, $T \in \{5,6,7,8\}$, to assess the sensitivity of planning quality to the selected periodicity. Under the fixed-interval policy, the demand sequence is assumed to be stationary, and the ambiguity set parameters are estimated accordingly, independent of the cycle length. In contrast, under the consistent and flexible policies, the demand sequence is treated as cyclic with period $T$, and the ambiguity set parameters are estimated under this assumption, enabling cycle-aware planning. The vehicle capacity is set to 48 units, corresponding to the maximum daily load observed under the operational fixed routes serving the selected retailers. This setting ensures a fair and consistent comparison between our experimental results and the real-world operational baseline. The backorder cost is set to 12, and the unit transportation cost is set to $\rho = 2.5$. To accommodate the computational complexity of the real-life instances, we allocate extended solving time: the second-level branch-and-price solver is limited to one hour per call, and the total runtime for solving each instance is capped at 24 hours.

\begin{table}[!htb]
\centering
\caption{Solution performance of different service policies under real-life data.}
\label{tab:real-solution}
\scriptsize
\begin{tabular}{c c c r r rr r r r r r r r}
\toprule
Inst                & \#Period           & Policy         & Time (s) & T.O.\% & Cost & \#Cluster & Avg I. & S.L. & Vehicle Util & O.\%   & Avg O. & E.T.\% & Avg E.T. \\
\midrule
\multirow{12}{*}{1} & \multirow{3}{*}{5} & Fixed-Interval & 1176     & 0\%    & 372  & 2.0       & 2.5    & 93\% & 34\%         & 0.0\%  & 0.0    & 2.9\%  & 56.0     \\
                    &                    & Consistent     & T.O.     & 83\%   & 397  & 2.0       & 1.2    & 96\% & 41\%         & 20.2\% & 3.6    & 5.1\%  & 16.5     \\
                    &                    & Flexible       & T.O.     & 5\%    & 352  & 2.0       & 1.2    & 97\% & 43\%         & 13.9\% & 4.0    & 5.9\%  & 17.0     \\
                                                                             \cmidrule(r){2-14}
& \multirow{3}{*}{6} & Fixed-Interval & 1175     & 0\%    & 372  & 2.0       & 2.5    & 93\% & 34\%         & 0.0\%  & 0.0    & 2.9\%  & 56.0     \\
                    &                    & Consistent     & T.O.     & 67\%   & 404  & 2.0       & 1.2    & 96\% & 31\%         & 13.4\% & 1.2    & 2.4\%  & 29.0     \\
                    &                    & Flexible       & T.O.     & 52\%   & 358  & 2.0       & 1.1    & 96\% & 36\%         & 10.2\% & 3.7    & 5.3\%  & 20.5     \\
                                                                                               \cmidrule(r){2-14}
  & \multirow{3}{*}{7} & Fixed-Interval & 1213     & 0\%    & 372  & 2.0       & 2.5    & 93\% & 34\%         & 0.0\%  & 0.0    & 2.9\%  & 56.0     \\
                    &                    & Consistent     & T.O.     & 67\%   & 329  & 2.0       & 2.0    & 93\% & 36\%         & 0.0\%  & 0.0    & 0.0\%  & 0.0      \\
                    &                    & Flexible       & T.O.     & 56\%   & 307  & 2.0       & 1.9    & 96\% & 40\%         & 0.0\%  & 0.0    & 3.6\%  & 13.0     \\
                                                                                                \cmidrule(r){2-14}
 & \multirow{3}{*}{8} & Fixed-Interval & 1217     & 0\%    & 372  & 2.0       & 2.5    & 93\% & 34\%         & 0.0\%  & 0.0    & 2.9\%  & 56.0     \\
                    &                    & Consistent     & T.O.     & 0\%    & 383  & 2.0       & 1.4    & 97\% & 39\%         & 22.3\% & 2.9    & 2.6\%  & 31.0     \\
                    &                    & Flexible       & T.O.     & 50\%   & 353  & 2.0       & 1.2    & 96\% & 41\%         & 20.8\% & 2.4    & 2.9\%  & 40.0     \\
                    \midrule
\multirow{12}{*}{2} & \multirow{3}{*}{5} & Fixed-Interval & 23       & 0\%    & 625  & 2.0       & 1.0    & 90\% & 36\%         & 0.0\%  & 0.0    & 7.1\%  & 20.3     \\
                    &                    & Consistent     & T.O.     & 0\%    & 728  & 2.0       & 1.2    & 97\% & 46\%         & 25.0\% & 3.0    & 5.4\%  & 19.7     \\
                    &                    & Flexible       & T.O.     & 100\%  & 590  & 2.0       & 1.2    & 93\% & 44\%         & 19.6\% & 3.3    & 8.9\%  & 11.6     \\
                                                                                                                   \cmidrule(r){2-14}
 & \multirow{3}{*}{6} & Fixed-Interval & 23       & 0\%    & 625  & 2.0       & 1.0    & 90\% & 36\%         & 0.0\%  & 0.0    & 7.1\%  & 20.3     \\
                    &                    & Consistent     & T.O.     & 0\%    & 844  & 3.0       & 1.1    & 97\% & 30\%         & 22.3\% & 2.5    & 2.4\%  & 25.0     \\
                    &                    & Flexible       & T.O.     & 100\%  & 591  & 2.0       & 1.2    & 94\% & 40\%         & 14.2\% & 2.2    & 7.1\%  & 18.0     \\
                                                                                                                    \cmidrule(r){2-14}
& \multirow{3}{*}{7} & Fixed-Interval & 24       & 0\%    & 625  & 2.0       & 1.0    & 90\% & 36\%         & 0.0\%  & 0.0    & 7.1\%  & 20.3     \\
                    &                    & Consistent     & T.O.     & 0\%    & 801  & 3.0       & 2.2    & 91\% & 37\%         & 8.3\%  & 1.7    & 1.8\%  & 16.0     \\
                    &                    & Flexible       & T.O.     & 67\%   & 487  & 2.0       & 2.2    & 88\% & 49\%         & 0.0\%  & 0.0    & 7.5\%  & 8.7      \\
                                                                                                                   \cmidrule(r){2-14}
 & \multirow{3}{*}{8} & Fixed-Interval & 25       & 0\%    & 625  & 2.0       & 1.0    & 90\% & 36\%         & 0.0\%  & 0.0    & 7.1\%  & 20.3     \\
                    &                    & Consistent     & T.O.     & 0\%    & 1032 & 5.0       & 1.0    & 97\% & 21\%         & 30.4\% & 3.9    & 0.0\%  & 0.0      \\
                    &                    & Flexible       & T.O.     & 100\%  & 596  & 2.0       & 1.4    & 92\% & 44\%         & 16.2\% & 4.2    & 5.4\%  & 11.7    \\
                                                   \bottomrule    
\end{tabular}
\end{table}

The results are summarized in Table~\ref{tab:real-solution}. For the consistent and flexible policies, performance is highly sensitive to the choice of the cycle length parameter $T$. Specifically, setting $T$ equal to the true cycle length (i.e., $T=7$) yields the lowest costs in most cases, except for the consistent policy under the second instance, where suboptimal routing plans, attributable to solution inefficiencies, lead to higher costs. This highlights the critical importance of accurately identifying the underlying demand periodicity for policy effectiveness. In contrast, the fixed-interval policy exhibits limited sensitivity to the choice of $T$, as its replenishment decisions rely primarily on a stationary demand interpretation rather than explicit cyclicality. Overall, when the true cycle length is employed, the flexible policy outperforms the fixed-interval policy, demonstrating superior adaptability. Moreover, even under cycle length mismatches, the flexible policy remains competitive, underscoring its robustness in realistic demand environments. By comparison, the consistent policy performs well only under the first instance with the correct cycle length; its relatively poorer performance elsewhere is due, in part, to solution inefficiencies that produce non-optimal routing plans, as will be further elaborated below. Additionally, the consistent policy’s performance variability across instances indicates limited robustness. Notable differences also emerge when comparing synthetic and real-world datasets. Vehicle utilization tends to be lower under real-life demand data, reflecting greater variability and uncertainty. Furthermore, emergency transportation occurs more frequently and in larger volumes in the real dataset, highlighting the challenges of routing under unpredictable demand conditions.

We visualize the routing solutions for Instance 2 with a cycle length of $T=7$ in Figures \ref{fig:real-policy1}–\ref{fig:real-policy3}, which correspond to the fixed-interval, consistent, and flexible policies, respectively. Under the fixed-interval policy (Figure \ref{fig:real-policy1}), retailers are partitioned into two clusters, with each cluster executing its assigned route every day. The consistent policy (Figure \ref{fig:real-policy2}) partitions retailers into three clusters: Cluster 1 is served in periods $\{2,3,4,5\}$, Cluster 2 in $\{2,3,4,5,6\}$, and Cluster 3 in $\{1,2,3,4,5\}$. The flexible policy partitions retailers into two clusters. Since routes under the flexible policy may vary across periods, Figure \ref{fig:real-policy3} shows one representative route that visits all retailers within each cluster, and the complete set of period-specific routes is provided in Figure \ref{fig:oa-real-policy} of Online Appendix F. Three key observations emerge from these results. First, the consistent policy does not yield an effective retailer partition, leading to substantial route overlap and low transportation efficiency. This does not necessarily imply that the consistent policy is incapable of finding a good retailer partition (high-quality routing plans); rather, it suggests that its solution procedure lacks efficiency and is unable to identify such a plan within the allotted time limit. Second, although Figure \ref{fig:real-data} shows that weekend demand is nearly zero, all service policies still schedule replenishments during weekends. A plausible explanation is that demand remains consistently low throughout the week, so the optimization shifts its emphasis from inventory cost reduction to routing efficiency. Third, under the consistent policy, certain retailers are visited without receiving any replenishment (not visible in the figures but confirmed in the detailed results, not shown). This highlights an advantage of the flexible policy: it can omit such retailers from the routing plan entirely, thereby lowering transportation costs. In contrast, the consistent policy must adhere to a fixed routing pattern, even when it results in unnecessary visits.

\begin{figure*}[!htb]
    \centering
    \subfigure[Cluster 1 Every period]{\includegraphics[width=0.32\linewidth]{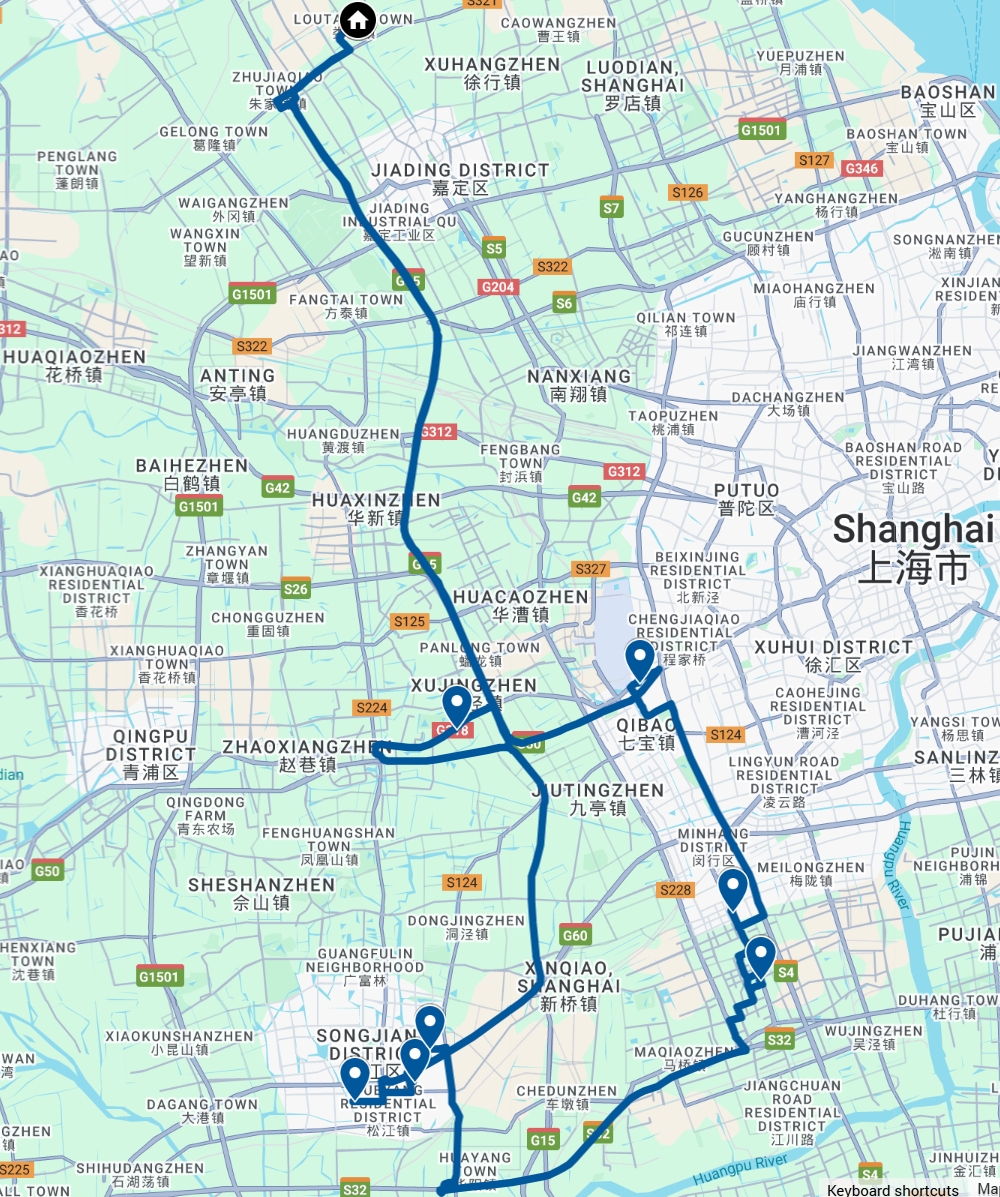}}
    \subfigure[Cluster 2 Every period]{\includegraphics[width=0.32\linewidth]{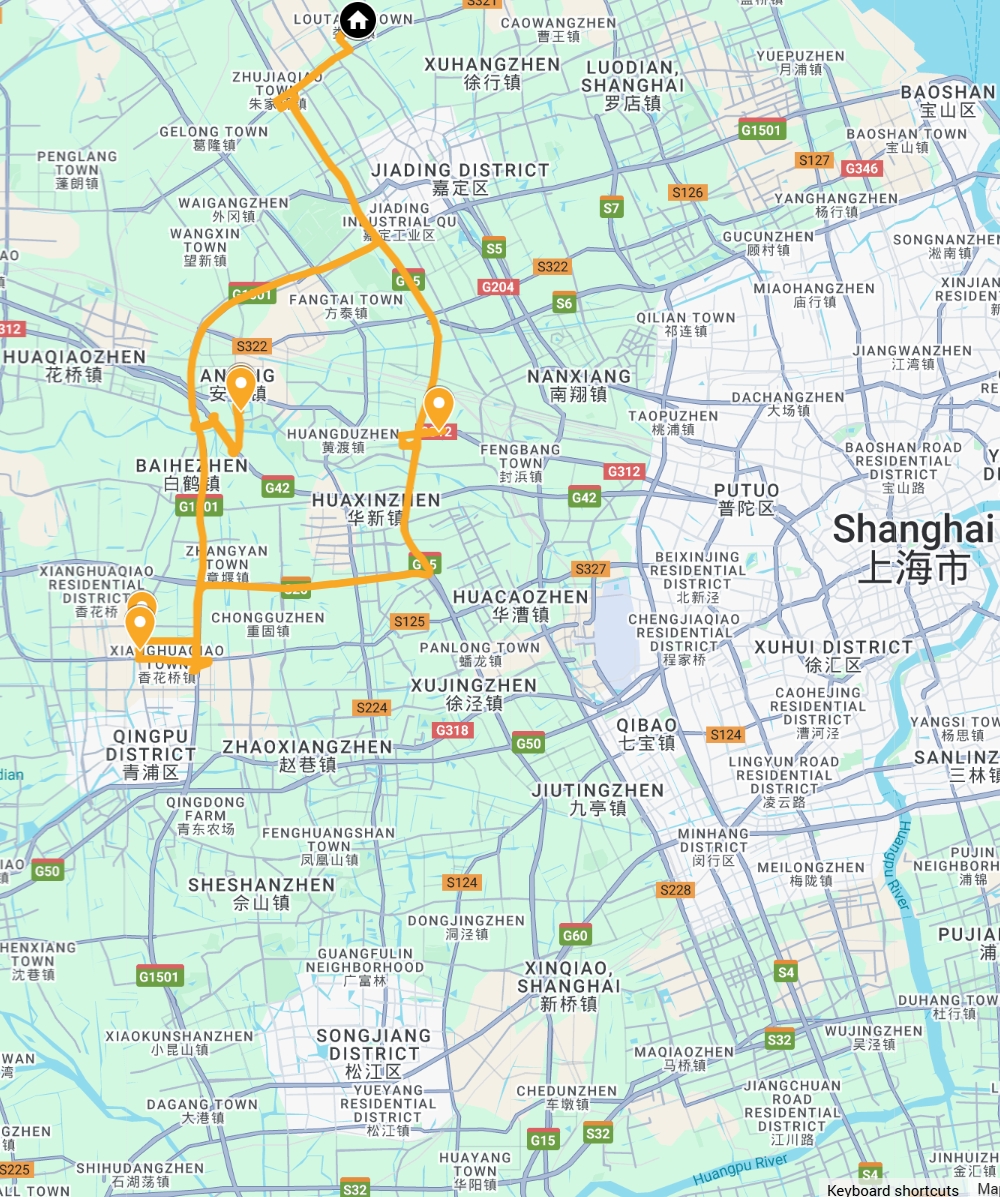}}
            \caption{Solution visualization of the fixed-interval service policy of instance 2 under real-life data.}
    \label{fig:real-policy1}
\end{figure*}

\begin{figure*}[!htb]
    \centering
    \subfigure[Cluster 1 Periods 2,3,4,5] {\includegraphics[width=0.32\linewidth]{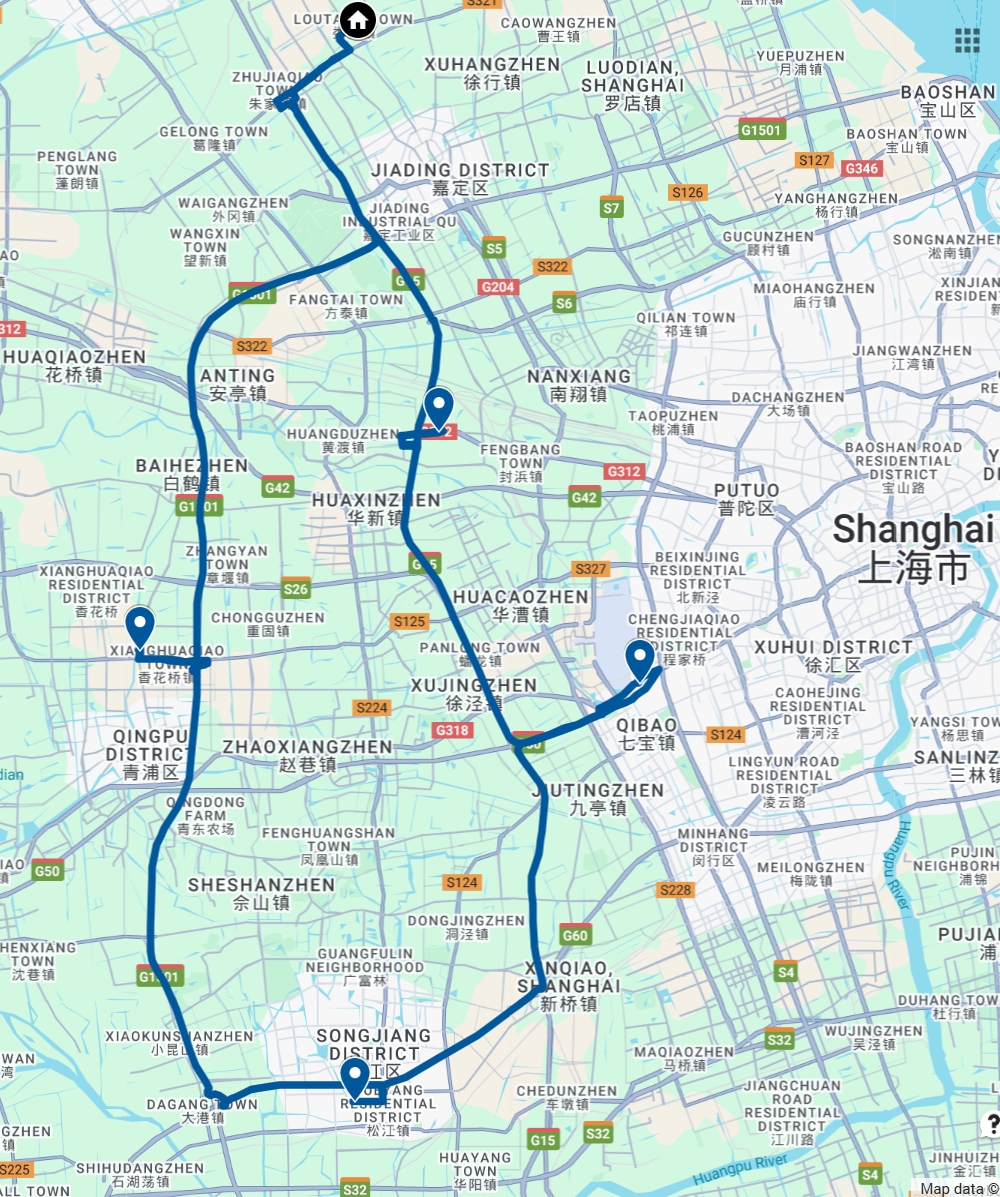}}
    \subfigure[Cluster 2 Periods 2,3,4,5,6]{\includegraphics[width=0.32\linewidth]{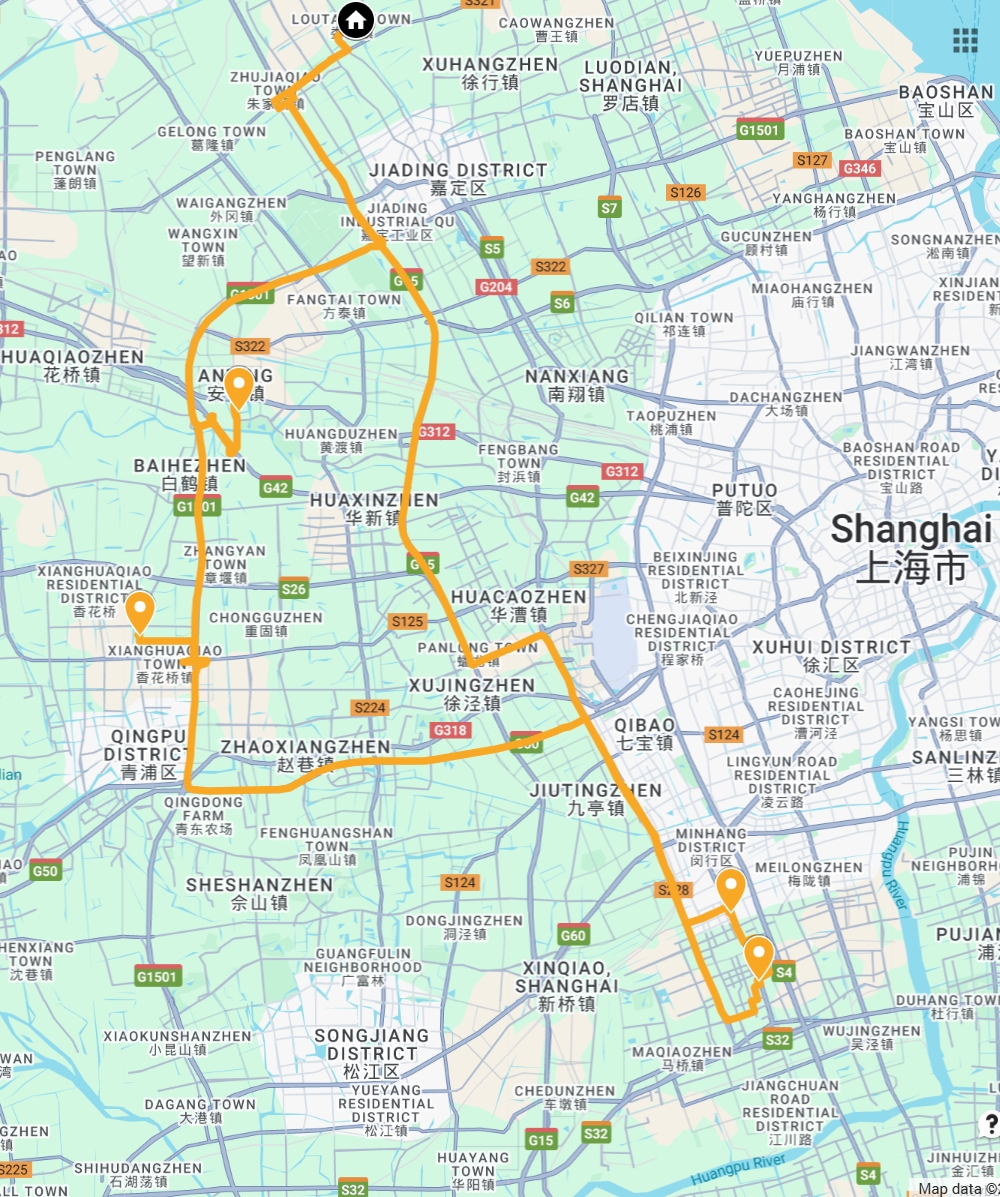}}
    \subfigure[Cluster 3 Periods 1,2,3,4,5]{\includegraphics[width=0.32\linewidth]{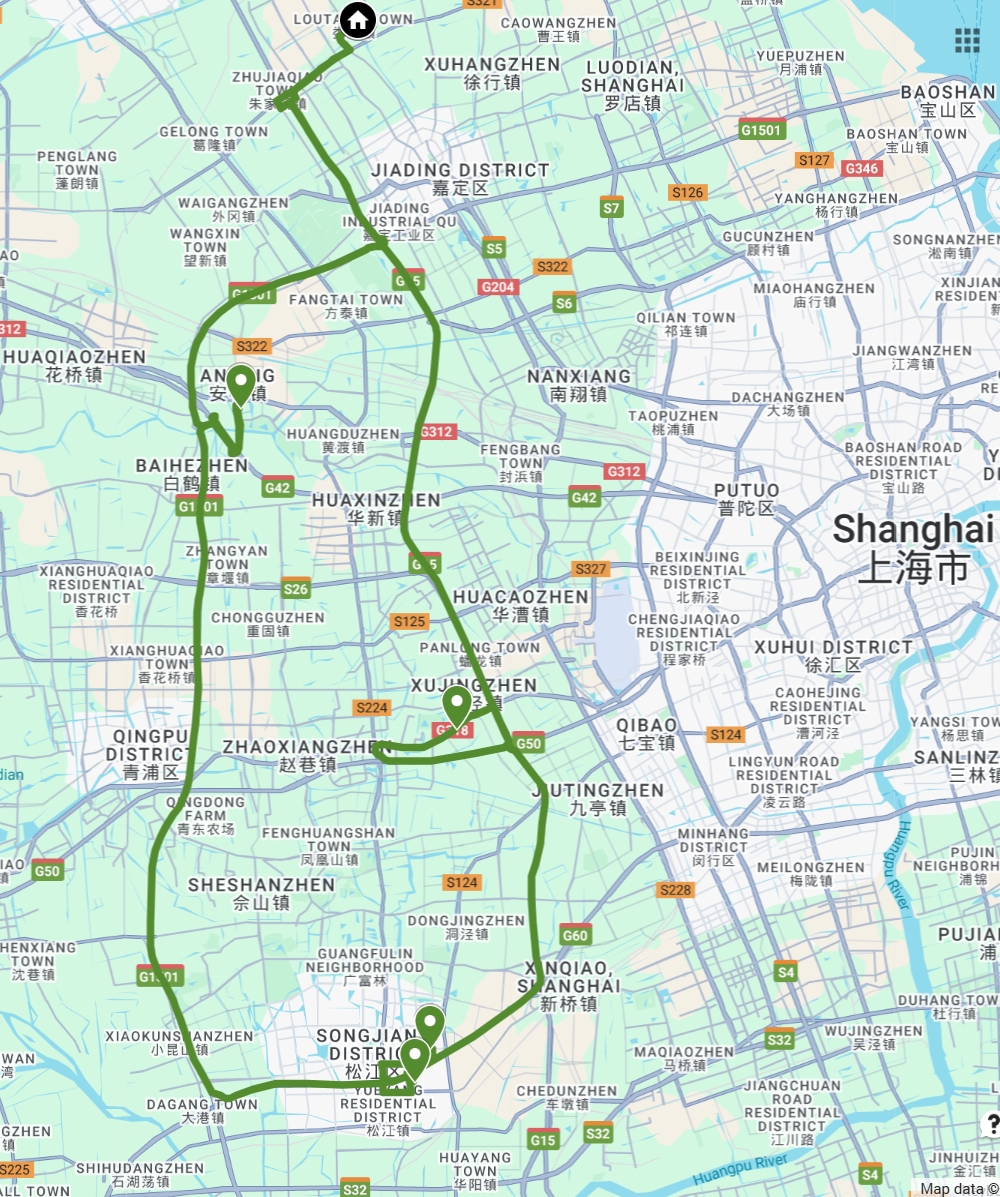}}
        \caption{Solution visualization of the consistent service policy of instance 2 under real-life data.}
    \label{fig:real-policy2}
\end{figure*}

\begin{figure*}[!htb]
    \centering
    \subfigure[Cluster 1]{\includegraphics[width=0.32\linewidth]{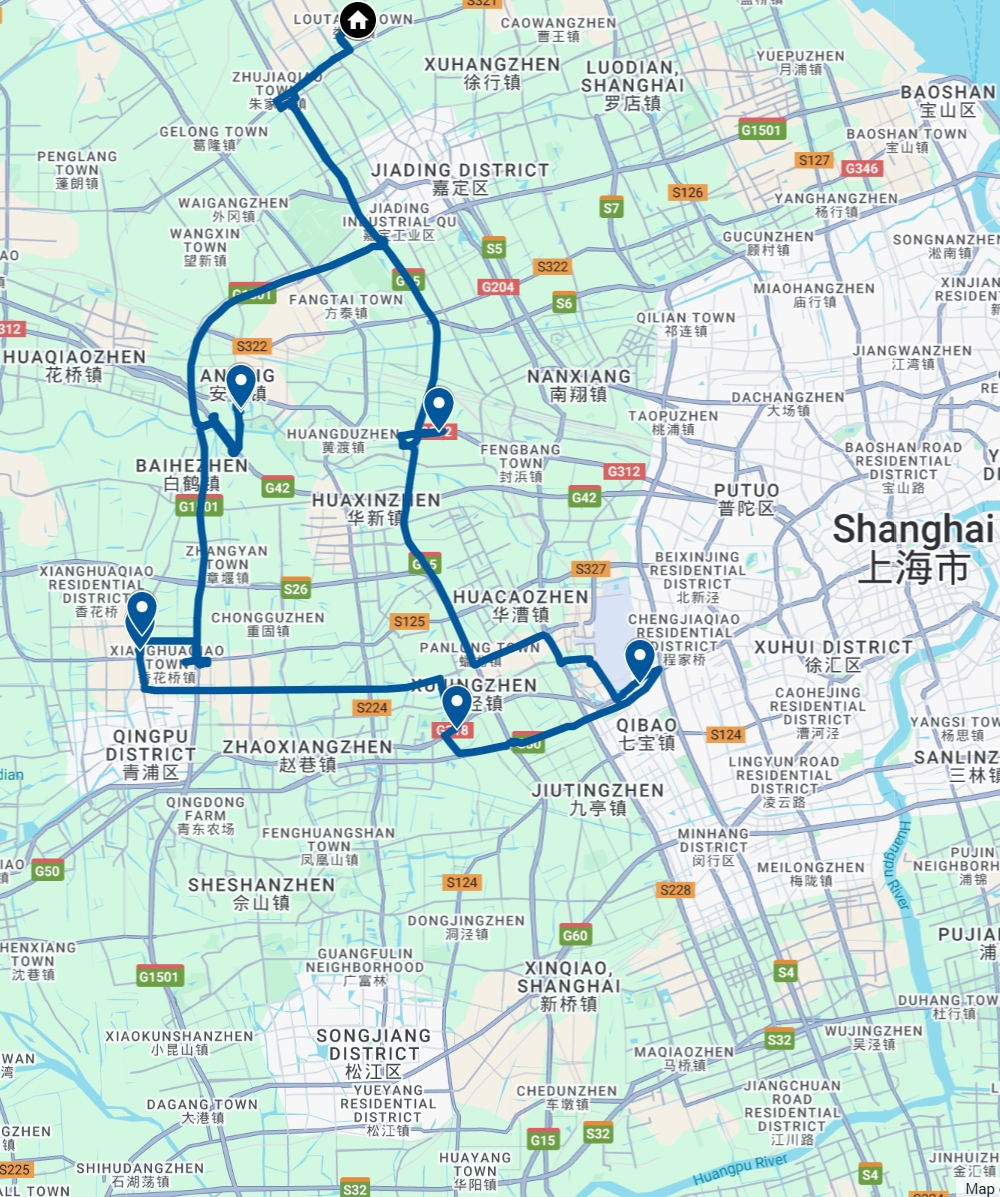}}
    \subfigure[Cluster 2]{\includegraphics[width=0.32\linewidth]{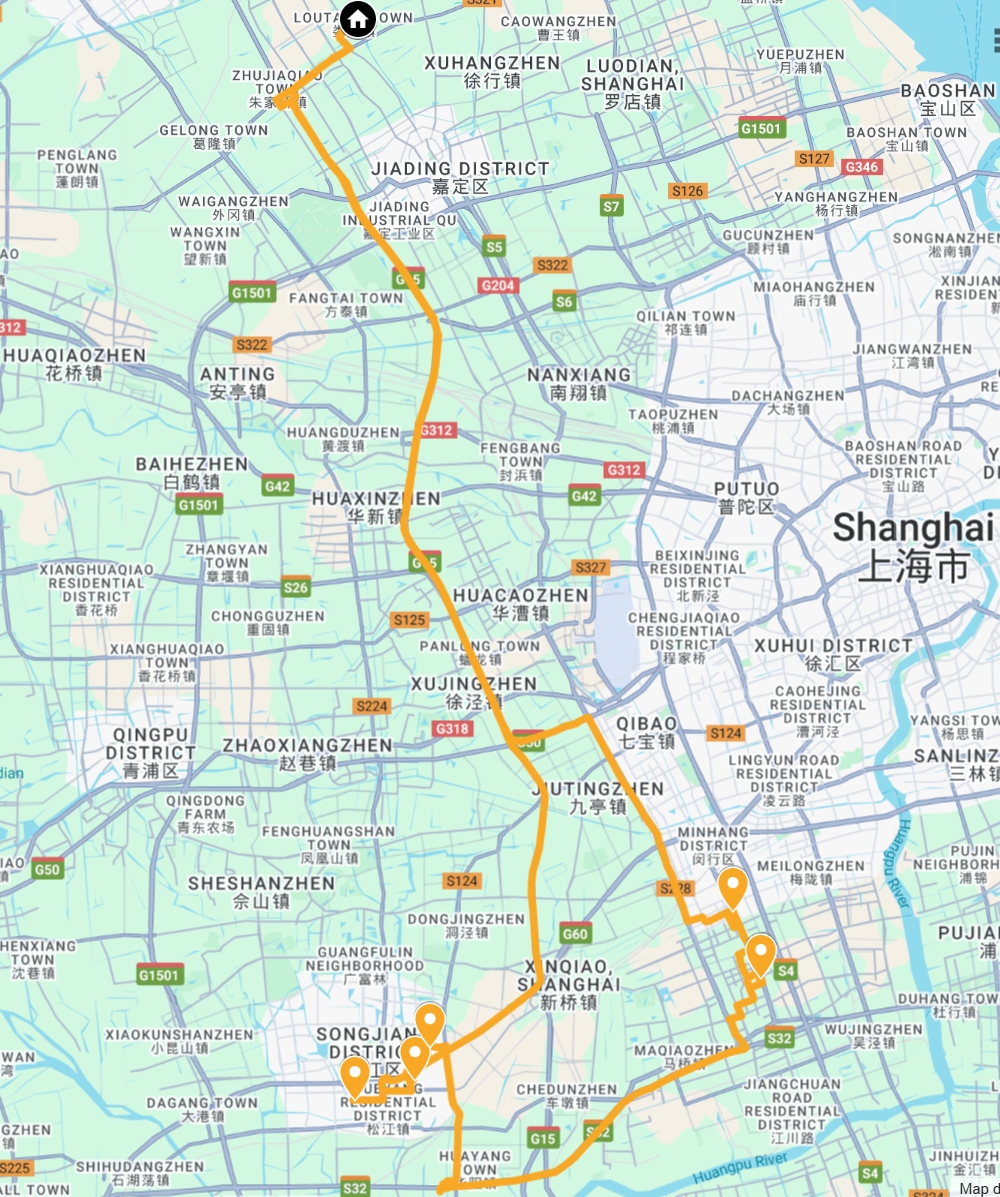}}
        \caption{Solution visualization of the flexible service policy of instance 2 under real-life data.}
    \label{fig:real-policy3}
\end{figure*}


Next, we compare the simulation-period performance of our proposed approach (with cycle length being 7) with the operational strategy currently employed by SAIC Volkswagen, referred to as the manual-planning policy. Two important clarifications are necessary for interpreting this comparison. First, the SAIC Volkswagen aftermarket system operates under a two-level structure: retailers independently determine their replenishment decisions based on local demand forecasts, and the warehouse subsequently fulfills these orders by dispatching the required spare parts. In our study, historical orders are used as a proxy for demand. Under this setting, the manual-planning policy can be interpreted as a one-level system in which the decision-maker has access to ground-truth demand and always applies a zero-inventory replenishment strategy, that is, delivering exactly the amount needed to meet demand on each day. As a result, there are no inventory holding costs or backorders at the retailer level, leading to zero inventory-related costs under this policy. It is important to note, however, that such a zero-inventory strategy may not be optimal from a total cost perspective, as alternative replenishment strategies could consolidate shipments to reduce transportation costs. Nonetheless, because the manual-planning policy assumes full knowledge of future demand, we do not compare replenishment strategies directly. Instead, our evaluation focuses solely on transportation performance, specifically, fixed vehicle usage and routing efficiency. Second, the delivery routes used in the manual-planning policy are conceptually similar to those employed in the fixed-interval or consistent policies with a daily replenishment. However, there are important differences. In our fixed-interval and consistent policies, transportation costs are incurred for visiting all retailers on the predetermined route, even if no delivery is made to a specific retailer on a given day, due to the route consistency requirement (see Assumption~\ref{a2}). In contrast, the manual-planning policy dynamically adjusts routes each day, skipping retailers that do not place orders, thereby achieving lower transportation costs. Moreover, the manual routes are based on extensive operational experience and are likely refined through long-term historical data and practical heuristics. In contrast, our approach constructs routing and inventory policies using only the first 259 days of demand data from 2020, highlighting the data efficiency and adaptability of our method. 

Table \ref{tab:manual} reports the performance of all policies in terms of the number of vehicles used, total transportation cost over the 28-day simulation period, and average vehicle utilization. Across both instances, the flexible policy consistently outperforms all others on every transportation metric, achieving cost reductions of 27.64\% and 36.89\% relative to the manual-planning policy in Instances 1 and 2, respectively. In Instance 1, the manual-planning policy performs worst, whereas in Instance 2, the consistent policy performs worst, followed by the manual-planning policy. These results demonstrate both the practical effectiveness and operational relevance of our framework, particularly when implemented with the flexible policy. For cases where rapid solutions are required, the fixed-interval policy within our framework remains attractive: it consistently outperforms the manual-planning policy, achieving cost reductions of 9.46\% and 1.89\% in Instances 1 and 2, respectively, while requiring substantially less computation time. Finally, although the flexible policy has longer computation times, this limitation is mitigated in practice because the resulting infinite-horizon replenishment plan is solved only once and can be repeatedly applied without re-optimization, rendering the one-time computation effort relatively unimportant in operational settings.

\begin{table}[!htb]
\centering
\caption{Simulation-period performance of different policies.}
\label{tab:manual}
\scriptsize
\begin{tabular}{ c c r r r}
\toprule
Inst               & Policy          & Vehicle Number & Transportation Cost & Vehicle Utilization \\
\midrule
\multirow{4}{*}{1} & Fixed-Interval  & 2              & 5945                & 34\%         \\
                   & Consistent      & 2              & 5434                & 36\%         \\
                   & Flexible        & 2              & 4751                & 40\%         \\
                   & Manual-Planning & 2              & 6566                & 19\%         \\
                   \midrule
\multirow{4}{*}{2} & Fixed-Interval  & 2              & 12326               & 37\%         \\
                   & Consistent      & 3              & 15012               & 37\%         \\
                   & Flexible        & 2              & 7929                & 49\%         \\
                   & Manual-Planning & 3              & 12564               & 22\%        \\
                                   \bottomrule
\end{tabular}
\end{table}

\section{Conclusion}\label{s6}
This paper studies the Cyclic Inventory Routing Problem (CIRP) under demand uncertainty, where demand distributions are assumed to reside within a moment-based ambiguity set. The resulting Distributionally Robust CIRP (DR-CIRP) is particularly challenging due to the combined complexities of distributionally robust optimization and the inherently combinatorial structure of CIRP. To address this, we first conduct theoretical analysis to reformulate the distributionally robust components into tractable forms. Building on these insights, we develop a nested branch-and-price framework that decomposes the DR-CIRP into mixed-integer linear subproblems, each solvable using standard optimization techniques. The proposed method is evaluated on both synthetic and real-world datasets. In a real-world case study using data from SAIC Volkswagen Automotive Co., Ltd., our best-performing approach achieves an average transportation cost reduction of 32.27\% compared with the company’s actual operational practice.

To the best of our knowledge, this is the first work to address CIRP under distributional ambiguity, offering a tractable and practically efficient solution approach. A central component of our framework involves the determination of cyclic inventory policies, a novel problem within the inventory management literature. Our work thus builds a bridge between the closely related domains of inventory management and inventory routing. Moreover, we propose three service policies to reflect diverse cyclic operational practices observed in the literature. These policies are seamlessly integrated into the unified decomposition framework, enabling consistent and comparative evaluation within a single modeling paradigm.

This study also opens several promising directions for future research. First, more efficient algorithms or tailored heuristics may be developed to further enhance computational scalability. In this work, our primary goal is to establish the tractability of DR-CIRP; hence, we adopt a general-purpose decomposition strategy with provable optimality guarantees. Second, while we focus on moment-based ambiguity sets, future research could investigate alternative ambiguity models, such as Wasserstein sets, to assess whether similar or improved performance can be achieved. Finally, some of our theoretical results, including Theorem~\ref{wd_theorem} on worst-case distributions, may have broader applications beyond CIRP. Extending these insights to other classes of distributionally robust operations management problems presents an important avenue for future work.

\bibliographystyle{informs2014}
\bibliography{0-reference}

\clearpage
\begin{APPENDICES} 

\setcounter{equation}{0}
\renewcommand{\theequation}{A\arabic{equation}}
\setcounter{table}{0} 
\setcounter{figure}{0}
\renewcommand{\thetable}{A\arabic{table}}
\renewcommand{\thefigure}{A\arabic{figure}}
\setcounter{algorithm}{0}
\renewcommand{\thealgorithm}{A\arabic{algorithm}}
\section{Gurobi-Solvable Formulation} \label{OA-QCQP}
This section presents a quadratically constrained quadratic programming (QCQP) formulation that is directly solvable using commercial solvers such as Gurobi. The model enables an off-the-shelf solution approach and serves as a benchmark for evaluating the performance of the decomposition algorithm.
\begin{align}
\min\ &\sum_{v\in[V]}\Bigl(pz_v+\frac{1}{T}\sum_{(i, j)\in A^{\rm spatial}} \sum_{t\in [T]}\rho c_{ij}x_{ijtv}+\frac{1}{T}\sum_{i\in[N]}\sum_{(t_1,t_2)\in A^{\rm temporal}}y_{it_1t_2v}\bigl(\alpha_{it_1t_2}+\boldsymbol{\mu}_{it_1t_2}^\top\boldsymbol{\beta}_{it_1t_2}+\boldsymbol{\sigma}_{it_1t_2}^\top\boldsymbol{\gamma}_{it_1t_2}\bigr)\Bigr),\nonumber\\
s.t. \ 
& \sum_{i\in[N]}\sum_{(t_1,t_2)\in \eta^+(\mathcal{G}^{\rm temporal},t_1)}y_{it_1t_2v}\Bigl(s_{it_2}-s_{it_1}+\sum_{t^{\prime}\in[t_1,t_2-1]}U_{t^{\prime}}^i\Bigr)\leq Q \quad \forall t_1\in [T],v\in[V],\nonumber\\
& \sum_{v\in[V]}\sum_{(t_1,t_2)\in \eta^+(\mathcal{G}^{\rm temporal},t_1)}y_{it_1t_2v}\Bigl(s_{it_2}-s_{it_1}+ \sum_{t^{\prime}\in[t_1,t_2-1]} L_{t^{\prime}}^i\Bigr)\geq 0 \quad \forall i\in[N], t_1\in[T],\nonumber\\
& \sum_{v\in [V]}w_{iv}= 1 \quad \forall i\in[N],\nonumber\\
&s_{it_1} \leq M \sum_{v\in[V]}\sum_{(t_1,t_2)\in\eta^+(\mathcal{G}^{\rm temporal},t_1)}y_{it_1t_2v}\quad \forall i\in[N], t_1\in[T],\nonumber\\
& \sum_{(i,j)\in\eta^+(\mathcal{G}^{\rm spatial},i)}x_{ijt_1v}\geq\sum_{(t_1,t_2)\in\eta^+(\mathcal{G}^{\rm temporal},t_1)}y_{it_1t_2v}\quad \forall i\in[N], t_1\in[T],v\in[V],\nonumber\\
& \sum_{t\in[T]}\sum_{(i,j)\in\eta^+(\mathcal{G}^{\rm spatial},i)}x_{ijtv} \leq M w_{iv} \quad \forall i\in[N], v\in[V],\nonumber\\
& \sum_{i\in [N]}w_{iv}\leq Mz_v\quad\forall v\in[V],\nonumber\\
& w_{iv}\in\{0,1\}\quad\forall i\in [N],v\in[V],\nonumber\\
& z_{v}\in\{0,1\}\quad\forall v\in[V],\nonumber\\
&s_{it} \in\mathbb{N} \quad \forall i\in[N], t\in[T], \nonumber\\
    &\sum_{(i,j)\in\eta^+(\mathcal{G}^{\rm spatial},i)}x_{ijtv}\leq 1  \quad\forall i\in[N], t\in[T],v\in [V],\nonumber\\
  &\sum_{i\in [N]]} x_{0itv} = \sum_{i\in [N]]} x_{i0tv} = 1 \quad\forall t\in[T],v\in [V],\nonumber\\
  &\sum_{(i,j)\in \eta^+(\mathcal{G}^{\rm spatial},i)} x_{ijtv} = \sum_{(j,i)\in \eta^-(\mathcal{G}^{\rm spatial},i)} x_{jitv} \quad\forall i\in [N], t\in[T], v\in[V],\nonumber\\
  &\sum_{i\in S}\sum_{j\in S, j\neq i}x_{ijtv}\leq |S|-1 \quad\forall S\subset [N], 2\leq|S|\leq N,t\in[T], v\in[V],\nonumber\\
  & \sum_{(i,j)\in A^{\rm spatial}}c_{ij}x_{ijtv}\leq L,\quad\forall v\in[V],\nonumber\\
  & x_{ijtv}\in \{0,1\} \quad\forall (i,j)\in A^{\rm spatial}, t\in[T],v\in[V],\nonumber\\
  & \sum_{v\in[V]}\sum_{(t_1,t_2)\in A^{\rm temporal}} y_{it_1t_2v} \geq 1 \quad\forall i\in[N], \nonumber\\
  & \sum_{(t_1,t_2)\in \eta^+(\mathcal{G}^{\rm temporal},t_1)} y_{it_1t_2v} \leq 1 \quad\forall i\in[N],t_1\in [T], v\in[V], \nonumber\\
  &\sum_{(t_1,t_2)\in\eta^+(\mathcal{G}^{\rm temporal},t_1)}y_{it_1t_2v} = \sum_{(t_2, t_1)\in \eta^-(\mathcal{G}^{\rm temporal},t_1)} y_{it_2t_1v} \quad\forall i\in[N], t_1\in [T], v\in[V], \nonumber\\
  & \sum_{v\in[V]}\sum_{t_1\in [T]}\sum_{t_2\in[1,t_1]}y_{it_1t_2v} = 1 \quad\forall i\in[N], \nonumber\\
  & y_{it_1t_2v}\in \{0,1\} \quad\forall i\in[N],  (t_1,t_2)\in A^{\rm temporal}, v\in[V], \nonumber\\
&-s_{it_1}\eta_{it_1t_20}^{\hat{t}} + s_{it_1}\eta_{it_1t_21}^{\hat{t}} +\underline{\boldsymbol{\zeta}}_{it_1t_2}^\top\underline{\boldsymbol{\nu}}_{it_1t_2}^{\hat{t}}-\bar{\boldsymbol{\zeta}}_{it_1t_2}^\top\bar{\boldsymbol{\nu}}_{it_1t_2}^{\hat{t}}-\boldsymbol{\mu}_{it_1t_2}^\top\boldsymbol{\pi}^{+\hat{t}}_{it_1t_2}+\boldsymbol{\mu}_{it_1t_2}^\top\boldsymbol{\pi}_{it_1t_2}^{-\hat{t}} \geq \nonumber\\
&\quad h(\hat{t}-t_1)s_{it_1}-b\Bigl(\bigl((t_2-1)-\hat{t}\bigr)+1\Bigr)s_{it_1}-\alpha_{it_1t_2} \quad \forall \hat{t}\in[t_1,t_2-1], (t_1,t_2)\in A^{\rm temporal},i\in[N],\nonumber\\
&\beta_{it_1t_2t}=-h(\hat{t}-t)^++b\Bigl(\bigl((t_2-1)-\hat{t}\bigr)+1\Bigr)-b(t-\hat{t})^+-\eta_{it_1t_20}^{\hat{t}}\mathbbm{1}_{t<\hat{t}}+\eta_{it_1t_21}^{\hat{t}}\mathbbm{1}_{t\leq\hat{t}}+\underline{\nu}_{it_1t_2t}^{\hat{t}}- \nonumber\\
&\qquad\qquad\quad \bar{\nu}_{it_1t_2t}^{\hat{t}}-\pi^{+\hat{t}}_{it_1t_2t}+\pi^{-\hat{t}}_{it_1t_2t}\quad\forall t\in[t_1,t_2-1],\hat{t}\in[t_1,t_2-1],(t_1,t_2)\in A^{\rm temporal}, i\in[N],\nonumber\\
&\boldsymbol{\pi}^{+\hat{t}}_{it_1t_2}+\boldsymbol{\pi}^{-\hat{t}}_{it_1t_2}\leq \boldsymbol{\gamma}_{it_1t_2}\quad \forall \hat{t}\in[t_1,t_2-1]\cup \{0\},(t_1,t_2)\in A^{\rm temporal}, i\in[N],\nonumber\\
&-s_{it_1}\eta_{it_1t_20}^{0} + s_{it_1}\eta_{it_1t_21}^{0} +\underline{\boldsymbol{\zeta}}_{it_1t_2}^\top\underline{\boldsymbol{\nu}}_{it_1t_2}^{0}-\bar{\boldsymbol{\zeta}}_{it_1t_2}^\top\bar{\boldsymbol{\nu}}_{it_1t_2}^{0}-\boldsymbol{\mu}_{it_1t_2}^\top\boldsymbol{\pi}_{it_1t_2}^{+0}+\boldsymbol{\mu}_{it_1t_2}^\top\boldsymbol{\pi}^{-0}_{it_1t_2} \geq \nonumber\\
&\qquad\qquad \qquad\qquad\qquad\qquad\ h\Bigl(\bigl((t_2-1)-t_1\bigr)+1\Bigr)s_{it_1}-\alpha_{it_1t_2}\quad\forall (t_1,t_2)\in A^{\rm temporal}, i\in[N],\nonumber\\
&\beta_{it_1t_2t}=-h\Bigl(\bigl((t_2-1)-t\bigr)+1\Bigr)-\eta_{it_1t_20}^{0}+\eta_{it_1t_21}^{0}+\underline{\nu}_{it_1t_2t}^{0}-\bar{\nu}_{it_1t_2t}^{0}-\pi^{+0}_{it_1t_2t}+\pi^{-0}_{it_1t_2t} \nonumber\\
&\qquad\qquad\qquad\qquad \qquad\qquad\qquad\qquad\qquad\qquad \forall t\in[t_1,t_2-1],(t_1,t_2)\in A^{\rm temporal}, i\in[N],\nonumber\\
& \eta_{it_1t_21}^{0} = 0\quad \forall (t_1,t_2)\in A^{\rm temporal}, i\in[N],\nonumber\\
& \underline{\boldsymbol{\nu}}_{it_1t_2}^{\hat{t}},\bar{\boldsymbol{\nu}}_{it_1t_2}^{\hat{t}},\boldsymbol{\pi}_{it_1t_2}^{+\hat{t}},\boldsymbol{\pi}_{it_1t_2}^{-\hat{t}},\boldsymbol{\gamma}_{it_1t_2}\in\mathbb{R}_+^{\hat{T}}\quad \forall \hat{t}\in[t_1,t_2-1]\cup \{0\},(t_1,t_2)\in A^{\rm temporal}, i\in[N],\nonumber\\
&\eta_{it_1t_20}^{\hat{t}},\eta_{it_1t_21}^{\hat{t}}\in\mathbb{R}_+\quad \forall \hat{t}\in[t_1,t_2-1]\cup \{0\},(t_1,t_2)\in A^{\rm temporal}, i\in[N],\nonumber\\
&\alpha_{it_1t_2}\in\mathbb{R}, \ \boldsymbol{\beta}_{it_1t_2}\in\mathbb{R}^{\hat{T}}\quad\forall (t_1,t_2)\in A^{\rm temporal}, i\in[N],\nonumber
\end{align}
where $\hat{T}=t_2-t_1$ if $t_2>t_1$, and $\hat{T}=t_2+T-t_1$ otherwise, representing the length of the interval.

\clearpage

\setcounter{equation}{0}
\renewcommand{\theequation}{B\arabic{equation}}
\setcounter{table}{0} 
\setcounter{figure}{0}
\renewcommand{\thetable}{B\arabic{table}}
\renewcommand{\thefigure}{B\arabic{figure}}
\setcounter{algorithm}{0}
\renewcommand{\thealgorithm}{B\arabic{algorithm}}
\section{Nested Branch-and-Price Decomposition for Alternative Policies} \label{OA-policies}
This section details the nested branch-and-price algorithm as applied to the remaining two service policies: the fixed-interval policy (Section~\ref{OA-policies.1}) and the flexible policy (Section~\ref{OA-policies.2}). For each policy, we describe the problem reformulation, the solution procedure, the branching strategy, and the construction of initial feasible solutions.

\subsection{Fixed-Interval Policy}\label{OA-policies.1}
\subsubsection{Reformulation and Solution Procedure}
Under Assumptions~\ref{a1}-\ref{a2}, the order-up-to levels remain consistent for each retailer across all replenishment periods. Specifically, for any retailer~$i$ in the cluster served by vehicle~$v$, we have $s_{it_1} = s_{it_2} = s_i$ for all $t_1, t_2 \in \mathcal{T}_v^*$, where $\mathcal{T}_v^*$ denotes the set of periods during which vehicle~$v$ replenishes its assigned retailer cluster. As a result, the Chance Constraints~\eqref{st_cc2} are always satisfied, ensuring that inventory overshooting is completely avoided, that is, inventory levels never exceed the prescribed order-up-to levels. 

We now reformulate the same first-level restricted master problem (1st-RMP) defined in Section~\ref{1st-rmp}. Under the fixed-interval policy, replenishment intervals and routes remain identical across periods. That is, each retailer~$i$ served by vehicle~$v$ is visited every $\kappa_v$ periods, and the same subset of retailers is served in each visit. The associated pricing problem for generating a feasible pattern under this policy can be formulated as:
\begin{align}
\min\ &p + \sum_{(i, j)\in A^{\rm spatial}} \tilde{c}_{ij} x_{ij}, \nonumber\\
s.t. \ 
  & \sum_{(i,j)\in A^{\rm spatial}}\kappa_v U^ix_{ij} \leq Q,\nonumber\\ 
  & x_{ij} \in \mathscr{X} \quad \forall (i,j)\in A^{\rm spatial}\nonumber,\\
  & \kappa_v \in \mathbb{N},\nonumber
\end{align}
where $\tilde{c}_{ij}=\frac{1}{\kappa_v}\Bigl(\rho c_{ij}+\sup_{\mathbb{P}\in\mathcal{P}}\mathbb{E}_{\mathbb{P}}\bigl[h\sum_{t=1}^{\kappa_v}(s_{i}-t\zeta^i)^++b\sum_{t=1}^{\kappa_v}(s_{i}-t \zeta^i)^-\bigr]\Bigr)-\pi_i$, with $s_i \in \mathbb{N}$ for all $i\in[N]$.

Let $\mathcal{K} = \{1, 2, \ldots, \kappa_v + 1\}$. Similar to Theorem~\ref{wd_theorem}, for given $s_i$ and $\kappa_v$, the worst-case expected inventory cost for retailer~$i$, defined as $\sup_{\mathbb{P}\in\mathcal{P}}\mathbb{E}_{\mathbb{P}}\bigl[h\sum_{t=1}^{\kappa_v}(s_{i}-t\zeta^i)^++b\sum_{t=1}^{\kappa_v}(s_{i}-t \zeta^i)^-\bigr]$, can be equivalently computed by solving the following optimization problem:
\begin{align}
    \max\ & \sum_{\hat{t}\in\mathcal{K}}\Bigl(h(\hat{t}-1)\bigl(s_i\pi_{\hat{t}}-\frac{\hat{t}}{2}\lambda_{\hat{t}}\bigr)-b(\kappa_v+1-\hat{t})\bigl(s_i\pi_{\hat{t}}-\frac{\kappa_v+\hat{t}}{2}\lambda_{\hat{t}}\bigr)\Bigr)\nonumber\\
    s.t.\ &\sum_{\hat{t}\in\mathcal{K}}\pi_{\hat{t}}=1,\nonumber\\ &\sum_{\hat{t}\in\mathcal{K}}\lambda_{\hat{t}}=\mu_i,\nonumber\\
    & \sum_{\hat{t}\in\mathcal{K}} |\lambda_{\hat{t}}-\pi_{\hat{t}}\mu_i|\leq \sigma_i,\nonumber\\
    & \pi_{\hat{t}}\underline{\zeta}_i\leq \lambda_{\hat{t}}\leq \pi_{\hat{t}}\bar{\zeta}_i\quad \forall \hat{t}\in\mathcal{K},\nonumber\\
    &(\hat{t}-1)\lambda_{\hat{t}}\leq \pi_{\hat{t}}s_i\quad\forall \hat{t}\in\mathcal{K},\nonumber\\
    &\hat{t}\lambda_{\hat{t}}\geq \pi_{\hat{t}}s_i\quad\forall \hat{t}\in\mathcal{K},\nonumber\\
    &\pi_{\hat{t}}\in\mathbb{R}_+, \lambda_{\hat{t}}\in\mathbb{R}\quad \forall \hat{t}\in\mathcal{K}.\nonumber
\end{align}

Accordingly, for each retailer~$i$ and fixed replenishment interval~$\kappa_v$, the optimal order-up-to level~$s_i^*$ and the associated worst-case expected inventory cost can be efficiently obtained using a golden-section search, similar as outlined in Algorithm 1. Once $\tilde{c}_{ij}$ values are computed for all $(i,j) \in A^{\rm spatial}$, the pricing problem reduces to an ESPPRC instance, which can be solved using the labeling algorithm described in Online Appendix \ref{oa-labeling}. The complete pricing routine thus consists of enumerating over all feasible values of $\kappa_v$ and solving the corresponding ESPPRC for each.

\subsubsection{Branching Strategies}
We adopt a single-layer branch-and-price framework under the fixed-interval policy. In the LP relaxation of the restricted master problem, let $q_{\omega}^{*}$ denote the value of variable $q_{\omega}$. When the solution is fractional, we identify the pattern whose $q_{\omega}^{*}$ is closest to 0.6. From the corresponding routing plan, we randomly select one arc and generate two child nodes: one excludes the selected arc from any future patterns, and the other enforces its inclusion.

\subsubsection{Initial Solution Construction}
Under the fixed-interval policy, a solution is constructed for $\kappa_v = 1$. The order-up-to level for each retailer is set to its optimal value $s_i^*$, as detailed in Section B.1.1. The corresponding vehicle routes are then determined by solving a vehicle routing problem, where each retailer’s demand is given by $\kappa_v WVaR^i$, following the specification in Section B.1.1.


\subsection{Flexible Policy}\label{OA-policies.2}
\subsubsection{Reformulation and Overall Framework}

We adopt the same overall decomposition framework for the consistent policy as described in Section~4.1.

\textbf{First Level: Retailer Partitioning.} The first-level restricted master problem remains the same as in Section~4.1.1, which was developed under the consistent policy.

\textbf{Second Level: Routing and Replenishment for One Cluster.}
We define replenishment decision variables as in Section 4.1.2. For routing plans, let $r(t)\in R(t)$ represent a routing plan in period $t$. Define $\lambda_{r(t)}$ as a binary variable indicating whether routing plan $r(t)$ is selected in period~$t$, and let $c_{r(t)}$ denote the associated transportation cost. The parameter $o_{r(t)}^i$ indicates whether routing plan $r(t)$ serves retailer~$i$. The second-level restricted master problem (2nd-RMP) is then formulated as:
\begin{align}
\min\ &p+\frac{1}{T}\sum_{t\in[T]}\sum_{r(t)\in R(t)} c_{r(t)}\lambda_{r(t)}+\sum_{i\in[N]}\sum_{l(i)\in L(i)}\bigl(\frac{1}{T}c_{l(i)}-\pi_i\bigr)\gamma_{l(i)}, \\
s.t. \ 
  & \sum_{r(t)\in R(t)}\lambda_{r(t)} \leq 1\quad\forall t\in[T],\label{pp1}\\
  & \sum_{l(i)\in L(i)}\gamma_{l(i)} \leq 1\quad\forall i\in[N],\label{pp2}\\
  & \sum_{r(t)\in R(t)}\lambda_{r(t)}o_{r(t)}^i\geq\sum_{l(i)\in L(i)}\gamma_{l(i)}\beta_{l(i)}^t \quad\forall i\in[N],t\in[T],\label{pp3}\\
  & \sum_{i\in[N]}\sum_{l(i)\in L(i)}\gamma_{l(i)}\beta_{l(i)}^t\bigl(s_{l(i)}^{t}-s_{l(i)}^{t^-}+\sum_{t^{\prime}\in[t^-,t-1]}U_{t^{\prime}}^i\bigr)\leq Q \quad \forall t\in [T],\label{pp4}\\ 
  & \lambda_{r(t)} \in \{0,1\} \quad \forall t\in[T], r(t)\in R(t),\\
  & \gamma_{l(i)} \in \{0,1\} \quad \forall i\in[N], l(i)\in L(i).
\end{align}

Let $\iota_t$, $\theta_i$, $\delta_{it}$, and $\psi_t$ be the dual variables associated with Constraints~\eqref{pp1}-\eqref{pp4}, respectively.

\textbf{Second Level: Routing Plan Generation.} 
The pricing problem for generating routing plan $r(t)$ in period $t$ is given by:
\begin{align}
   \min\ & \sum_{(i,j)\in A^{\rm spatial}}\bigl(\frac{1}{T}\rho c_{ij}-\delta_{it}\bigr)x_{ij}^t+\iota_t,\nonumber\\
    s.t.\ 
  & x_{ijt}\in \mathscr{X}\quad\forall (i,j)\in A^{\rm spatial}\nonumber.
\end{align}
This is an ESPPRC, which can be solved using the labeling algorithm described in Online Appendix~\ref{oa-labeling}.

\textbf{Second Level: Replenishment Plan Generation.} 
The pricing problem for the replenishment plan of retailer~$i$ is identical to the 2nd-PP2($i$) formulation in Section~4.1.4 and is omitted here for brevity.

\subsubsection{Inventory Policy Solving.}
The inventory policy is solved in the same manner as in Section~4.2 under the consistent policy, and we therefore omit repetition of the procedure here.

\subsubsection{Branching Strategies}
Similar to the consistent policy, let $q_{\omega}^{*}$ denote the value of $q_{\omega}$ in the LP relaxation of the first-layer restricted master problem. Under the flexible policy, routing patterns vary across periods, rendering arc-based branching (as used in the consistent policy) ineffective. Consequently, when the solution is fractional, we identify the pattern with $q_{\omega}^{*}$ closest to 0.6 and perform pattern-based branching. Specifically, we create two child nodes: one where the selected pattern is enforced and another where it is excluded.

To prevent regeneration of a forbidden pattern in the corresponding child node, we incorporate additional constraints into the second-layer restricted master problem (2nd-RMP). For each forbidden pattern $\mathcal{M}\in \mathscr{M}$, we introduce the following constraint:
\begin{align}
\sum_{i \in [N] \setminus \mathcal{M}} \sum_{l(i) \in L(i)} \gamma_{l(i)} - \sum_{i \in \mathcal{M}} \sum_{l(i) \in L(i)} \gamma_{l(i)} \geq -|\mathcal{M}| + 1, \nonumber
\end{align}
where $|\mathcal{M}|$ is the number of retailers in pattern $\mathcal{M}$.

It is important to note that these constraints alter the structure of the second-layer pricing problem. Specifically, the reduced cost associated with replenishment plans must be revised to incorporate the dual penalties corresponding to the added constraints. This adjustment ensures accurate evaluation of candidate plans during column generation.

Finally, the second-layer branch-and-price algorithm proceeds by branching on replenishment plans using three sequential refinement steps, as detailed in Section~4.3.

\subsubsection{Initial Solution Construction}
Since any feasible solution under the consistent policy remains feasible under the flexible policy, we construct the initial solution using the same procedure developed for the consistent policy.

\clearpage

\setcounter{equation}{0}
\renewcommand{\theequation}{C\arabic{equation}}
\setcounter{table}{0} 
\setcounter{figure}{0}
\renewcommand{\thetable}{C\arabic{table}}
\renewcommand{\thefigure}{C\arabic{figure}}
\setcounter{algorithm}{0}
\renewcommand{\thealgorithm}{C\arabic{algorithm}}
\section{Labeling Algorithms} \label{oa-labeling}
\subsection{Labeling Algorithm for the Fixed-Interval Policy}
The pricing problem for the fixed-interval policy is described in Section B.1.1 of Online Appendix~\ref{OA-policies}. For a given value of \(\kappa_v\) and precomputed arc costs \(\tilde{c}_{ij}\) for all 
\((i,j)\in A^{\rm spatial}\), this problem can be solved using a labeling algorithm.

Each label $\tau$ is characterized by five elements: \(\vec{V}(\tau)\), \(U(\tau)\), \(Q(\tau)\), \(D(\tau)\), and \(C(\tau)\), which respectively represent the current partial route, the set of unreachable retailers (i.e., those already visited or infeasible to visit), the accumulated vehicle capacity usage, the total distance traveled, and the cumulative objective value (comprising transportation and inventory costs adjusted by dual variables). Label initialization begins with a visit to the warehouse, given by \(\vec{V}(\tau) = (0), U(\tau) = \emptyset\), and \(Q(\tau) = D(\tau) = C(\tau) = 0.\)

The resource-extension functions update the label \(\tau\) when a retailer \(j \in [N]\) is appended to the current route. Let the most recently visited retailer in $\vec{V}(\tau)$ be denoted by $i$. The updated label \(\tau^{\prime}\) is then computed as follows:
\begin{align}
    \vec{V}(\tau^{\prime}) &= \vec{V}(\tau)\cup (j), \nonumber \\
    U(\tau^{\prime}) &= U(\tau) \cup \{j\}\cup\bigl\{k\big|Q(\tau) + \kappa_v U^j+\kappa_v U^k> Q\ \forall k\in[N]\bigr\}\cup\nonumber\\
    &\quad\bigl\{k\big|D(\tau) + c_{ij}+c_{jk}>L\ \forall k\in[N]\bigr\}, \nonumber \\
    Q(\tau^{\prime}) &= Q(\tau) + \kappa_v U^j, \nonumber \\
    D(\tau^{\prime}) &= D(\tau) + c_{ij}, \nonumber \\
    C(\tau^{\prime}) &= C(\tau) + \tilde{c}_{ij}. \nonumber
\end{align}

Labels are recursively extended to explore all feasible routes. To enhance computational efficiency, a dominance rule is employed to prune inferior labels. Specifically, a label \(\tau\) is said to be dominated by another label \(\tau^{\prime}\) if both labels terminate at the same retailer and the following conditions are satisfied:
\begin{align}
    U(\tau^{\prime}) \subseteq U(\tau), \quad Q(\tau^{\prime}) \leq Q(\tau), \quad D(\tau^{\prime}) \leq D(\tau), \quad C(\tau^{\prime}) \leq C(\tau). \nonumber
\end{align}

The labeling algorithm functions as a dynamic programming procedure, where each label represents a system state and label extensions define transitions. The dominance rule effectively prunes the search space by retaining only the most promising partial solutions. The algorithm terminates when no further feasible extensions can be made, ensuring that the pricing problem is solved efficiently under the given resource constraints.

\subsection{Labeling Algorithm for the Consistent Policy}
Under the consistent policy, the nested branch-and-price framework solves two pricing problems: one for generating routing plans and another for generating replenishment plans. We describe below the labeling algorithms used to solve each of these subproblems.

\subsubsection{Routing Plan.} 
In the routing pricing problem, a label \(\tau\) is defined by four components: \(\vec{V}(\tau)\), \(U(\tau)\), \(D(\tau)\), and \(C(\tau)\), representing the current partial route, the set of unreachable or already visited retailers, the cumulative travel distance, and the objective value composed of transportation cost minus dual contributions, respectively. The label is initialized with a single warehouse visit: \(\vec{V}(\tau) = (0), U(\tau) = \emptyset\), and \(D(\tau) = C(\tau) = 0.\)

When a retailer \(j \in [N]\) is appended to the current route, the label is extended. Let $i$ denote the last retailer visited in $\vec{V}(\tau)$. The updated label \(\tau^{\prime}\) is computed as follows:
\begin{align}
    \vec{V}(\tau^{\prime}) &= \vec{V}(\tau)\cup (j), \nonumber \\
    U(\tau^{\prime}) &= U(\tau) \cup \{j\}\cup\bigl\{k\big|D(\tau) + c_{ij}+c_{jk}>L\ \forall k\in[N]\bigr\}, \nonumber \\
    D(\tau^{\prime}) &= D(\tau) + c_{ij}, \nonumber \\
    C(\tau^{\prime}) &= C(\tau) + \sum_{t \in \hat{\mathcal{T}}} \bigl( \frac{1}{T} \rho c_{ij} - \delta_{it} \bigr). \nonumber
\end{align}

To improve computational efficiency, we apply a dominance rule: a label \(\tau\) is dominated by another label \(\tau^{\prime}\) if \(\tau^{\prime}\) terminates at the same retailer and satisfies the following conditions:
\begin{align}
    U(\tau^{\prime}) \subseteq U(\tau), \quad D(\tau^{\prime}) \leq D(\tau), \quad C(\tau^{\prime}) \leq C(\tau). \nonumber
\end{align}

\subsubsection{Replenishment Plan.} 
For any given retailer \(i \in [N]\) and a specified cycle end period \(t^e \in [T]\), with precomputed inventory costs \(c^i_{t_1t_2}\) for all \((t_1,t_2) \in A_t\), we construct labels defined by three components: \(\vec{V}(\tau)\), \(U(\tau)\), and \(C(\tau)\), representing the current replenishment schedule (as a sequence of periods), unreachable or already covered periods, and the cumulative objective value comprising inventory cost minus the incurred dual costs, respectively. The label is initialized as:\(\vec{V}(\tau) = (t^e),  U(\tau) = \{t|t> t^e\ \forall t\in[T]\}\), and \(C(\tau) = 0.\)

To extend the label \(\tau\) by appending a new period \(t^a\in [T]\), let $t^l$ denote the most recently visited period in $\vec{V}(\tau)$. The updated label \(\tau^{\prime}\) is computed as:
\begin{align}
    \vec{V}(\tau^{\prime}) &= \vec{V}(\tau)\cup(t^a), \nonumber \\
    U(\tau^{\prime}) &= U(\tau) \cup \bigl\{t\big|t\leq t^a\ \forall t\in[T]\bigr\}, \nonumber \\
    C(\tau^{\prime}) &= C(\tau) + c^i_{t^lt^a}. \nonumber
\end{align}

A dominance rule is also applied: label \(\tau\) is dominated by label \(\tau^{\prime}\) if both terminate at the same period and the following conditions are satisfied:
\begin{align}
    U(\tau^{\prime}) \subseteq U(\tau), \quad C(\tau^{\prime}) \leq C(\tau). \nonumber
\end{align}

\subsection{Labeling Algorithm for the Flexible Policy}
This subsection outlines the labeling algorithm used to solve the pricing problem under the flexible policy. As the pricing problem for generating replenishment plans is structurally identical to that under the consistent policy, we omit its description here and focus exclusively on the labeling algorithm for constructing routing plans.

For each period $t\in [T]$, we define a label \(\tau\) by four elements: \(\vec{V}(\tau)\), \(U(\tau)\), \(D(\tau)\), and \(C(\tau)\). , representing the current partial route, the set of unreachable retailers (i.e., those already visited or excluded due to resource constraints), the cumulative travel distance, and the reduced-cost value (i.e., the transportation cost minus the accumulated dual values), respectively. The initial label corresponds to a visit to the warehouse only, defined as: \(\vec{V}(\tau) = (0), U(\tau) = \emptyset\), and \(D(\tau) = C(\tau) = 0.\)

Label extension is performed by appending a retailer $j\in[N]$ to the current route. Let $i$ denote the last visited retailer in $\vec{V}(\tau)$. The extended label \(\tau^{\prime}\) is computed as follows:
\begin{align}
    \vec{V}(\tau^{\prime}) &= \vec{V}(\tau)\cup(j), \nonumber \\
    U(\tau^{\prime}) &= U(\tau) \cup \{j\}\cup\bigl\{k\big|D(\tau) + c_{ij}+c_{jk}>L\ \forall k\in[N]\bigr\}, \nonumber \\
    D(\tau^{\prime}) &= D(\tau) + c_{ij}, \nonumber \\
    C(\tau^{\prime}) &= C(\tau) + \frac{1}{T} \rho c_{ij} - \delta_{it}. \nonumber
\end{align}

To maintain computational tractability, we employ a dominance rule to discard non-promising labels. Specifically, a label \(\tau\) is dominated by another label \(\tau^{\prime}\) if both terminate at the same retailer and the following conditions hold:
\begin{align}
    U(\tau^{\prime}) \subseteq U(\tau), \quad D(\tau^{\prime}) \leq D(\tau), \quad C(\tau^{\prime}) \leq C(\tau). \nonumber
\end{align}
This dominance criterion ensures that only the most efficient partial paths are retained during label propagation, thereby significantly reducing the number of labels considered.

\clearpage

\setcounter{equation}{0}
\renewcommand{\theequation}{D\arabic{equation}}
\setcounter{table}{0} 
\setcounter{figure}{0}
\renewcommand{\thetable}{D\arabic{table}}
\renewcommand{\thefigure}{D\arabic{figure}}
\setcounter{algorithm}{0}
\renewcommand{\thealgorithm}{D\arabic{algorithm}}
\section{Proofs} \label{oa-proofs}
\subsection{Proof of Proposition 1.}
The expression for the worst-case VaR follows directly from Proposition 2 of \citet{ghosal2020distributionally}. The result for the worst-case LQ is its lower-tail analogue. The proof proceeds mutatis mutandis from their argument, with $\sup$ replaced by $\inf$, $1-\varepsilon$ replaced by $\varepsilon$, inequalities reversed as appropriate, and the corresponding primal/dual signs adjusted. Under these substitutions, their derivation carries through step by step, yielding the claimed closed-form expression.
\hfill\qed

\subsection{Proof of Proposition 2.}
Define the auxiliary functions:
\begin{align}
    f_0(s_{it_1}, \boldsymbol{\zeta}) := s_{it_1} - \sum_{t^{\prime} \in [t_1, t]} \zeta_{t^{\prime}}^i, \quad
    f_1(s_{it_1}, \boldsymbol{\zeta}) := \sum_{t^{\prime} \in [t_1, t]} \zeta_{t^{\prime}}^i - s_{it_1}. \nonumber
\end{align}
Then the function can be written as:
\begin{align}
    f^{inv}(i,t_1,t_2,s_{it_1}) = \sup_{\mathbb{P} \in \mathcal{P}} \mathbb{E}_{\mathbb{P}} \Bigl[ h \sum_{t \in [t_1, t_2 - 1]} \bigl( f_0(s_{it_1}, \boldsymbol{\zeta}) \bigr)^+ + b \sum_{t \in [t_1, t_2 - 1]} \bigl( f_1(s_{it_1}, \boldsymbol{\zeta}) \bigr)^+ \Bigr]. \nonumber
\end{align}

Note that \( f_0(s_{it_1}, \boldsymbol{\zeta})\) and \( f_1(s_{it_1}, \boldsymbol{\zeta}) \) are affine in \( s_{it_1} \), and the operations of taking the nonnegative part (i.e., \( (\cdot)^+ \)), summation, expectation, and pointwise supremum over convex functions all preserve convexity \citep[Chapter~3.2]{boyd2004convex}. Therefore, \( f^{inv}(i,t_1,t_2,s_{it_1}) \) is convex in \( s_{it_1} \). \hfill\qed

\subsection{Proof of Theorem 1.}
The worst-case expected holding and backorder cost over this interval can be expressed as:
\begin{align}
    \sup_{\mathbb{P}\in\mathcal{P}}\mathbb{E}_{\mathbb{P}}\Bigl[h\sum_{t\in [t_1,t_2-1]}\bigl(s_{it_1}-\sum_{t^{\prime}\in [t_1,t]} \zeta_{t^{\prime}}^i\bigr)^++b\sum_{t\in [t_1,t_2-1]}\bigl(s_{it_1}-\sum_{t^{\prime}\in [t_1,t]} \zeta_{t^{\prime}}^i\bigr)^-\Bigr].\nonumber
\end{align} 
This is equivalent to the following moment problem:
\begin{align}
\max \ & \int_{[\boldsymbol{\underline{\zeta}},\boldsymbol{\bar{\zeta}}]}\Bigl(h\sum_{t\in [t_1,t_2-1]}\bigl(s-\sum_{t^{\prime}\in [t_1,t]} \zeta_{t^{\prime}}^i\bigr)^++b\sum_{t\in [t_1,t_2-1]}\bigl(s-\sum_{t^{\prime}\in [t_1,t]} \zeta_{t^{\prime}}^i\bigr)^-\Bigr)\mathbb{P}(\mathrm{d}\boldsymbol{\zeta}),\nonumber\\
s.t.\  & \int_{[\boldsymbol{\underline{\zeta}},\boldsymbol{\bar{\zeta}}]}\mathbb{P}(\mathrm{d}\boldsymbol{\zeta})=1,\nonumber\\
& \int_{[\boldsymbol{\underline{\zeta}},\boldsymbol{\bar{\zeta}}]}\boldsymbol{\zeta}\mathbb{P}(\mathrm{d}\boldsymbol{\zeta})=\boldsymbol{\mu},\nonumber\\
& \int_{[\boldsymbol{\underline{\zeta}},\boldsymbol{\bar{\zeta}}]}|\boldsymbol{\zeta}-\boldsymbol{\mu}|\mathbb{P}(\mathrm{d}\boldsymbol{\zeta})\leq \boldsymbol{\sigma},\nonumber\\
&\mathbb{P}\in \mathcal{M}_+\bigl([\boldsymbol{\underline{\zeta}},\boldsymbol{\bar{\zeta}}]\bigr),\nonumber
\end{align}
where we drop the retailer and time indices for clarity and denote the random demand vector over the interval as $\boldsymbol{\zeta}\subseteq \mathbb{R}^{\hat{T}}$, with $\hat{T} = t_2 - t_1$ if $t_2>t_1$, and $\hat{T}=t_2+T-t_1$ otherwise.

By Lemma EC.1 of \citet{ghosal2020distributionally}, strong duality holds between this primal formulation and the following semi-infinite dual:
\begin{align}
\min \ & \alpha+\boldsymbol{\mu}^\top\boldsymbol{\beta}+\boldsymbol{\sigma}^\top\boldsymbol{\gamma},\nonumber\\
s.t.\ &\alpha+\boldsymbol{\zeta}^\top\boldsymbol{\beta}+|\boldsymbol{\zeta}-\boldsymbol{\mu}|^\top\boldsymbol{\gamma}\geq h\sum_{t\in [t_1,t_2-1]}\bigl(s-\sum_{t^{\prime}\in [t_1,t]} \zeta_{t^{\prime}}\bigr)^++b\sum_{t\in [t_1,t_2-1]}\bigl(s-\sum_{t^{\prime}\in [t_1,t]} \zeta_{t^{\prime}}\bigr)^- \quad\forall \boldsymbol{\zeta}\in [\boldsymbol{\underline{\zeta}},\boldsymbol{\bar{\zeta}}],\nonumber\\
& \alpha\in\mathbb{R}, \ \boldsymbol{\beta}\in\mathbb{R}^{\hat{T}}, \ \boldsymbol{\gamma}\in \mathbb{R}^{\hat{T}}_+.\nonumber
\end{align}

By splitting the semi-infinite constraints, we obtain the following optimization problem:
\begin{align}
{\rm (S)}\min \ & \alpha+\boldsymbol{\mu}^\top\boldsymbol{\beta}+\boldsymbol{\sigma}^\top\boldsymbol{\gamma},\nonumber\\
s.t.\ &\alpha+\boldsymbol{\zeta}^\top\boldsymbol{\beta}+|\boldsymbol{\zeta}-\boldsymbol{\mu}|^\top\boldsymbol{\gamma}\geq h\sum_{t\in [t_1,\hat{t})}\bigl(s-\sum_{t^{\prime}\in [t_1,t]} \zeta_{t^{\prime}}\bigr)+b\sum_{t\in [\hat{t},t_2-1]}\bigl(\sum_{t^{\prime}\in [t_1,t]} \zeta_{t^{\prime}}-s\bigr) \nonumber\\
&\qquad\qquad\qquad\qquad\qquad\forall \boldsymbol{\zeta}\in [\boldsymbol{\underline{\zeta}},\boldsymbol{\bar{\zeta}}]: \sum_{t\in[t_1,\hat{t})}\zeta_t^i<s\leq \sum_{t\in[t_1,\hat{t}]}\zeta_t^i, \ \hat{t}\in[t_1,t_2-1], \label{eq3-1}\\
&\alpha+\boldsymbol{\zeta}^\top\boldsymbol{\beta}+|\boldsymbol{\zeta}-\boldsymbol{\mu}|^\top\boldsymbol{\gamma}\geq h\sum_{t\in [t_1,t_2-1]}\bigl(s-\sum_{t^{\prime}\in [t_1,t]} \zeta_{t^{\prime}}\bigr) \quad \forall \boldsymbol{\zeta}\in [\boldsymbol{\underline{\zeta}},\boldsymbol{\bar{\zeta}}]: \sum_{t\in[t_1,t_2-1]}\zeta_t^i<s, \label{eq3-that0}\\
& \alpha\in\mathbb{R}, \ \boldsymbol{\beta}\in\mathbb{R}^{\hat{T}}, \ \boldsymbol{\gamma}\in \mathbb{R}^{\hat{T}}_+.\nonumber
\end{align}

For a given $\hat{t}\in[t_1,t_2-1]$, Constraints \eqref{eq3-1} is equivalent to the following constraint:
\begin{align}
\left[
\begin{aligned}
   {\rm (P)} \inf \ &\boldsymbol{\zeta}^\top\boldsymbol{\beta}+|\boldsymbol{\zeta}-\boldsymbol{\mu}|^\top\boldsymbol{\gamma}+\\
   &h\sum_{t\in [t_1,\hat{t})}\sum_{t^{\prime}\in [t_1,t]} \zeta_{t^{\prime}}-b\sum_{t\in [\hat{t},t_2-1]}\sum_{t^{\prime}\in [t_1,t]} \zeta_{t^{\prime}},\\
    s.t. \ &\sum_{t\in[t_1,\hat{t})}\zeta_t< s, \\
    &\sum_{t\in[t_1,\hat{t}]}\zeta_t\geq s,\\
    & \boldsymbol{\zeta}\in [\boldsymbol{\underline{\zeta}},\boldsymbol{\bar{\zeta}}].
\end{aligned}
\right]\geq h(\hat{t}-t_1)s-b\Bigl(\bigl((t_2-1)-\hat{t}\bigr)+1\Bigr)s-\alpha.\label{eq3-2}
\end{align}

Assume first that $\sum_{t\in[t_1,\hat{t}]}\underline{\zeta}_t\neq s$. In this case, the constraint set $\bigl\{\boldsymbol{\zeta}\in [\boldsymbol{\underline{\zeta}},\boldsymbol{\bar{\zeta}}]: \sum_{t\in[t_1,\hat{t})}\zeta_t<s\bigr\}$ is nonempty if and only if its closure $\bigl\{\boldsymbol{\zeta}\in [\boldsymbol{\underline{\zeta}},\boldsymbol{\bar{\zeta}}]: \sum_{t\in[t_1,\hat{t})}\zeta_t\leq s\bigr\}$ is nonempty. Thus, the problem (P) can be reformulated as 
\begin{align}
   {\rm (P)} \min \ &\boldsymbol{\zeta}^\top\boldsymbol{\beta}+|\boldsymbol{\zeta}-\boldsymbol{\mu}|^\top\boldsymbol{\gamma}+h\sum_{t\in [t_1,\hat{t})}\sum_{t^{\prime}\in [t_1,t]} \zeta_{t^{\prime}}-b\sum_{t\in [\hat{t},t_2-1]}\sum_{t^{\prime}\in [t_1,t]} \zeta_{t^{\prime}},\nonumber\\
    s.t. \ &\sum_{t\in[t_1,\hat{t})}\zeta_t\leq s, \nonumber\\
    &\sum_{t\in[t_1,\hat{t}]}\zeta_t\geq s,\nonumber\\
    & \boldsymbol{\zeta}\in [\boldsymbol{\underline{\zeta}},\boldsymbol{\bar{\zeta}}].\nonumber
\end{align}

We analyze by considering three mutually exclusive and collectively exhaustive cases:
\begin{itemize}
    \item[($i$)] There exists $\boldsymbol{\zeta}\in (\boldsymbol{\underline{\zeta}},\boldsymbol{\bar{\zeta}})$ such that $\sum_{t\in[t_1,\hat{t})}\zeta_t<s$ and $\sum_{t\in[t_1,\hat{t}]}\zeta_t> s$.
    \item[($ii$)] No $\boldsymbol{\zeta}\in [\boldsymbol{\underline{\zeta}},\boldsymbol{\bar{\zeta}}]$ satisfies $\sum_{t\in[t_1,\hat{t})}\zeta_t\leq s$ and $\sum_{t\in[t_1,\hat{t}]}\zeta_t\geq s$.
    \item[($iii$)] There exists $\boldsymbol{\zeta}\in [\boldsymbol{\underline{\zeta}},\boldsymbol{\bar{\zeta}}]$ satisfying $\sum_{t\in[t_1,\hat{t})}\zeta_t\leq s$ and \(\sum_{t\in[t_1,\hat{t}]}\zeta_t\geq s\), but no such \(\boldsymbol{\zeta} \in (\boldsymbol{\underline{\zeta}},\boldsymbol{\bar{\zeta}})\) satisfies both \(\sum_{t \in [t_1, \hat{t})} \zeta_t < s\) and \(\sum_{t \in [t_1, \hat{t}]} \zeta_t > s\).
\end{itemize}

In Case ($i$), strong duality holds by Slater’s condition, and Constraint \eqref{eq3-2} is satisfied if and only if the optimal value of the corresponding dual problem (D) exceeds the right-hand side. The dual takes the form:
\begin{align}
{\rm (D)}\ \max \ & -s\eta_0 + s\eta_1 +\underline{\boldsymbol{\zeta}}^\top\underline{\boldsymbol{\nu}}-\bar{\boldsymbol{\zeta}}^\top\bar{\boldsymbol{\nu}}-\boldsymbol{\mu}^\top\boldsymbol{\pi}^++\boldsymbol{\mu}^\top\boldsymbol{\pi}^-,\nonumber\\
s.t. \ & \beta_t=-h(\hat{t}-t)^++b\Bigl(\bigl((t_2-1)-\hat{t}\bigr)+1\Bigr)-b(t-\hat{t})^+-\eta_0\mathbbm{1}_{t<\hat{t}}+\eta_1\mathbbm{1}_{t\leq\hat{t}}+\underline{\nu}_t-\bar{\nu}_t-\pi^+_t+\pi^-_t \nonumber\\
&\qquad\qquad\qquad\qquad\qquad\qquad\qquad\qquad\qquad\qquad\qquad\qquad\qquad\quad \forall t\in[t_1,t_2-1],\nonumber\\
&\boldsymbol{\pi}^++\boldsymbol{\pi}^-\leq \boldsymbol{\gamma},\nonumber\\
& \underline{\boldsymbol{\nu}},\bar{\boldsymbol{\nu}},\boldsymbol{\pi}^+,\boldsymbol{\pi}^-\in\mathbb{R}_+^{\hat{T}},\ \eta_0,\eta_1\in\mathbb{R}_+,\nonumber
\end{align}
where the dependence of the variables on $\hat{t}$ is omitted for notational simplicity.

In Case ($ii$), define vectors $\boldsymbol{\mathrm{e}}$ and $\boldsymbol{0}$ as the all-ones and all-zeros vectors, respectively. Let $\boldsymbol{\mathrm{e}}_{[t_a, t_b]}$ denote the indicator vector such that $\boldsymbol{\mathrm{e}}_{[t_a, t_b]}[t]=\mathbbm{1}_{t\in[t_a,t_b]}$ for all $t\in [t_1,t_2-1]$. In this case, two mutually exclusive situations can arise:
\begin{itemize}
    \item[(a)] $\sum_{t\in[t_1,\hat{t})}\zeta_t> s$; or
    \item[(b)] $\sum_{t\in[t_1,\hat{t}]}\zeta_t< s$.
\end{itemize}

Fix $\eta_0, \eta_1\in\mathbb{R}_+$, and define the dual variables as follows: $\underline{\nu}_t=\biggl(\beta_t+h(\hat{t}-t)^+-b\Bigl(\bigl((t_2-1)-\hat{t}\bigr)+1\Bigr)+b(t-\hat{t})^++\eta_0\mathbbm{1}_{t<\hat{t}}\biggr)$, $\bar{\boldsymbol{\nu}}=\eta_1\boldsymbol{\mathrm{e}}_{[t_1,\hat{t}]}$, and $\boldsymbol{\pi}^+=\boldsymbol{\pi}^-=\boldsymbol{0}$. This choice is feasible for the dual problem (D), and yields the following objective value:
\begin{align}
    &\quad -s\eta_0 +s\eta_1+\sum_{t\in[t_1,t_2-1]}\underline{\zeta}_t\biggl(\beta_t+h(\hat{t}-t)^+-b\Bigl(\bigl((t_2-1)-\hat{t}\bigr)+1\Bigr)+b(t-\hat{t})^++\eta_0\mathbbm{1}_{t<\hat{t}}\biggr)-\bar{\boldsymbol{\zeta}}^\top\eta_1\boldsymbol{\mathrm{e}}_{[t_1,\hat{t}]}\nonumber\\
    &=\eta_0\Biggl(-s+\sum_{t\in[t_1,t_2-1]}\underline{\zeta}_t\biggl(\frac{\beta_t+h(\hat{t}-t)^+-b\Bigl(\bigl((t_2-1)-\hat{t}\bigr)+1\Bigr)+b(t-\hat{t})^+}{\eta_0}+\mathbbm{1}_{t<\hat{t}}\biggr)\Biggr)+\eta_1(s-\bar{\boldsymbol{\zeta}}^\top\boldsymbol{\mathrm{e}}_{[t_1,\hat{t}]})\nonumber.
\end{align}
This expression tends to $+\infty$ under the following conditions:
\begin{itemize}
    \item[(a)] Let $\eta_0\rightarrow +\infty$ and $\eta_1=0$. Then the summation converges to $\underline{\boldsymbol{\zeta}}^\top \boldsymbol{\mathrm{e}}_{[t_1,\hat{t})}$, which exceeds $s$ by assumption, and the entire expression diverges to $+\infty$.
    \item[(b)] Let $\eta_0=0$ and $\eta_1\rightarrow +\infty$. Since $\sum_{t\in[t_1,\hat{t}]}\zeta_t< s$ by assumption, the expression again diverges to $+\infty$.
\end{itemize}

In Case (iii), we consider the perturbed formulations $({\rm P}_{\varepsilon})$ and $({\rm D}_{\varepsilon})$, defined by expanding the support set to $[\underline{\boldsymbol{\zeta}}-\varepsilon\boldsymbol{\mathrm{e}}, \bar{\boldsymbol{\zeta}}+\varepsilon\boldsymbol{\mathrm{e}}]$. For all $\varepsilon> 0$, strong duality holds. The mapping $\varepsilon\mapsto ({\rm P}_{\varepsilon})$ is right-continuous at $\varepsilon= 0$, and since the value of $({\rm D}_{\varepsilon})$ is non-decreasing and bounded below by that of (D), we conclude that (P) and (D) have equal optimal values.

It is straightforward to verify that Constraints \eqref{eq3-that0} yield the same result when setting $\eta_1=0$. We refer to this condition as the case where $\hat{t}=0$.

Hence, for all $s\in \mathbb{R}_+, \sum_{t\in[t_1,\hat{t}]}\zeta_t\neq s$, the original worst-case expected inventory cost expression $\sup_{\mathbb{P}\in\mathcal{P}}\mathbb{E}_{\mathbb{P}}\bigl[h\sum_{t\in [t_1,t_2-1]}(s-\sum_{t^{\prime}\in [t_1,t]} \zeta_{t^{\prime}})^++b\sum_{t\in [t_1,t_2-1]}(s-\sum_{t^{\prime}\in [t_1,t]} \zeta_{t^{\prime}})^-\bigr]$ is equivalent to the linear formulation (F), where the constraints are given by a family of dual problems indexed by $\hat{t}\in[t_1,t_2-1]\cup \{0\}$.

If $\sum_{t\in[t_1,\hat{t}]}\underline{\zeta}_t= s$, continuity arguments must be invoked. We define the feasible set mapping $\mathcal{S}(s)$ as the set of all variables satisfying the constraint system in (F) at a given $s$. The graph of $\mathcal{S}$ is closed since the constraints are convex and continuous, implying that $\mathcal{S}(s)$ is outer semicontinuous. Therefore, the strong duality conclusion extends to this boundary case.

Finally, since problem (F) admits a Slater point, we apply strong duality to obtain the equivalent dual formulation (O):
\begin{align}
  {\rm (O)}
  \max\ & \sum_{\hat{t}\in[t_1,t_2-1]}\Bigl(h\sum_{t\in[t_1,\hat{t})}\bigl(\pi_{\hat{t}}s_{it_1}-\sum_{t^{\prime}\in [t_1,t]}\lambda_{\hat{t}t^{\prime}}\bigr)+b\sum_{t\in[\hat{t},t_2-1]}\bigl(\sum_{t^{\prime}\in [t_1,t]}\lambda_{\hat{t}t^{\prime}}-\pi_{\hat{t}}s_{it_1}\bigr)\Bigr)+\nonumber\\
  &
  h\sum_{t\in[t_1,t_2-1]}\bigl(\pi_{0}s_{it_1}-\sum_{t^{\prime}\in [t_1,t]}\lambda_{0t^{\prime}}\bigr)\nonumber\\
    s.t.\ &\sum_{\hat{t}\in[t_1,t_2-1]\cup \{0\}}\pi_{\hat{t}}=1,\nonumber\\
    &\sum_{\hat{t}\in[t_1,t_2-1]\cup \{0\}}\lambda_{\hat{t}t}=\mu_{t}^i\quad\forall t\in [t_1,t_2-1],\nonumber\\
    & \sum_{\hat{t}\in[t_1,t_2-1]\cup \{0\}} |\lambda_{\hat{t}t}-\pi_{\hat{t}}\mu_{t}^i|\leq \sigma_{t}^i\quad\forall t\in [t_1,t_2-1],\nonumber\\
    & \pi_{\hat{t}}\underline{\zeta}_{t}^i\leq \lambda_{\hat{t}t}\leq \pi_{\hat{t}}\bar{\zeta}_{t}^i\quad \forall t\in [t_1,t_2-1],\hat{t}\in[t_1,t_2-1]\cup \{0\},\nonumber\\
    &\sum_{t\in[t_1,\hat{t})}\lambda_{\hat{t}t}\leq \pi_{\hat{t}}s_{it_1}\leq \sum_{t\in[t_1,\hat{t}]}\lambda_{\hat{t}t} \quad\forall \hat{t}\in[t_1,t_2-1],\nonumber\\
    &\sum_{t\in[t_1,t_2-1]}\lambda_{0t}\leq \pi_{0}s_{it_1},\nonumber\\
    &\pi_{\hat{t}}\in\mathbb{R}_+, \boldsymbol{\lambda}_{\hat{t}}\in\mathbb{R}^{\hat{T}}\ \forall \hat{t}\in[t_1,t_2-1]\cup \{0\}.\nonumber
\end{align}

We conclude that the worst-case expected inventory cost is equivalently given by the linear program (O). \hfill\qed

\subsection{Proof of Corollary 1.}
We construct a discrete distribution $\mathbb{P}^*$ supported on at most $\hat{T}+1$ points as follows:
\begin{align}
    \mathbb{P}^*=\sum_{\hat{t}\in[t_1,t_2-1],\ \pi^{*}_{\hat{t}}\neq 0}\pi_{\hat{t}}^{*}\delta_{\frac{\boldsymbol{\lambda}_{\hat{t}}^{*}}{\pi_{\hat{t}}^{*}}}+{\mathbbm{1}_{\pi^*_0 \neq 0}}\pi_{0}^{*}\delta_{\frac{\boldsymbol{\lambda}_{0}^{*}}{\pi_{0}^{*}}}, \nonumber
\end{align}
where $(\pi^*_{\hat{t}}, \boldsymbol{\lambda}^*_{\hat{t}})$ for all $\hat{t} \in [t_1, t_2-1] \cup \{0\}$ denotes the optimal solution to Formulation~(O) in the proof of Theorem~1, and $\delta_{\boldsymbol{\zeta}}$ denotes the Dirac measure concentrated at $\boldsymbol{\zeta}$.

We claim that:
\begin{itemize}
    \item[($i$)] $\mathbb{P}^* \in \mathcal{P}$; and  
    \item[($ii$)] $\mathbb{P}^*$ achieves the worst-case expected inventory cost, i.e.,
\begin{align}
    \mathbb{E}_{\mathbb{P}^*}\bigl[f(\boldsymbol{\zeta})\bigr] = \sup_{\mathbb{P} \in \mathcal{P}} \mathbb{E}_{\mathbb{P}}\bigl[f(\boldsymbol{\zeta})\bigr], \nonumber
\end{align}
where the cost function is defined as
\begin{align}
    f(\boldsymbol{\zeta}) := h \sum_{t \in [t_1, t_2 - 1]} \bigl(s - \sum_{t' \in [t_1, t]} \zeta_{t'}\bigr)^+ + b \sum_{t \in [t_1, t_2 - 1]} \bigl(s - \sum_{t' \in [t_1, t]} \zeta_{t'}\bigr)^-. \nonumber
\end{align}
\end{itemize}

\textit{Proof of ($i$).} By the first constraint of Formulation~(O), the weights sum to one: $\sum_{\hat{t}\in[t_1,t_2-1]\cup \{0\}}\pi_{\hat{t}}^*=1$, so $\mathbb{P}^*$ is a valid probability distribution. By the fourth constraint, each atom satisfies $\boldsymbol{\lambda}^*_{\hat{t}} / \pi^*_{\hat{t}} \in [\underline{\boldsymbol{\zeta}}, \bar{\boldsymbol{\zeta}}]$, ensuring $\operatorname{supp}(\mathbb{P}^*) \subseteq [\boldsymbol{\underline{\zeta}},\boldsymbol{\bar{\zeta}}]$. The moment condition follows directly from:
\begin{align}
    \mathbb{E}_{\mathbb{P}^*}[\boldsymbol{\zeta}] &= \sum_{\hat{t}\in[t_1,t_2-1]} \pi^*_{\hat{t}} \frac{\boldsymbol{\lambda}^*_{\hat{t}}}{\pi^*_{\hat{t}}}+\pi_{0}^{*}\frac{\boldsymbol{\lambda}_{0}^{*}}{\pi_{0}^{*}} = \sum_{\hat{t}\in[t_1,t_2-1]\cup\{0\}} \boldsymbol{\lambda}^*_{\hat{t}} = \boldsymbol{\mu}, \nonumber
\end{align}
by the second constraints. Likewise, the deviation constraint satisfies:
\begin{align}
    \mathbb{E}_{\mathbb{P}^*}\bigl[\boldsymbol{|\boldsymbol{\zeta}-\boldsymbol{\mu}|}\bigr]=&\sum_{\hat{t}\in[t_1,t_2-1]}\pi_{\hat{t}}^{*}\bigl|\frac{\boldsymbol{\lambda}_{\hat{t}}^{*}}{\pi_{\hat{t}}^{*}}-\boldsymbol{\mu}\bigr|+\pi_{0}^{*}\bigl|\frac{\boldsymbol{\lambda}_{0}^{*}}{\pi_{0}^{*}}-\boldsymbol{\mu}\bigr|\leq \boldsymbol{\sigma},\nonumber
\end{align}
by the third constraints. Thus, $\mathbb{P}^* \in \mathcal{P}$.

\textit{Proof of (ii).}  
By Theorem~1, the optimal value of Formulation~(O) equals the worst-case expectation $\sup_{\mathbb{P} \in \mathcal{P}} \mathbb{E}_{\mathbb{P}}\bigl[f(\boldsymbol{\zeta})\bigr]$. It remains to verify that $\mathbb{P}^*$ achieves this value:
\begin{align}
    &\mathbb{E}_{\mathbb{P}^*}\bigl[f(\boldsymbol{\zeta})\bigr] \nonumber\\
    =&\sum_{\hat{t}\in[t_1,t_2-1]\cup \{0\}}\pi_{\hat{t}}^{*}\Bigl(h\sum_{t\in [t_1,t_2-1]}\bigl(s-\sum_{t^{\prime}\in [t_1,t]} \frac{\lambda_{\hat{t}t^{\prime}}^{*}}{\pi_{\hat{t}}^{*}}\bigr)^++b\sum_{t\in [t_1,t_2-1]}\bigl(s-\sum_{t^{\prime}\in [t_1,t]} \frac{\lambda_{\hat{t}t^{\prime}}^{*}}{\pi_{\hat{t}}^{*}}\bigr)^-\Bigr)\nonumber\\
    =&\sum_{\hat{t}\in[t_1,t_2-1]}\pi_{\hat{t}}^{*}\Bigl(h\sum_{t\in [t_1,\hat{t})}\bigl(s-\sum_{t^{\prime}\in [t_1,t]} \frac{\lambda_{\hat{t}t^{\prime}}^{*}}{\pi_{\hat{t}}^{*}}\bigr)+b\sum_{t\in [\hat{t},t_2-1]}\bigl(\sum_{t^{\prime}\in [t_1,t]} \frac{\lambda_{\hat{t}t^{\prime}}^{*}}{\pi_{\hat{t}}^{*}}-s\bigr)\Bigr)+\pi_{0}^{*}h\sum_{t\in [t_1,t_2-1]}\bigl(s-\sum_{t^{\prime}\in [t_1,t]} \frac{\lambda_{0t^{\prime}}^{*}}{\pi_{0}^{*}}\bigr),\nonumber
\end{align}
where the last equality follows from the fifth and sixth constraint sets in Formulation~(O), and the expression coincides with the objective value of Formulation~(O). Hence, $\mathbb{P}^*$ is optimal.\hfill\qed

\subsection{Proof of Proposition 3.}
We consider two cases separately.

\textit{Case 1: $\boldsymbol{\sigma}^i = \boldsymbol{0}$.}
In this case, the ambiguity set $\mathcal{P}$ degenerates to a singleton that contains only the deterministic vector $\boldsymbol{\mu}^i$. The expected cost function becomes deterministic and piecewise linear, leading to a unique turning period. Therefore, the optimal policy exhibits a deterministic structure.

\textit{Case 2: $\boldsymbol{\sigma}^i \geq \max\{ \bar{\boldsymbol{\zeta}}^i - \boldsymbol{\mu}^i,\ \boldsymbol{\mu}^i - \underline{\boldsymbol{\zeta}}^i \}$.}
Under this condition, the ambiguity set $\mathcal{P}$ is entirely characterized by the support $[\underline{\boldsymbol{\zeta}}^i, \bar{\boldsymbol{\zeta}}^i]$ and the mean vector $\boldsymbol{\mu}^i$, since the deviation constraint is rendered inactive. The corresponding worst-case expectation problem can be written as the following moment problem:
\begin{align*}
\max & \int_{[\boldsymbol{\underline{\zeta}}^i,\boldsymbol{\bar{\zeta}}^i]} \Bigl( h \sum_{t=t_1}^{t_2 - 1} \bigl(s_{it_1} - \sum_{t^{\prime}=t_1}^{t} \zeta_{t^{\prime}}^i\bigr)^+ + b \sum_{t=t_1}^{t_2 - 1} \bigl(s_{it_1} - \sum_{t^{\prime}=t_1}^{t} \zeta_{t^{\prime}}^i\bigr)^- \Bigr) \mathbb{P}(\mathrm{d}\boldsymbol{\zeta}^i), \\
s.t.\ & \int_{[\boldsymbol{\underline{\zeta}}^i,\boldsymbol{\bar{\zeta}}^i]} \mathbb{P}(\mathrm{d}\boldsymbol{\zeta}^i) = 1, \nonumber\\
&\int_{[\boldsymbol{\underline{\zeta}}^i,\boldsymbol{\bar{\zeta}}^i]} \boldsymbol{\zeta}^i \mathbb{P}(\mathrm{d}\boldsymbol{\zeta}^i) = \boldsymbol{\mu}^i,\nonumber\\
&\mathbb{P}\in \mathcal{M}_+\bigl([\boldsymbol{\underline{\zeta}}^i,\boldsymbol{\bar{\zeta}}^i]\bigr)\nonumber.
\end{align*}

By Lemma EC.1 in \citet{ghosal2020distributionally}, strong duality holds for this moment problem. Its dual formulation is given by:
\begin{align}
\min\ & \alpha + {\boldsymbol{\mu}^i}^\top \boldsymbol{\beta}, \label{eq:dual-objective} \\
s.t.\ & \alpha + {\boldsymbol{\zeta}^i}^\top \boldsymbol{\beta} \geq f(\boldsymbol{\zeta}^i), \quad \forall \boldsymbol{\zeta}^i \in [\underline{\boldsymbol{\zeta}}^i, \bar{\boldsymbol{\zeta}}^i], \label{eq:dual-constraint}\\
& \alpha\in\mathbb{R}, \ \boldsymbol{\beta}\in\mathbb{R}^{\hat{T}},\nonumber
\end{align}
where
\[
f(\boldsymbol{\zeta}^i) := h \sum_{t=t_1}^{t_2 - 1} \bigl(s_{it_1} - \sum_{t^{\prime}=t_1}^{t} \zeta_{t^{\prime}}^i\bigr)^+ + b \sum_{t=t_1}^{t_2 - 1} \bigl(s_{it_1} - \sum_{t^{\prime}=t_1}^{t} \zeta_{t^{\prime}}^i\bigr)^-.
\]

Suppose now that the scaling condition
\[
m(\bar{\boldsymbol{\zeta}}^i - \boldsymbol{\mu}^i) = n(\boldsymbol{\mu}^i - \underline{\boldsymbol{\zeta}}^i), \quad \text{for some } m,n \geq 0,
\]
holds. Then, $\boldsymbol{\mu}^i$ lies on the convex hull of $\underline{\boldsymbol{\zeta}}^i$ and $\bar{\boldsymbol{\zeta}}^i$, and can be written as
\[
\boldsymbol{\mu}^i = \frac{m}{m+n} \bar{\boldsymbol{\zeta}}^i + \frac{n}{m+n} \underline{\boldsymbol{\zeta}}^i.
\]

We now construct a feasible dual solution $(\alpha^*, \boldsymbol{\beta}^*)$ that satisfies the dual Constraints \eqref{eq:dual-constraint} with equality at both endpoints:
\begin{align}
\alpha^* + {\bar{\boldsymbol{\zeta}}}^{i\top} \boldsymbol{\beta}^* &= f(\bar{\boldsymbol{\zeta}}^i), \label{eq:dual-tight-bar} \\
\alpha^* + {\underline{\boldsymbol{\zeta}}}^{i\top} \boldsymbol{\beta}^* &= f(\underline{\boldsymbol{\zeta}}^i). \label{eq:dual-tight-under}
\end{align}
Subtracting \eqref{eq:dual-tight-under} from \eqref{eq:dual-tight-bar} yields
\[
(\bar{\boldsymbol{\zeta}}^i - \underline{\boldsymbol{\zeta}}^i)^\top \boldsymbol{\beta}^* = f(\bar{\boldsymbol{\zeta}}^i) - f(\underline{\boldsymbol{\zeta}}^i),
\]
a linear equation in $\boldsymbol{\beta}^*$. Since $\bar{\boldsymbol{\zeta}}^i \neq \underline{\boldsymbol{\zeta}}^i$ by assumption, a solution $\boldsymbol{\beta}^*$ exists. Plugging into either \eqref{eq:dual-tight-bar} or \eqref{eq:dual-tight-under} gives a corresponding $\alpha^*$, ensuring dual feasibility.

Substituting $(\alpha^*, \boldsymbol{\beta}^*)$ into the dual objective \eqref{eq:dual-objective} and using the convex decomposition of $\boldsymbol{\mu}^i$ leads to:
\begin{align*}
\alpha^* + {\boldsymbol{\mu}}^{i\top} \boldsymbol{\beta}^* 
&= \frac{m}{m+n} (\alpha^* + {\bar{\boldsymbol{\zeta}}}^{i\top} \boldsymbol{\beta}^* ) 
+ \frac{n}{m+n} (\alpha^* + {\underline{\boldsymbol{\zeta}}}^{i\top} \boldsymbol{\beta}^* ) \\
&= \frac{m}{m+n} f(\bar{\boldsymbol{\zeta}}^i) + \frac{n}{m+n} f(\underline{\boldsymbol{\zeta}}^i).
\end{align*}

Thus, the dual objective matches the primal objective evaluated under the two-point distribution:
\[
\mathbb{P}^* = \frac{m}{m+n} \delta_{\bar{\boldsymbol{\zeta}}^i} + \frac{n}{m+n} \delta_{\underline{\boldsymbol{\zeta}}^i}.
\]

By strong duality, this distribution $\mathbb{P}^*$ is optimal. Therefore, the worst-case distribution is supported on the two extreme points $\bar{\boldsymbol{\zeta}}^i$ and $\underline{\boldsymbol{\zeta}}^i$, as claimed.
\hfill\qed

\subsection{Proof of Theorem 2.}
Consider the Formulation (2nd-RMP). Let $\iota$, $\theta_i$, $\delta_{it}$, $\psi_t$, and $\alpha_{rt}$ denote the dual variables associated with constraints~\eqref{iota}-\eqref{alpha}, respectively. Then, the corresponding dual problem can be written as:
\begin{align}
    {\rm (2nd-RMP-D)} \max \ &-\iota-\sum_{i\in [N]}\theta_i-Q\sum_{t\in [T]}\psi_t+p,\nonumber\\
    s.t. \ &\sum_{i\in [N]}o_r^i\delta_{it}-\alpha_{rt}\leq \frac{1}{T}c_r \quad\forall r\in R, t\in [T],\nonumber\\
    &-\iota+\sum_{t\in [T]}\alpha_{rt}\leq 0\quad\forall r\in R,\nonumber\\
    & -\theta_i-\sum_{t\in [T]}\beta^t_{l(i)}\delta_{it}-\sum_{t\in [T]}\beta^t_{l(i)}\bigl(s_{l(i)}^{t}-s_{l(i)}^{t^-}+\sum_{t^{\prime}\in[t^-,t-1]}a_{t^{\prime}}^i\bigr)\psi_t\leq \frac{1}{T}c_{l(i)}-\pi_i\nonumber\\
    &\qquad\qquad\qquad\qquad\qquad\qquad\qquad\qquad\qquad\qquad\quad \forall l(i)\in L(i), i\in [N]\nonumber. 
\end{align}

From the first and second constraints, we have:
\begin{align}
    -\iota \leq -\sum_{t \in [T]} \alpha_{rt} \leq \min\Bigl\{ \sum_{t \in [T]} \bigl( \frac{1}{T} c_r - \sum_{i \in [N]} o_r^i \delta_{it} \bigr),\ 0 \Bigr\}. \nonumber
\end{align}
Since the objective function maximizes $-\iota$, its optimal value is achieved when $-\iota$ is as large as possible. Thus, the optimal value of $-\sum_{t \in [T]} \alpha_{rt}$ also attains its maximum under this constraint, implying:
\begin{align}
    -\sum_{t \in [T]} \alpha_{rt} = \min\Bigl\{ \sum_{t \in [T]} \bigl( \frac{1}{T} c_r - \sum_{i \in [N]} o_r^i \delta_{it} \bigr),\ 0 \Bigr\}. \nonumber
\end{align}
As a result, the reduced cost of adding a new variable $\lambda_r$ is given by:
\begin{align}
    -\iota - \sum_{t \in [T]} \alpha_{rt} 
    = -\iota + \min\Bigl\{ \sum_{t \in [T]} \bigl( \frac{1}{T} c_r - \sum_{i \in [N]} o_r^i \delta_{it} \bigr),\ 0 \Bigr\}. \nonumber
\end{align}

According to the Complementary Slackness Theorem, the Constraints~\eqref{alpha} and its associated dual variable $\alpha_{rt}$ must satisfy the complementary slackness condition. In particular, if $\alpha_{rt} > 0$, then the constraint must be binding, that is,
\[
    \lambda_r - x_r^t = 0 \quad \Rightarrow \quad \lambda_r = x_r^t.
\]
Since $\alpha_{rt} =  \frac{1}{T} c_r - \sum_{i \in [N]} o_r^i \delta_{it} $ in this case, we conclude that route $r$ is executed in period~$t$.
\hfill\qed

\clearpage

\setcounter{equation}{0}
\renewcommand{\theequation}{E\arabic{equation}}
\setcounter{table}{0} 
\setcounter{figure}{0}
\renewcommand{\thetable}{E\arabic{table}}
\renewcommand{\thefigure}{E\arabic{figure}}
\setcounter{algorithm}{0}
\renewcommand{\thealgorithm}{E\arabic{algorithm}}
\section{Comprehensive Computational Experiment Results} \label{tables}

\begin{table}[!htb]
\centering
\caption{Solution performance of different service policies across different numbers of periods.}
\label{tab:oa-solution}
\tiny
\begin{tabular}{c c c r r r r r r r r r r r}
\toprule
\#Retailer &\#Period         & Policy    & Time (s)  &T.O.\%   & Cost & \#Cluster & Avg I. & S.L. & Vehicle Util & O.\% & Avg O. & E.T.\% & Avg E.T. \\
\midrule
\multirow{12}{*}{5}  & \multirow{3}{*}{4} & Fixed-Interval& 4 & 0\% & 218  & 3.0       & 2.7          & 92\%          & 55\%         & 0.0\%       & 0.0           & 0.0\%  & 0.0      \\
                     &                    & Consistent  & 266 & 0\%   & 195  & 2.0       & 1.5          & 94\%          & 55\%         & 0.0\%       & 0.0           & 0.0\%  & 0.0      \\
                     &                    & Flexible & 271  & 0\%     & 178  & 1.0       & 1.7          & 94\%          & 68\%         & 0.0\%       & 0.0           & 0.0\%  & 0.0      \\
                     \cmidrule(r){2-14}
                     & \multirow{3}{*}{5} & Fixed-Interval& 4 & 0\% & 217  & 3.0       & 2.6          & 92\%          & 54\%         & 0.0\%       & 0.0           & 0.0\%  & 0.0      \\
                     &                    & Consistent& 623  & 33\%    & 180  & 2.0       & 1.9          & 95\%          & 69\%         & 0.0\%       & 0.0           & 0.0\%  & 0.0      \\
                     &                    & Flexible  & 631   & 0\%   & 171  & 1.0       & 1.9          & 95\%          & 69\%         & 0.0\%       & 0.0           & 0.3\%  & 0.3      \\
                                          \cmidrule(r){2-14}
& \multirow{3}{*}{6} & Fixed-Interval& 4 & 0\% & 216  & 3.0       & 2.7          & 92\%          & 56\%         & 0.0\%       & 0.0           & 0.0\%  & 0.0      \\
                     &                    & Consistent & 839 & 100\%    & 187  & 2.0       & 1.7          & 93\%          & 60\%         & 0.0\%       & 0.0           & 0.0\%  & 0.0      \\
                     &                    & Flexible  & 914  & 100\%    & 175  & 1.0       & 2.0          & 94\%          & 70\%         & 0.0\%       & 0.0           & 0.0\%  & 0.0      \\
                                          \cmidrule(r){2-14}
& \multirow{3}{*}{7} & Fixed-Interval& 4 & 0\% & 217  & 3.0       & 2.7          & 92\%          & 56\%         & 0.0\%       & 0.0           & 0.0\%  & 0.0      \\
                     &                    & Consistent & 909 & 100\%    & 188  & 2.0       & 1.8          & 94\%          & 64\%         & 0.0\%       & 0.0           & 0.0\%  & 0.0      \\
                     &                    & Flexible  & 1074  & 100\%    & 175  & 1.0       & 2.1          & 94\%          & 70\%         & 0.0\%       & 0.0           & 0.0\%  & 0.0      \\
                     \midrule
\multirow{12}{*}{7}  & \multirow{3}{*}{4} & Fixed-Interval& 4 & 0\% & 271  & 2.0       & 1.5          & 89\%          & 63\%         & 0.0\%       & 0.0           & 0.0\%  & 0.0      \\
                     &                    & Consistent  & 827 & 100\%  & 240  & 2.0       & 1.4          & 94\%          & 63\%         & 0.0\%       & 0.0           & 0.0\%  & 0.0      \\
                     &                    & Flexible & 1857  & 100\%     & 227  & 2.0       & 1.6          & 95\%          & 63\%         & 0.0\%       & 0.0           & 0.4\%  & 5.7      \\
                                          \cmidrule(r){2-14}
& \multirow{3}{*}{5} & Fixed-Interval & 4& 0\% & 270  & 2.0       & 1.5          & 88\%          & 63\%         & 0.0\%       & 0.0           & 0.0\%  & 0.0      \\
                     &                    & Consistent  & 2105 & 95\%   & 242  & 2.0       & 1.5          & 94\%          & 64\%         & 0.0\%       & 0.0           & 0.2\%  & 2.7      \\
                     &                    & Flexible   & 2728  & 68\%   & 238  & 2.0       & 2.1          & 93\%          & 64\%         & 0.0\%       & 0.0           & 0.0\%  & 0.0      \\
                                          \cmidrule(r){2-14}
& \multirow{3}{*}{6} & Fixed-Interval& 5 & 0\% & 273  & 2.3       & 1.6          & 89\%          & 58\%         & 0.0\%       & 0.0           & 0.0\%  & 0.0      \\
                     &                    & Consistent & 940 & 100\%    & 249  & 2.0       & 1.6          & 94\%          & 64\%         & 0.0\%       & 0.0           & 0.4\%  & 5.3      \\
                     &                    & Flexible  & 2247 & 72\%     & 234  & 2.0       & 1.8          & 94\%          & 64\%         & 0.0\%       & 0.0           & 0.2\%  & 1.0      \\
                                         \cmidrule(r){2-14}
 & \multirow{3}{*}{7} & Fixed-Interval& 5 & 0\% & 273  & 2.3       & 1.6          & 89\%          & 59\%         & 0.0\%       & 0.0           & 0.0\%  & 0.0      \\
                     &                    & Consistent  & 3368 & 88\%   & 254  & 2.0       & 1.4          & 94\%          & 64\%         & 0.6\%       & 1.1           & 0.0\%  & 0.0      \\
                     &                    & Flexible  & 731  & 100\%    & 238  & 2.0       & 1.8          & 94\%          & 60\%         & 0.0\%       & 0.0           & 0.0\%  & 0.0      \\
                     \midrule
\multirow{12}{*}{10} & \multirow{3}{*}{4} & Fixed-Interval& 4 & 0\% & 370  & 4.0       & 1.8          & 91\%          & 55\%         & 0.0\%       & 0.0           & 0.0\%  & 0.0      \\
                     &                    & Consistent & 1730  & 100\%   & 334  & 3.3       & 1.7          & 95\%          & 67\%         & 0.0\%       & 0.0           & 0.2\%  & 3.3      \\
                     &                    & Flexible & 2502 & 60\%      & 315  & 2.0       & 1.7          & 94\%          & 69\%         & 0.0\%       & 0.0           & 0.0\%  & 0.0      \\
                                          \cmidrule(r){2-14}
& \multirow{3}{*}{5} & Fixed-Interval & 5& 0\% & 369  & 4.0       & 1.8          & 90\%          & 55\%         & 0.0\%       & 0.0           & 0.0\%  & 0.0      \\
                     &                    & Consistent  & 3611  & 0\%  & 345  & 3.7       & 1.9          & 94\%          & 67\%         & 0.0\%       & 0.0           & 0.0\%  & 0.0      \\
                     &                    & Flexible   & 3611  & 77\%   & 328  & 2.7       & 1.8          & 94\%          & 63\%         & 0.1\%       & 0.3           & 0.0\%  & 0.0      \\
                                          \cmidrule(r){2-14}
& \multirow{3}{*}{6} & Fixed-Interval& 5 & 0\% & 369  & 4.0       & 1.8          & 90\%          & 55\%         & 0.0\%       & 0.0           & 0.0\%  & 0.0      \\
                     &                    & Consistent & 3617 & 0\%    & 473  & 3.3       & 1.1          & 99\%          & 47\%         & 1.4\%       & 1.9           & 0.0\%  & 0.0      \\
                     &                    & Flexible  & 2513 & 89\%     & 320  & 2.0       & 1.7          & 94\%          & 69\%         & 0.1\%       & 0.7           & 0.0\%  & 0.0      \\
                                          \cmidrule(r){2-14}
& \multirow{3}{*}{7} & Fixed-Interval & 5 & 0\%   & 369  & 4.0       & 1.8          & 90\%          & 55\%         & 0.0\%       & 0.0           & 0.0\%  & 0.0      \\
                     &                    & Consistent  & 3623 & 0\%   & 476  & 4.0       & 1.0          & 100\%         & 35\%         & 1.8\%       & 1.9           & 0.0\%  & 0.0      \\
                     &                    & Flexible    & 3622 & 77\%   & 339  & 2.7       & 1.7          & 95\%          & 55\%         & 0.2\%       & 1.2           & 0.0\%  & 0.0      \\
                     \midrule
\multirow{12}{*}{12} & \multirow{3}{*}{4} & Fixed-Interval& 5 & 0\% & 411  & 3.7       & 1.5          & 90\%          & 60\%         & 0.0\%       & 0.0           & 0.0\%  & 0.0      \\
                     &                    & Consistent  & 3613 & 0\%   & 518  & 4.0       & 1.0          & 100\%         & 42\%         & 1.9\%       & 2.0           & 0.0\%  & 0.0      \\
                     &                    & Flexible   & 1874  & 89\%   & 358  & 3.0       & 1.5          & 94\%          & 66\%         & 0.0\%       & 0.0           & 0.0\%  & 0.0      \\
                                          \cmidrule(r){2-14}
& \multirow{3}{*}{5} & Fixed-Interval& 6 & 0\% & 406  & 3.7       & 1.5          & 89\%          & 60\%         & 0.0\%       & 0.0           & 0.0\%  & 0.0      \\
                     &                    & Consistent  & 3615  & 0\%  & 517  & 4.0       & 1.0          & 100\%         & 42\%         & 1.7\%       & 1.7           & 0.0\%  & 0.0      \\
                     &                    & Flexible  & 2261  & 92\%    & 367  & 3.0       & 1.7          & 94\%          & 66\%         & 0.0\%       & 0.0           & 0.0\%  & 0.0      \\
                                          \cmidrule(r){2-14}
& \multirow{3}{*}{6} & Fixed-Interval& 7 & 0\% & 407  & 3.7       & 1.5          & 90\%          & 60\%         & 0.0\%       & 0.0           & 0.0\%  & 0.0      \\
                     &                    & Consistent & 3619  & 0\%   & 515  & 4.0       & 1.0          & 100\%         & 42\%         & 1.7\%       & 1.8           & 0.0\%  & 0.0      \\
                     &                    & Flexible  & 2847  & 92\%    & 365  & 3.0       & 1.7          & 94\%          & 65\%         & 0.0\%       & 0.0           & 0.0\%  & 0.0      \\
                                          \cmidrule(r){2-14}
& \multirow{3}{*}{7} & Fixed-Interval& 7 & 0\% & 419  & 4.0       & 1.5          & 89\%          & 56\%         & 0.0\%       & 0.0           & 0.0\%  & 0.0      \\
                     &                    & Consistent  & 3626 & 0\%   & 515  & 4.0       & 1.0          & 100\%         & 42\%         & 1.5\%       & 1.9           & 0.0\%  & 0.0      \\
                     &                    & Flexible  & 3443  & 75\%    & 371  & 3.0       & 1.7          & 94\%          & 62\%         & 0.0\%       & 0.3           & 0.0\%  & 0.0      \\
                     \midrule
\multirow{12}{*}{15} & \multirow{3}{*}{4} & Fixed-Interval & 8& 0\% & 486  & 4.0       & 1.3          & 89\%          & 60\%         & 0.0\%       & 0.0           & 0.0\%  & 0.0      \\
                     &                    & Consistent   & 3610 & 0\% & 602  & 5.0       & 1.0          & 100\%         & 42\%         & 1.7\%       & 2.0           & 0.0\%  & 0.0      \\
                     &                    & Flexible  & 3612  & 0\%   & 458  & 3.7       & 1.5          & 95\%          & 59\%         & 0.3\%       & 0.9           & 0.1\%  & 0.3      \\
                                         \cmidrule(r){2-14}
 & \multirow{3}{*}{5} & Fixed-Interval& 6& 0\% & 480  & 3.7       & 1.2          & 89\%          & 64\%         & 0.0\%       & 0.0           & 0.0\%  & 0.0      \\
                     &                    & Consistent & 3616 & 0\%   & 601  & 5.0       & 1.0          & 100\%         & 43\%         & 1.6\%       & 1.6           & 0.0\%  & 0.0      \\
                     &                    & Flexible  & 3489  & 33\%   & 450  & 3.3       & 1.4          & 95\%          & 65\%         & 0.4\%       & 0.9           & 0.3\%  & 1.0      \\
                                          \cmidrule(r){2-14}
& \multirow{3}{*}{6} & Fixed-Interval& 7 & 0\%& 483  & 3.7       & 1.2          & 89\%          & 63\%         & 0.0\%       & 0.0           & 0.0\%  & 0.0      \\
                     &                    & Consistent  & 3620 & 0\%  & 600  & 5.0       & 1.0          & 100\%         & 42\%         & 1.7\%       & 1.7           & 0.0\%  & 0.0      \\
                     &                    & Flexible  & 3638  & 0\%   & 477  & 3.7       & 1.4          & 96\%          & 59\%         & 1.1\%       & 2.9           & 0.1\%  & 1.0      \\
                                         \cmidrule(r){2-14}
 & \multirow{3}{*}{7} & Fixed-Interval& 7& 0\% & 491  & 4.0       & 1.3          & 89\%          & 61\%         & 0.0\%       & 0.0           & 0.0\%  & 0.0      \\
                     &                    & Consistent & 3628 & 0\%   & 599  & 5.0       & 1.0          & 100\%         & 42\%         & 1.5\%       & 1.9           & 0.0\%  & 0.0      \\
                     &                    & Flexible  & 3636  & 0\%   & 475  & 3.7       & 1.5          & 96\%          & 60\%         & 1.1\%       & 2.6           & 0.0\%  & 0.0    \\
                     \bottomrule
\end{tabular}
\end{table}

\begin{table}[!htb]
\centering
\caption{Computation performance of different service policies across different numbers of periods.}
\label{tab:oa-compu}
\tiny
\begin{tabular}{c c c r r r r r r r r r r r}
\toprule
\multirow{2}{*}{\#Retailer} & \multirow{2}{*}{\#Period} & \multirow{2}{*}{Policy} & \multirow{2}{*}{Cost} & \multirow{2}{*}{Time (s)} & \multirow{2}{*}{T.O.\%} & \multicolumn{3}{c}{First-Level} & \multicolumn{5}{c}{Second-Level}                         \\
\cmidrule(r){7-9}
\cmidrule(r){10-14}
                            &                           &                         &                       &                           &                         & \#Node   & \#Column  & PP T.  & \#Node & \#Route & Route T. & \#Replnsh & Replnsh T. \\
                            \midrule
\multirow{12}{*}{5}         & \multirow{3}{*}{4}        & Fixed-Interval          & 218                   & 4                         & 0\%                     & 7        & 40        & 0        & -      & -       & -          & -         & -            \\
                            &                           & Consistent              & 195                   & 266                       & 0\%                     & 1        & 15        & 261      & 3049   & 10535   & 34         & 6524      & 133          \\
                            &                           & Flexible                & 178                   & 271                       & 0\%                     & 1        & 20        & 263      & 4859   & 4888    & 9          & 7130      & 126          \\
                            \cmidrule(r){2-14}
                            & \multirow{3}{*}{5}        & Fixed-Interval          & 217                   & 4                         & 0\%                     & 7        & 42        & 0        & -      & -       & -          & -         & -            \\
                            &                           & Consistent              & 180                   & 623                       & 33\%                    & 1        & 17        & 617      & 4022   & 15108   & 93         & 13667     & 345          \\
                            &                           & Flexible                & 171                   & 631                       & 0\%                     & 1        & 25        & 625      & 5946   & 6034    & 16         & 14669     & 380          \\
                                                        \cmidrule(r){2-14}
& \multirow{3}{*}{6}        & Fixed-Interval          & 216                   & 4                         & 0\%                     & 7        & 44        & 0        & -      & -       & -          & -         & -            \\
                            &                           & Consistent              & 187                   & 839                       & 100\%                   & 1        & 22        & 830      & 3137   & 17124   & 152        & 13381     & 475          \\
                            &                           & Flexible                & 175                   & 914                       & 100\%                   & 1        & 22        & 905      & 6127   & 4959    & 14         & 13515     & 621          \\
                                                        \cmidrule(r){2-14}
& \multirow{3}{*}{7}        & Fixed-Interval          & 217                   & 4                         & 0\%                     & 7        & 41        & 0        & -      & -       & -          & -         & -            \\
                            &                           & Consistent              & 188                   & 909                       & 100\%                   & 1        & 19        & 895      & 2168   & 18154   & 250        & 11908     & 429          \\
                            &                           & Flexible                & 175                   & 1074                      & 100\%                   & 1        & 21        & 1061     & 4680   & 4498    & 14         & 12511     & 670          \\
                            \midrule
\multirow{12}{*}{7}         & \multirow{3}{*}{4}        & Fixed-Interval          & 271                   & 4                         & 0\%                     & 10       & 87        & 0        & -      & -       & -          & -         & -            \\
                            &                           & Consistent              & 240                   & 827                       & 100\%                   & 1        & 26        & 821      & 7763   & 28486   & 195        & 16014     & 251          \\
                            &                           & Flexible                & 227                   & 1857                      & 100\%                   & 1        & 41        & 1849     & 23348  & 34624   & 98         & 29530     & 548          \\
                                                        \cmidrule(r){2-14}
& \multirow{3}{*}{5}        & Fixed-Interval          & 270                   & 4                         & 0\%                     & 8        & 93        & 0        & -      & -       & -          & -         & -            \\
                            &                           & Consistent              & 242                   & 2105                      & 95\%                    & 3        & 40        & 2097     & 9947   & 50503   & 593        & 35933     & 847          \\
                            &                           & Flexible                & 238                   & 2728                      & 68\%                    & 8        & 41        & 2719     & 19369  & 24681   & 103        & 37238     & 1160         \\
                                                       \cmidrule(r){2-14}
 & \multirow{3}{*}{6}        & Fixed-Interval          & 273                   & 5                         & 0\%                     & 8        & 83        & 0        & -      & -       & -          & -         & -            \\
                            &                           & Consistent              & 249                   & 940                       & 100\%                   & 1        & 26        & 928      & 2056   & 29000   & 360        & 11107     & 302          \\
                            &                           & Flexible                & 234                   & 2247                      & 72\%                    & 5        & 52        & 2236     & 6542   & 19739   & 43         & 20387     & 691          \\
                                                       \cmidrule(r){2-14}
 & \multirow{3}{*}{7}        & Fixed-Interval          & 273                   & 5                         & 0\%                     & 9        & 86        & 0        & -      & -       & -          & -         & -            \\
                            &                           & Consistent              & 254                   & 3368                      & 88\%                    & 5        & 40        & 3351     & 3951   & 73299   & 1728       & 22622     & 802          \\
                            &                           & Flexible                & 238                   & 731                       & 100\%                   & 1        & 18        & 714      & 1019   & 7079    & 11         & 4749      & 250          \\
                            \midrule
\multirow{12}{*}{10}        & \multirow{3}{*}{4}        & Fixed-Interval          & 370                   & 4                         & 0\%                     & 1        & 87        & 0        & -      & -       & -          & -         & -            \\
                            &                           & Consistent              & 334                   & 1730                      & 100\%                   & 1        & 42        & 1723     & 2805   & 51860   & 1250       & 8103      & 83           \\
                            &                           & Flexible                & 315                   & 2502                      & 60\%                    & 8        & 94        & 2493     & 17024  & 81069   & 240        & 30570     & 381          \\
                                                       \cmidrule(r){2-14}
 & \multirow{3}{*}{5}        & Fixed-Interval          & 369                   & 5                         & 0\%                     & 1        & 98        & 0        & -      & -       & -          & -         & -            \\
                            &                           & Consistent              & 345                   & 3611                      & 0\%                     & 1        & 34        & 3600     & 282    & 44229   & 3392       & 2864      & 72           \\
                            &                           & Flexible                & 328                   & 3611                      & 77\%                    & 4        & 91        & 3600     & 16811  & 77253   & 261        & 30665     & 940          \\
                                                       \cmidrule(r){2-14}
 & \multirow{3}{*}{6}        & Fixed-Interval          & 369                   & 5                         & 0\%                     & 1        & 95        & 0        & -      & -       & -          & -         & -            \\
                            &                           & Consistent              & 473                   & 3617                      & 0\%                     & 1        & 14        & 3600     & 85     & 19225   & 3447       & 1247      & 101          \\
                            &                           & Flexible                & 320                   & 2513                      & 89\%                    & 2        & 53        & 2496     & 1223   & 51469   & 95         & 12266     & 353          \\
                                                      \cmidrule(r){2-14}
  & \multirow{3}{*}{7}        & Fixed-Interval          & 369                   & 5                         & 0\%                     & 1        & 84        & 0        & -      & -       & -          & -         & -            \\
                            &                           & Consistent              & 476                   & 3623                      & 0\%                     & 1        & 8         & 3600     & 57     & 12959   & 3415       & 809       & 150          \\
                            &                           & Flexible                & 339                   & 3622                      & 77\%                    & 4        & 77        & 3600     & 4383   & 77678   & 155        & 20673     & 808          \\
                            \midrule
\multirow{12}{*}{12}        & \multirow{3}{*}{4}        & Fixed-Interval          & 411                   & 5                         & 0\%                     & 26       & 281       & 0        & -      & -       & -          & -         & -            \\
                            &                           & Consistent              & 518                   & 3613                      & 0\%                     & 1        & 13        & 3605     & 82     & 7157    & 3569       & 647       & 22           \\
                            &                           & Flexible                & 358                   & 1874                      & 89\%                    & 2        & 78        & 1863     & 5355   & 91303   & 395        & 16390     & 170          \\
                                                       \cmidrule(r){2-14}
 & \multirow{3}{*}{5}        & Fixed-Interval          & 406                   & 6                         & 0\%                     & 19       & 245       & 0        & -      & -       & -          & -         & -            \\
                            &                           & Consistent              & 517                   & 3615                      & 0\%                     & 1        & 11        & 3603     & 67     & 7030    & 3529       & 625       & 53           \\
                            &                           & Flexible                & 367                   & 2261                      & 92\%                    & 2        & 69        & 2249     & 3539   & 77738   & 348        & 11938     & 288          \\
                                                      \cmidrule(r){2-14}
  & \multirow{3}{*}{6}        & Fixed-Interval          & 407                   & 7                         & 0\%                     & 36       & 352       & 1        & -      & -       & -          & -         & -            \\
                            &                           & Consistent              & 515                   & 3619                      & 0\%                     & 1        & 12        & 3600     & 74     & 11015   & 3451       & 1001      & 116          \\
                            &                           & Flexible                & 365                   & 2847                      & 92\%                    & 2        & 64        & 2828     & 978    & 88226   & 381        & 13930     & 373          \\
                                                      \cmidrule(r){2-14}
  & \multirow{3}{*}{7}        & Fixed-Interval          & 419                   & 7                         & 0\%                     & 35       & 322       & 0        & -      & -       & -          & -         & -            \\
                            &                           & Consistent              & 515                   & 3626                      & 0\%                     & 1        & 11        & 3600     & 65     & 14370   & 3374       & 1008      & 187          \\
                            &                           & Flexible                & 371                   & 3443                      & 75\%                    & 2        & 60        & 3413     & 694    & 108437  & 427        & 12296     & 525          \\
                            \midrule
\multirow{12}{*}{15}        & \multirow{3}{*}{4}        & Fixed-Interval          & 486                   & 8                         & 0\%                     & 76       & 834       & 2        & -      & -       & -          & -         & -            \\
                            &                           & Consistent              & 602                   & 3610                      & 0\%                     & 1        & 13        & 3601     & 81     & 6164    & 3565       & 560       & 25           \\
                            &                           & Flexible                & 458                   & 3612                      & 0\%                     & 1        & 69        & 3600     & 1121   & 96775   & 3021       & 9047      & 70           \\
                                                       \cmidrule(r){2-14}
 & \multirow{3}{*}{5}        & Fixed-Interval          & 480                   & 6                         & 0\%                     & 14       & 358       & 1        & -      & -       & -          & -         & -            \\
                            &                           & Consistent              & 601                   & 3616                      & 0\%                     & 1        & 14        & 3603     & 86     & 8976    & 3514       & 866       & 62           \\
                            &                           & Flexible                & 450                   & 3489                      & 33\%                    & 1        & 56        & 3476     & 555    & 88954   & 2673       & 6996      & 122          \\
                                                      \cmidrule(r){2-14}
  & \multirow{3}{*}{6}        & Fixed-Interval          & 483                   & 7                         & 0\%                     & 14       & 387       & 1        & -      & -       & -          & -         & -            \\
                            &                           & Consistent              & 600                   & 3620                      & 0\%                     & 1        & 11        & 3600     & 65     & 10739   & 3451       & 836       & 124          \\
                            &                           & Flexible                & 477                   & 3638                      & 0\%                     & 1        & 48        & 3618     & 336    & 92864   & 2585       & 9055      & 238          \\
                                                   \cmidrule(r){2-14}
     & \multirow{3}{*}{7}        & Fixed-Interval          & 491                   & 7                         & 0\%                     & 25       & 453       & 1        & -      & -       & -          & -         & -            \\
                            &                           & Consistent              & 599                   & 3628                      & 0\%                     & 1        & 9         & 3600     & 45     & 11992   & 3372       & 810       & 196          \\
                            &                           & Flexible                & 475                   & 3636                      & 0\%                     & 1        & 41        & 3604     & 290    & 90405   & 2162       & 7612      & 369   \\
                            \bottomrule
\end{tabular}
\end{table}

\begin{table}[!htb]
\centering
\caption{Solution performance of different service policies across different numbers of periods under stationary demand.}
\label{tab:oa-sta-solution}
\tiny
\begin{tabular}{c c c r r r r r r r r r r r}
\toprule
\#Retailer           & \#Period           & Policy         & Time (s) & T.O.\% & Cost & \#Cluster & Avg I. & S.L. & Vehicle Util & O.\%  & Avg O. & E.T.\% & Avg E.T. \\
\midrule
\multirow{12}{*}{5}  & \multirow{3}{*}{4} & Fixed-Interval & 4        & 0\%    & 191  & 2.3       & 2.2          & 83\%          & 65\%         & 0.0\% & 0.0    & 0.7\%  & 1.0      \\
                     &                    & Consistent     & 227      & 0\%    & 190  & 2.0       & 1.7          & 84\%          & 59\%         & 0.0\% & 0.0    & 0.3\%  & 0.3      \\
                     &                    & Flexible       & 276      & 0\%    & 174  & 1.0       & 1.8          & 93\%          & 64\%         & 0.0\% & 0.0    & 0.0\%  & 0.0      \\
                                        \cmidrule(r){2-14}
  & \multirow{3}{*}{5} & Fixed-Interval & 4        & 0\%    & 191  & 2.3       & 2.2          & 82\%          & 65\%         & 0.0\% & 0.0    & 1.6\%  & 3.0      \\
                     &                    & Consistent     & 624      & 100\%  & 187  & 2.0       & 1.8          & 86\%          & 65\%         & 0.0\% & 0.0    & 0.3\%  & 0.3      \\
                     &                    & Flexible       & 655      & 33\%   & 176  & 1.0       & 1.9          & 94\%          & 64\%         & 0.0\% & 0.0    & 0.0\%  & 0.0      \\
                                       \cmidrule(r){2-14}
   & \multirow{3}{*}{6} & Fixed-Interval & 4        & 0\%    & 191  & 2.3       & 2.2          & 82\%          & 66\%         & 0.0\% & 0.0    & 2.9\%  & 2.7      \\
                     &                    & Consistent     & 844      & 100\%  & 184  & 2.0       & 1.8          & 83\%          & 63\%         & 0.3\% & 0.9    & 1.0\%  & 1.3      \\
                     &                    & Flexible       & 365      & 67\%   & 173  & 1.0       & 2.0          & 93\%          & 64\%         & 0.0\% & 0.0    & 0.0\%  & 0.0      \\
                                       \cmidrule(r){2-14}
   & \multirow{3}{*}{7} & Fixed-Interval & 4        & 0\%    & 191  & 2.3       & 2.2          & 82\%          & 67\%         & 0.0\% & 0.0    & 3.2\%  & 3.4      \\
                     &                    & Consistent     & 728      & 100\%  & 197  & 2.0       & 1.8          & 83\%          & 59\%         & 0.0\% & 0.0    & 0.0\%  & 0.0      \\
                     &                    & Flexible       & 1303     & 87\%   & 181  & 1.3       & 1.9          & 92\%          & 61\%         & 0.0\% & 0.0    & 0.0\%  & 0.0      \\
\midrule
\multirow{12}{*}{7}  & \multirow{3}{*}{4} & Fixed-Interval & 4        & 0\%    & 245  & 2.7       & 2.2          & 82\%          & 70\%         & 0.0\% & 0.0    & 0.5\%  & 0.8      \\
                     &                    & Consistent     & 704      & 100\%  & 233  & 2.0       & 1.5          & 87\%          & 67\%         & 0.3\% & 0.8    & 2.1\%  & 1.0      \\
                     &                    & Flexible       & 2981     & 57\%   & 226  & 2.0       & 1.9          & 93\%          & 63\%         & 0.0\% & 0.0    & 0.0\%  & 0.0      \\
                                        \cmidrule(r){2-14}
  & \multirow{3}{*}{5} & Fixed-Interval & 4        & 0\%    & 245  & 2.7       & 2.2          & 83\%          & 69\%         & 0.0\% & 0.0    & 1.1\%  & 1.2      \\
                     &                    & Consistent     & 1142     & 100\%  & 240  & 2.0       & 1.5          & 88\%          & 65\%         & 0.2\% & 0.8    & 0.5\%  & 0.8      \\
                     &                    & Flexible       & 2360     & 83\%   & 236  & 2.0       & 1.9          & 95\%          & 61\%         & 0.0\% & 0.0    & 0.0\%  & 0.0      \\
                                        \cmidrule(r){2-14}
  & \multirow{3}{*}{6} & Fixed-Interval & 4        & 0\%    & 245  & 2.7       & 2.2          & 82\%          & 71\%         & 0.0\% & 0.0    & 1.8\%  & 5.0      \\
                     &                    & Consistent     & 1957     & 96\%   & 239  & 2.7       & 1.9          & 84\%          & 70\%         & 0.3\% & 0.5    & 1.9\%  & 2.4      \\
                     &                    & Flexible       & 2967     & 56\%   & 228  & 2.0       & 2.0          & 93\%          & 62\%         & 0.0\% & 0.0    & 0.0\%  & 0.0      \\
                                        \cmidrule(r){2-14}
  & \multirow{3}{*}{7} & Fixed-Interval & 5        & 0\%    & 245  & 2.7       & 2.2          & 81\%          & 71\%         & 0.0\% & 0.0    & 1.2\%  & 3.8      \\
                     &                    & Consistent     & 1433     & 100\%  & 244  & 2.0       & 1.7          & 86\%          & 68\%         & 0.3\% & 0.7    & 2.5\%  & 3.1      \\
                     &                    & Flexible       & 1803     & 72\%   & 233  & 2.0       & 1.8          & 94\%          & 59\%         & 0.0\% & 0.0    & 0.0\%  & 0.0      \\
\midrule
\multirow{12}{*}{10} & \multirow{3}{*}{4} & Fixed-Interval & 7        & 0\%    & 335  & 3.7       & 2.0          & 82\%          & 69\%         & 0.0\% & 0.0    & 2.5\%  & 1.5      \\
                     &                    & Consistent     & 3039     & 87\%   & 328  & 3.7       & 1.7          & 85\%          & 64\%         & 0.1\% & 0.3    & 0.4\%  & 1.2      \\
                     &                    & Flexible       & 2641     & 69\%   & 304  & 2.0       & 2.1          & 94\%          & 64\%         & 0.0\% & 0.0    & 0.0\%  & 0.0      \\
                                        \cmidrule(r){2-14}
  & \multirow{3}{*}{5} & Fixed-Interval & 7        & 0\%    & 335  & 3.7       & 2.0          & 82\%          & 69\%         & 0.0\% & 0.0    & 1.3\%  & 2.1      \\
                     &                    & Consistent     & 3436     & 33\%   & 343  & 3.3       & 1.7          & 87\%          & 70\%         & 0.5\% & 0.5    & 2.7\%  & 2.6      \\
                     &                    & Flexible       & 2981     & 84\%   & 323  & 2.3       & 1.7          & 95\%          & 61\%         & 0.0\% & 0.0    & 0.0\%  & 0.0      \\
                                         \cmidrule(r){2-14}
 & \multirow{3}{*}{6} & Fixed-Interval & 7        & 0\%    & 335  & 3.7       & 2.0          & 82\%          & 68\%         & 0.0\% & 0.0    & 1.6\%  & 3.3      \\
                     &                    & Consistent     & 3615     & 22\%   & 369  & 3.0       & 1.3          & 91\%          & 63\%         & 0.0\% & 0.0    & 4.3\%  & 3.6      \\
                     &                    & Flexible       & 2720     & 90\%   & 311  & 2.0       & 2.0          & 93\%          & 64\%         & 0.0\% & 0.0    & 0.0\%  & 0.0      \\
                                        \cmidrule(r){2-14}
  & \multirow{3}{*}{7} & Fixed-Interval & 7        & 0\%    & 335  & 3.7       & 2.0          & 82\%          & 69\%         & 0.0\% & 0.0    & 1.7\%  & 3.2      \\
                     &                    & Consistent     & 3620     & 0\%    & 386  & 3.0       & 1.0          & 98\%          & 46\%         & 0.0\% & 0.0    & 0.0\%  & 0.0      \\
                     &                    & Flexible       & 3623     & 67\%   & 324  & 2.0       & 2.0          & 93\%          & 64\%         & 0.0\% & 0.0    & 0.0\%  & 0.0      \\
\midrule
\multirow{12}{*}{12} & \multirow{3}{*}{4} & Fixed-Interval & 7        & 0\%    & 366  & 3.0       & 1.3          & 81\%          & 66\%         & 0.0\% & 0.0    & 0.0\%  & 0.0      \\
                     &                    & Consistent     & 3608     & 0\%    & 433  & 3.0       & 1.0          & 98\%          & 55\%         & 0.0\% & 0.0    & 0.0\%  & 0.0      \\
                     &                    & Flexible       & 3608     & 58\%   & 376  & 3.0       & 1.7          & 92\%          & 59\%         & 3.3\% & 6.3    & 0.0\%  & 0.0      \\
                                          \cmidrule(r){2-14}
& \multirow{3}{*}{5} & Fixed-Interval & 8        & 0\%    & 366  & 3.0       & 1.3          & 80\%          & 67\%         & 0.0\% & 0.0    & 0.0\%  & 0.0      \\
                     &                    & Consistent     & 3613     & 0\%    & 433  & 3.0       & 1.0          & 98\%          & 56\%         & 0.0\% & 0.0    & 0.0\%  & 0.0      \\
                     &                    & Flexible       & 3612     & 47\%   & 375  & 3.0       & 1.6          & 94\%          & 62\%         & 0.9\% & 3.3    & 0.0\%  & 0.0      \\
                                         \cmidrule(r){2-14}
 & \multirow{3}{*}{6} & Fixed-Interval & 8        & 0\%    & 366  & 3.0       & 1.3          & 81\%          & 66\%         & 0.0\% & 0.0    & 0.1\%  & 0.7      \\
                     &                    & Consistent     & 3620     & 0\%    & 433  & 3.0       & 1.0          & 98\%          & 56\%         & 0.0\% & 0.0    & 0.0\%  & 0.0      \\
                     &                    & Flexible       & 3618     & 73\%   & 363  & 3.0       & 1.8          & 92\%          & 64\%         & 0.0\% & 0.0    & 0.0\%  & 0.0      \\
                                         \cmidrule(r){2-14}
 & \multirow{3}{*}{7} & Fixed-Interval & 8        & 0\%    & 366  & 3.0       & 1.3          & 80\%          & 67\%         & 0.0\% & 0.0    & 0.0\%  & 0.0      \\
                     &                    & Consistent     & 3623     & 0\%    & 433  & 3.0       & 1.0          & 98\%          & 56\%         & 0.0\% & 0.0    & 0.0\%  & 0.0      \\
                     &                    & Flexible       & 3633     & 39\%   & 390  & 3.0       & 1.7          & 93\%          & 57\%         & 2.4\% & 2.0    & 0.0\%  & 0.0      \\
\midrule
\multirow{12}{*}{15} & \multirow{3}{*}{4} & Fixed-Interval & 10       & 0\%    & 424  & 3.3       & 1.3          & 82\%          & 71\%         & 0.0\% & 0.0    & 1.2\%  & 1.9      \\
                     &                    & Consistent     & 3608     & 0\%    & 499  & 4.0       & 1.0          & 98\%          & 53\%         & 0.0\% & 0.0    & 0.0\%  & 0.0      \\
                     &                    & Flexible       & 3609     & 17\%   & 444  & 3.0       & 1.6          & 93\%          & 66\%         & 0.0\% & 0.0    & 0.0\%  & 0.0      \\
                     \cmidrule(r){2-14}
                     & \multirow{3}{*}{5} & Fixed-Interval & 10       & 0\%    & 424  & 3.3       & 1.3          & 82\%          & 72\%         & 0.0\% & 0.0    & 0.9\%  & 1.9      \\
                     &                    & Consistent     & 3611     & 0\%    & 499  & 4.0       & 1.0          & 98\%          & 53\%         & 0.0\% & 0.0    & 0.0\%  & 0.0      \\
                     &                    & Flexible       & 3612     & 17\%   & 435  & 3.0       & 1.4          & 95\%          & 67\%         & 0.2\% & 0.4    & 0.1\%  & 2.0      \\
                                         \cmidrule(r){2-14}
 & \multirow{3}{*}{6} & Fixed-Interval & 11       & 0\%    & 424  & 3.3       & 1.3          & 82\%          & 71\%         & 0.0\% & 0.0    & 0.9\%  & 1.7      \\
                     &                    & Consistent     & 3617     & 0\%    & 499  & 4.0       & 1.0          & 98\%          & 53\%         & 0.0\% & 0.0    & 0.0\%  & 0.0      \\
                     &                    & Flexible       & 3617     & 0\%    & 471  & 3.7       & 1.2          & 98\%          & 55\%         & 0.0\% & 0.0    & 0.0\%  & 0.0      \\
                                         \cmidrule(r){2-14}
 & \multirow{3}{*}{7} & Fixed-Interval & 9        & 0\%    & 424  & 3.3       & 1.3          & 82\%          & 71\%         & 0.0\% & 0.0    & 1.2\%  & 1.8      \\
                     &                    & Consistent     & 3625     & 0\%    & 499  & 4.0       & 1.0          & 98\%          & 53\%         & 0.0\% & 0.0    & 0.0\%  & 0.0      \\
                     &                    & Flexible       & 3627     & 0\%    & 459  & 3.3       & 1.3          & 96\%          & 61\%         & 0.7\% & 2.0    & 0.0\%  & 0.0     \\
                     \bottomrule
\end{tabular}
\end{table}

\begin{table}[!htb]
\centering
\caption{Computation performance of different service policies across different numbers of periods under stationary demand.}
\label{tab:oa-sta-compu}
\tiny
\begin{tabular}{c c c r r r r r r r r r r r}
\toprule
\multirow{2}{*}{\#Retailer} & \multirow{2}{*}{\#Period} & \multirow{2}{*}{Policy} & \multirow{2}{*}{Cost} & \multirow{2}{*}{Time (s)} & \multirow{2}{*}{T.O.\%} & \multicolumn{3}{c}{First-Level} & \multicolumn{5}{c}{Second-Level}                         \\
\cmidrule(r){7-9}
\cmidrule(r){10-14}
                            &                           &                         &                       &                           &                         & \#Node   & \#Column  & PP T.  & \#Node & \#Route & Route T. & \#Replnsh & Replnsh T. \\
                            \midrule
\multirow{12}{*}{5}         & \multirow{3}{*}{4}        & Fixed-Interval          & 191                   & 4                         & 0\%                     & 6        & 39        & 0        & 0      & 0       & 0          & 0         & 0            \\
                            &                           & Consistent              & 190                   & 227                       & 0\%                     & 1        & 12        & 222      & 2487   & 8351    & 30         & 4813      & 123          \\
                            &                           & Flexible                & 174                   & 276                       & 0\%                     & 1        & 17        & 270      & 4253   & 4631    & 8          & 5664      & 171          \\
                                                    \cmidrule(r){2-14}
    & \multirow{3}{*}{5}        & Fixed-Interval          & 191                   & 4                         & 0\%                     & 6        & 42        & 0        & 0      & 0       & 0          & 0         & 0            \\
                            &                           & Consistent              & 187                   & 624                       & 100\%                   & 1        & 18        & 618      & 3423   & 13815   & 93         & 11237     & 359          \\
                            &                           & Flexible                & 176                   & 655                       & 33\%                    & 1        & 25        & 649      & 6848   & 7336    & 17         & 12965     & 404          \\
                                                    \cmidrule(r){2-14}
    & \multirow{3}{*}{6}        & Fixed-Interval          & 191                   & 4                         & 0\%                     & 6        & 44        & 0        & 0      & 0       & 0          & 0         & 0            \\
                            &                           & Consistent              & 184                   & 844                       & 100\%                   & 1        & 16        & 835      & 2557   & 12358   & 149        & 10438     & 501          \\
                            &                           & Flexible                & 173                   & 365                       & 67\%                    & 1        & 13        & 356      & 1913   & 2366    & 6          & 4572      & 268          \\
                                               \cmidrule(r){2-14}
         & \multirow{3}{*}{7}        & Fixed-Interval          & 191                   & 4                         & 0\%                     & 6        & 44        & 0        & 0      & 0       & 0          & 0         & 0            \\
                            &                           & Consistent              & 197                   & 728                       & 100\%                   & 1        & 17        & 715      & 821    & 11124   & 132        & 3988      & 477          \\
                            &                           & Flexible                & 181                   & 1303                      & 87\%                    & 2        & 20        & 1290     & 3407   & 3164    & 12         & 8729      & 975          \\
                            \midrule
\multirow{12}{*}{7}         & \multirow{3}{*}{4}        & Fixed-Interval          & 245                   & 4                         & 0\%                     & 46       & 251       & 0        & 0      & 0       & 0          & 0         & 0            \\
                            &                           & Consistent              & 233                   & 704                       & 100\%                   & 1        & 24        & 698      & 8226   & 31131   & 177        & 14449     & 220          \\
                            &                           & Flexible                & 226                   & 2981                      & 57\%                    & 10       & 58        & 2975     & 42400  & 69089   & 176        & 44524     & 1077         \\
                                                    \cmidrule(r){2-14}
    & \multirow{3}{*}{5}        & Fixed-Interval          & 245                   & 4                         & 0\%                     & 46       & 261       & 0        & 0      & 0       & 0          & 0         & 0            \\
                            &                           & Consistent              & 240                   & 1142                      & 100\%                   & 1        & 36        & 1135     & 6108   & 34042   & 350        & 18357     & 399          \\
                            &                           & Flexible                & 236                   & 2360                      & 83\%                    & 4        & 46        & 2353     & 19170  & 29206   & 84         & 29387     & 830          \\
                                                     \cmidrule(r){2-14}
   & \multirow{3}{*}{6}        & Fixed-Interval          & 245                   & 4                         & 0\%                     & 46       & 266       & 0        & 0      & 0       & 0          & 0         & 0            \\
                            &                           & Consistent              & 239                   & 1957                      & 96\%                    & 4        & 30        & 1947     & 5834   & 51619   & 710        & 26389     & 726          \\
                            &                           & Flexible                & 228                   & 2967                      & 56\%                    & 11       & 57        & 2957     & 14185  & 23427   & 77         & 29849     & 1415         \\
                                                     \cmidrule(r){2-14}
   & \multirow{3}{*}{7}        & Fixed-Interval          & 245                   & 5                         & 0\%                     & 46       & 269       & 0        & 0      & 0       & 0          & 0         & 0            \\
                            &                           & Consistent              & 244                   & 1433                      & 100\%                   & 1        & 24        & 1420     & 1315   & 33011   & 591        & 7471      & 594          \\
                            &                           & Flexible                & 233                   & 1803                      & 72\%                    & 5        & 40        & 1790     & 1968   & 13799   & 21         & 8938      & 970          \\
                            \midrule
\multirow{12}{*}{10}        & \multirow{3}{*}{4}        & Fixed-Interval          & 335                   & 7                         & 0\%                     & 104      & 529       & 1        & 0      & 0       & 0          & 0         & 0            \\
                            &                           & Consistent              & 328                   & 3039                      & 87\%                    & 4        & 56        & 3033     & 11644  & 111926  & 1586       & 23571     & 254          \\
                            &                           & Flexible                & 304                   & 2641                      & 69\%                    & 6        & 98        & 2634     & 22348  & 74867   & 266        & 28921     & 532          \\
                                                     \cmidrule(r){2-14}
   & \multirow{3}{*}{5}        & Fixed-Interval          & 335                   & 7                         & 0\%                     & 112      & 550       & 1        & 0      & 0       & 0          & 0         & 0            \\
                            &                           & Consistent              & 343                   & 3436                      & 33\%                    & 1        & 45        & 3426     & 773    & 57004   & 3065       & 4672      & 115          \\
                            &                           & Flexible                & 323                   & 2981                      & 84\%                    & 4        & 68        & 2971     & 9380   & 58076   & 176        & 18275     & 593          \\
                                                      \cmidrule(r){2-14}
  & \multirow{3}{*}{6}        & Fixed-Interval          & 335                   & 7                         & 0\%                     & 111      & 543       & 1        & 0      & 0       & 0          & 0         & 0            \\
                            &                           & Consistent              & 369                   & 3615                      & 22\%                    & 2        & 28        & 3600     & 479    & 45668   & 3312       & 4552      & 139          \\
                            &                           & Flexible                & 311                   & 2720                      & 90\%                    & 5        & 76        & 2705     & 6419   & 57362   & 171        & 21649     & 858          \\
                                                      \cmidrule(r){2-14}
  & \multirow{3}{*}{7}        & Fixed-Interval          & 335                   & 7                         & 0\%                     & 111      & 546       & 1        & 0      & 0       & 0          & 0         & 0            \\
                            &                           & Consistent              & 386                   & 3620                      & 0\%                     & 1        & 19        & 3600     & 117    & 29912   & 3334       & 1785      & 183          \\
                            &                           & Flexible                & 324                   & 3623                      & 67\%                    & 3        & 67        & 3603     & 3364   & 70651   & 122        & 15213     & 1002         \\
                            \midrule
\multirow{12}{*}{12}        & \multirow{3}{*}{4}        & Fixed-Interval          & 366                   & 7                         & 0\%                     & 81       & 803       & 2        & 0      & 0       & 0          & 0         & 0            \\
                            &                           & Consistent              & 433                   & 3608                      & 0\%                     & 1        & 27        & 3600     & 168    & 17589   & 3546       & 1357      & 23           \\
                            &                           & Flexible                & 376                   & 3608                      & 58\%                    & 3        & 90        & 3600     & 16251  & 136003  & 643        & 32516     & 410          \\
                                                     \cmidrule(r){2-14}
   & \multirow{3}{*}{5}        & Fixed-Interval          & 366                   & 8                         & 0\%                     & 81       & 817       & 2        & 0      & 0       & 0          & 0         & 0            \\
                            &                           & Consistent              & 433                   & 3613                      & 0\%                     & 1        & 22        & 3602     & 138    & 17379   & 3505       & 1230      & 55           \\
                            &                           & Flexible                & 375                   & 3612                      & 47\%                    & 3        & 91        & 3600     & 5441   & 101144  & 436        & 14577     & 447          \\
                                                       \cmidrule(r){2-14}
 & \multirow{3}{*}{6}        & Fixed-Interval          & 366                   & 8                         & 0\%                     & 81       & 828       & 2        & 0      & 0       & 0          & 0         & 0            \\
                            &                           & Consistent              & 433                   & 3620                      & 0\%                     & 1        & 18        & 3603     & 111    & 18205   & 3455       & 1249      & 106          \\
                            &                           & Flexible                & 363                   & 3618                      & 73\%                    & 5        & 102       & 3600     & 3507   & 112384  & 491        & 22537     & 679          \\
                                                    \cmidrule(r){2-14}
    & \multirow{3}{*}{7}        & Fixed-Interval          & 366                   & 8                         & 0\%                     & 81       & 833       & 2        & 0      & 0       & 0          & 0         & 0            \\
                            &                           & Consistent              & 433                   & 3623                      & 0\%                     & 1        & 16        & 3600     & 97     & 20531   & 3372       & 1221      & 184          \\
                            &                           & Flexible                & 390                   & 3633                      & 39\%                    & 2        & 57        & 3609     & 549    & 99366   & 335        & 9558      & 653          \\
                            \midrule
\multirow{12}{*}{15}        & \multirow{3}{*}{4}        & Fixed-Interval          & 424                   & 10                        & 0\%                     & 91       & 917       & 4        & 0      & 0       & 0          & 0         & 0            \\
                            &                           & Consistent              & 499                   & 3608                      & 0\%                     & 1        & 22        & 3601     & 138    & 10608   & 3560       & 1011      & 23           \\
                            &                           & Flexible                & 444                   & 3609                      & 17\%                    & 1        & 71        & 3601     & 1560   & 117119  & 2898       & 9990      & 75           \\
                                                      \cmidrule(r){2-14}
  & \multirow{3}{*}{5}        & Fixed-Interval          & 424                   & 10                        & 0\%                     & 89       & 1012      & 4        & 0      & 0       & 0          & 0         & 0            \\
                            &                           & Consistent              & 499                   & 3611                      & 0\%                     & 1        & 18        & 3601     & 114    & 12876   & 3532       & 1099      & 47           \\
                            &                           & Flexible                & 435                   & 3612                      & 17\%                    & 1        & 63        & 3601     & 702    & 112342  & 2612       & 7495      & 107          \\
                                                       \cmidrule(r){2-14}
 & \multirow{3}{*}{6}        & Fixed-Interval          & 424                   & 11                        & 0\%                     & 83       & 999       & 4        & 0      & 0       & 0          & 0         & 0            \\
                            &                           & Consistent              & 499                   & 3617                      & 0\%                     & 1        & 15        & 3600     & 92     & 14252   & 3484       & 1084      & 92           \\
                            &                           & Flexible                & 471                   & 3617                      & 0\%                     & 1        & 55        & 3601     & 426    & 124977  & 2630       & 10796     & 205          \\
                            \cmidrule(r){2-14}
                            & \multirow{3}{*}{7}        & Fixed-Interval          & 424                   & 9                         & 0\%                     & 85       & 1002      & 3        & 0      & 0       & 0          & 0         & 0            \\
                            &                           & Consistent              & 499                   & 3625                      & 0\%                     & 1        & 13        & 3600     & 71     & 15582   & 3393       & 1003      & 177          \\
                            &                           & Flexible                & 459                   & 3627                      & 0\%                     & 1        & 40        & 3603     & 284    & 92007   & 2203       & 6682      & 378      \\
                            \bottomrule
\end{tabular}
\end{table}
\clearpage

\setcounter{equation}{0}
\renewcommand{\theequation}{F\arabic{equation}}
\setcounter{table}{0} 
\setcounter{figure}{0}
\renewcommand{\thetable}{F\arabic{table}}
\renewcommand{\thefigure}{F\arabic{figure}}
\setcounter{algorithm}{0}
\renewcommand{\thealgorithm}{F\arabic{algorithm}}
\section{Extension to Include Emergency Transportation Costs}\label{etc}
Let $e$ denote the unit emergency transportation cost. Then, the average worst-case expected emergency transportation cost per period is defined as
\begin{align}
C^{\rm etc}=\sup_{\mathbb{P}\in\mathcal{P}}\mathbb{E}_{\mathbb{P}}\biggl[\frac{1}{T}\sum_{t\in[T]}e\sum_{v\in [V]}\Bigl(Q-
\sum_{i\in[N]}\ \sum_{(t_1,t_2)\in \eta^+(\mathcal{G}^{\rm temporal},t_1)}
y_{it_1t_2v}\,\bigl(s_{it_2}-s_{it_1}+\!\!\sum_{t'\in [t_1,t_2{-}1]}\!\! \zeta_{t'}^i\bigr)^+\Bigr)\biggr].\nonumber
\end{align}

Under the assumption that overshooting does not occur, and given the independence of $\zeta_t^i$ in $\mathcal{P}$, the supremum over $\mathbb{P}\in\mathcal{P}$ and the expectation can be taken separately for each $\zeta_t^i$, yielding the equivalent expression:
\begin{align}
C^{\rm etc}=&\frac{1}{T}\sum_{t\in[T]}e\sum_{v\in [V]}\Bigl(Q-
\sum_{i\in[N]}\ \sum_{(t_1,t_2)\in \eta^+(\mathcal{G}^{\rm temporal},t_1)}
y_{it_1t_2v}\,\bigl(s_{it_2}-s_{it_1}+\!\!\sum_{t'\in [t_1,t_2{-}1]}\!\! \sup_{\mathbb{P}\in\mathcal{P}}\mathbb{E}_{\mathbb{P}}[\zeta_{t'}^i]\bigr)\Bigr),\nonumber\\
=&\frac{1}{T}\sum_{t\in[T]}e\sum_{v\in [V]}\Bigl(Q-
\sum_{i\in[N]}\ \sum_{(t_1,t_2)\in \eta^+(\mathcal{G}^{\rm temporal},t_1)}
y_{it_1t_2v}\,\bigl(s_{it_2}-s_{it_1}+\!\!\sum_{t'\in [t_1,t_2{-}1]}\mu_{t'}^i\bigr)\Bigr).\nonumber
\end{align}

With this modification, the distributionally robust CIRP (DR-CIRP) model, originally defined in Model~\eqref{dr-cirp}, can be reformulated as
\begin{align}
\min_{x,\ y,\ s}\left\{C^{\rm vehicle}+C^{\rm trans}+C^{\rm DR\textit{-}inv}+C^{\rm etc}\big|{\rm Constraints}\ \eqref{eq:partition}-\eqref{eq:domain3},\ \eqref{st:sandy}-\eqref{eq:chance}\right\}.\nonumber
\end{align}

Under this extension, the second-level RMP (2nd-RMP) introduced in Section~\ref{2-rmp} has the following objective:
\begin{align}
    \min\ &p+\frac{1}{T}\sum_{t\in[T]}\sum_{r\in R} c_{r}x_{r}^t+\sum_{i\in[N]}\sum_{l(i)\in L(i)}\bigl(\frac{1}{T}c_{l(i)}-\pi_i\bigr)\gamma_{l(i)} +eVQ,\nonumber
\end{align}
where the cost $c_{l(i)}$ incurred when retailer $i$ follows plan $l(i)$ now reflects not only the inventory cost but also a partial component of the corresponding emergency transportation cost.

Correspondingly, the second-level subproblem PP2 for retailer $i$ (2nd-PP2($i$)) introduced in Section \ref{2-pp2} is modified as follows:
\begin{align}
&\sum_{(t_1,t_2)\in A^{\rm temporal}}y_{it_1t_2}\bigl(\frac{1}{T}f^{inv}(i, t_1,t_2,s_{it_1})+\psi_{t_2}(s_{it_2}-s_{it_1})+\delta_{it_1}+\nonumber\\
&\psi_{t_2}\sum_{t\in[t_1,t_2-1]}U^i_{t}\bigr)+\theta_i-\pi_i-e\bigl(s_{it_2}-s_{it_1}+\sum_{t\in[t_1,t_2-1]}\mu^i_{t}\bigr), \nonumber\\
=&\sum_{(t_1,t_2)\in A^{\rm temporal}}y_{it_1t_2}\bigl(\frac{1}{T}f^{inv}(i, t_1,t_2,s_{it_1})+(\psi_{t_2}-e)(s_{it_2}-s_{it_1})+\delta_{it_1}+\nonumber\\
&\psi_{t_2}\sum_{t\in[t_1,t_2-1]}U^i_{t}\bigr)+\theta_i-\pi_i-e\sum_{t\in[t_1,t_2-1]}\mu^i_{t}.\nonumber
\end{align}

Importantly, for both the 2nd-RMP and 2nd-PP2($i$) modifications, the emergency transportation costs enter as additive constants relative to the original model. Therefore, these extensions do not alter the solution process and can be incorporated without affecting the overall computational procedure.
\clearpage

\setcounter{equation}{0}
\renewcommand{\theequation}{G\arabic{equation}}
\setcounter{table}{0} 
\setcounter{figure}{0}
\renewcommand{\thetable}{G\arabic{table}}
\renewcommand{\thefigure}{G\arabic{figure}}
\setcounter{algorithm}{0}
\renewcommand{\thealgorithm}{G\arabic{algorithm}}
\section{Others}\label{others}
\begin{figure*}[!htb]
    \centering
    \subfigure[Cluster 1 Period 1]{\includegraphics[width=0.24\linewidth]{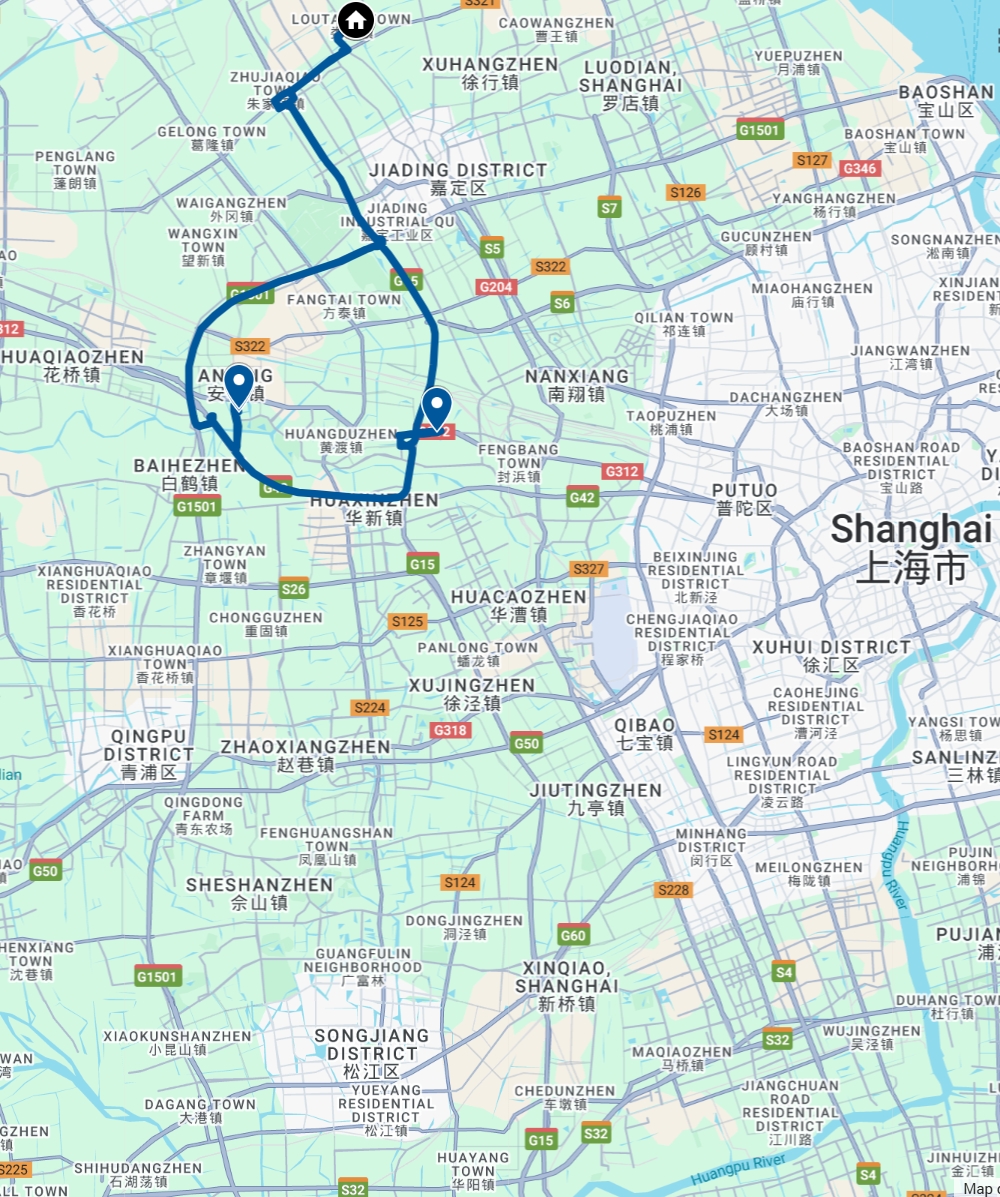}}
    \subfigure[Cluster 1 Period 2]{\includegraphics[width=0.24\linewidth]{fig-2-fle-1-1.png}}
    \subfigure[Cluster 1 Period 3]{\includegraphics[width=0.24\linewidth]{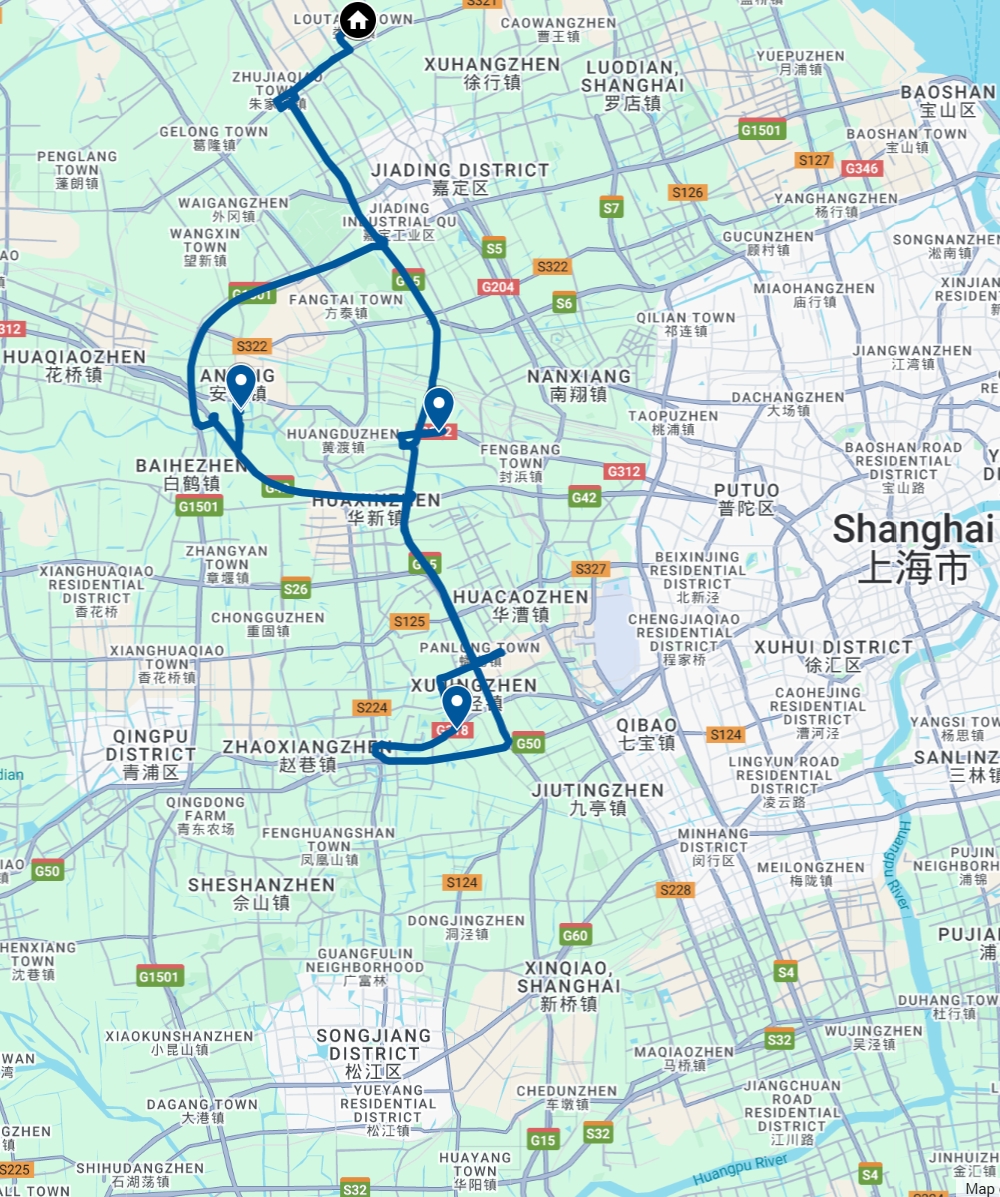}}\\
    \subfigure[Cluster 1 Period 4]{\includegraphics[width=0.24\linewidth]{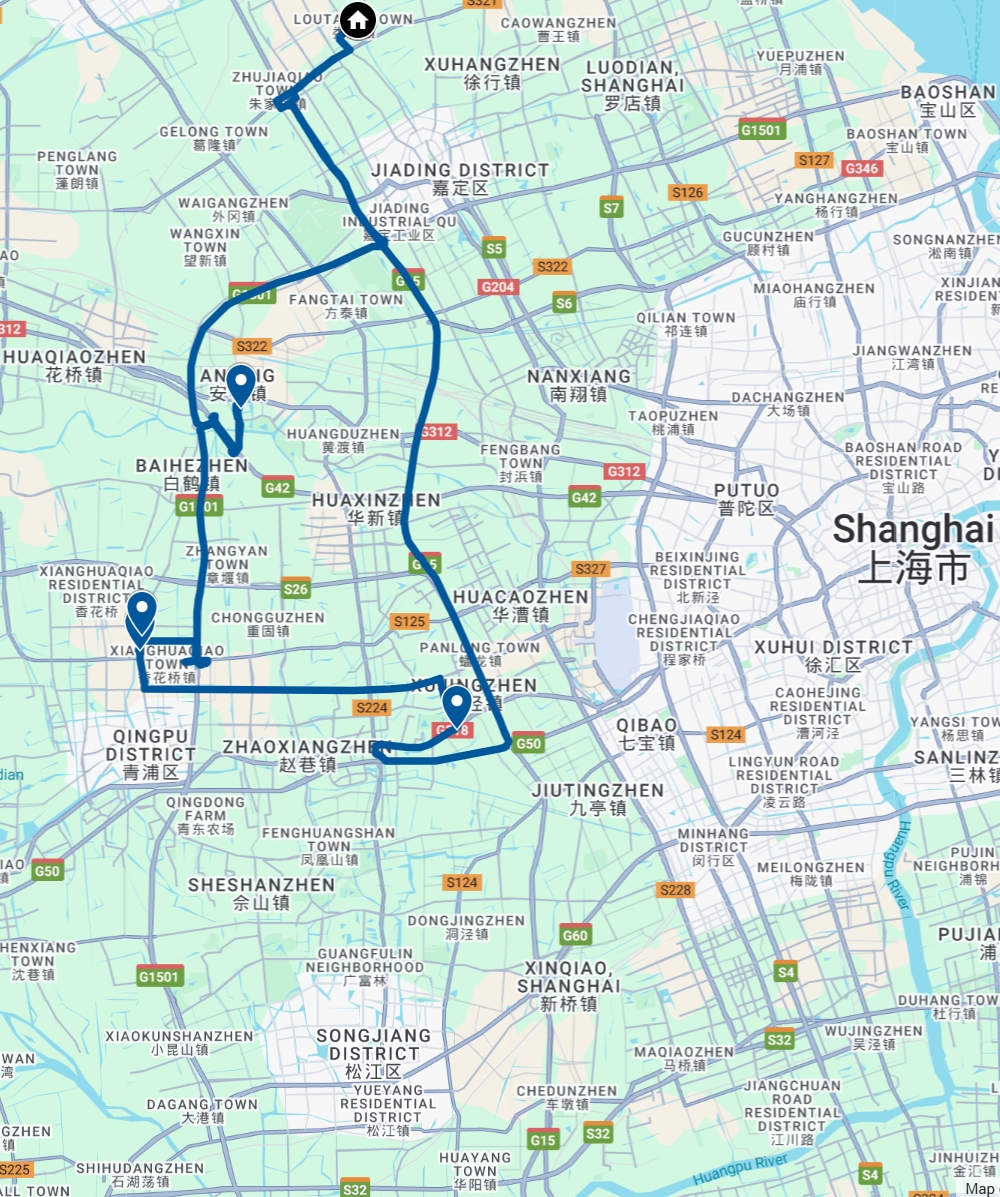}}
    \subfigure[Cluster 1 Period 5]{\includegraphics[width=0.24\linewidth]{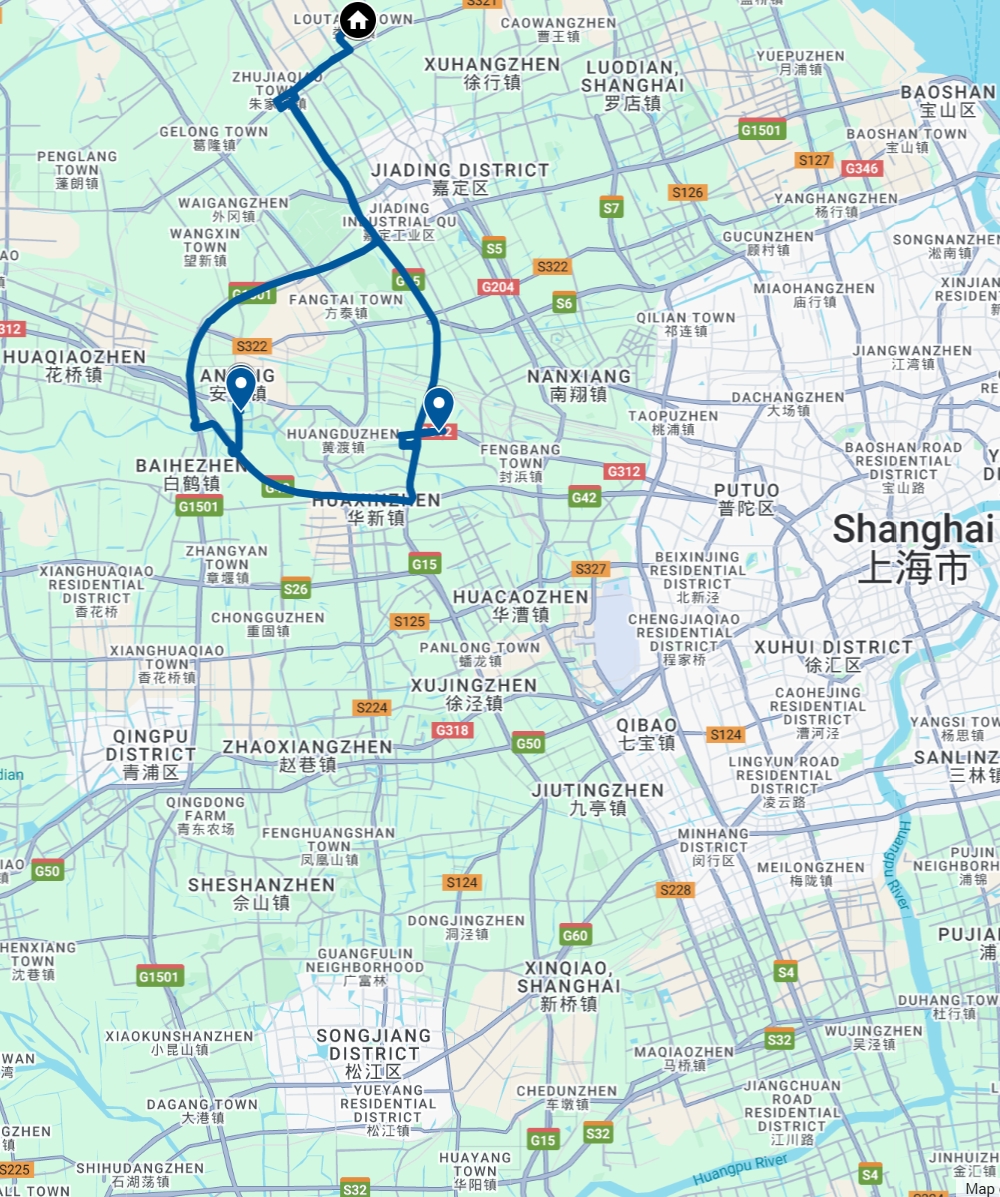}}
    \subfigure[Cluster 1 Period 6]{\includegraphics[width=0.24\linewidth]{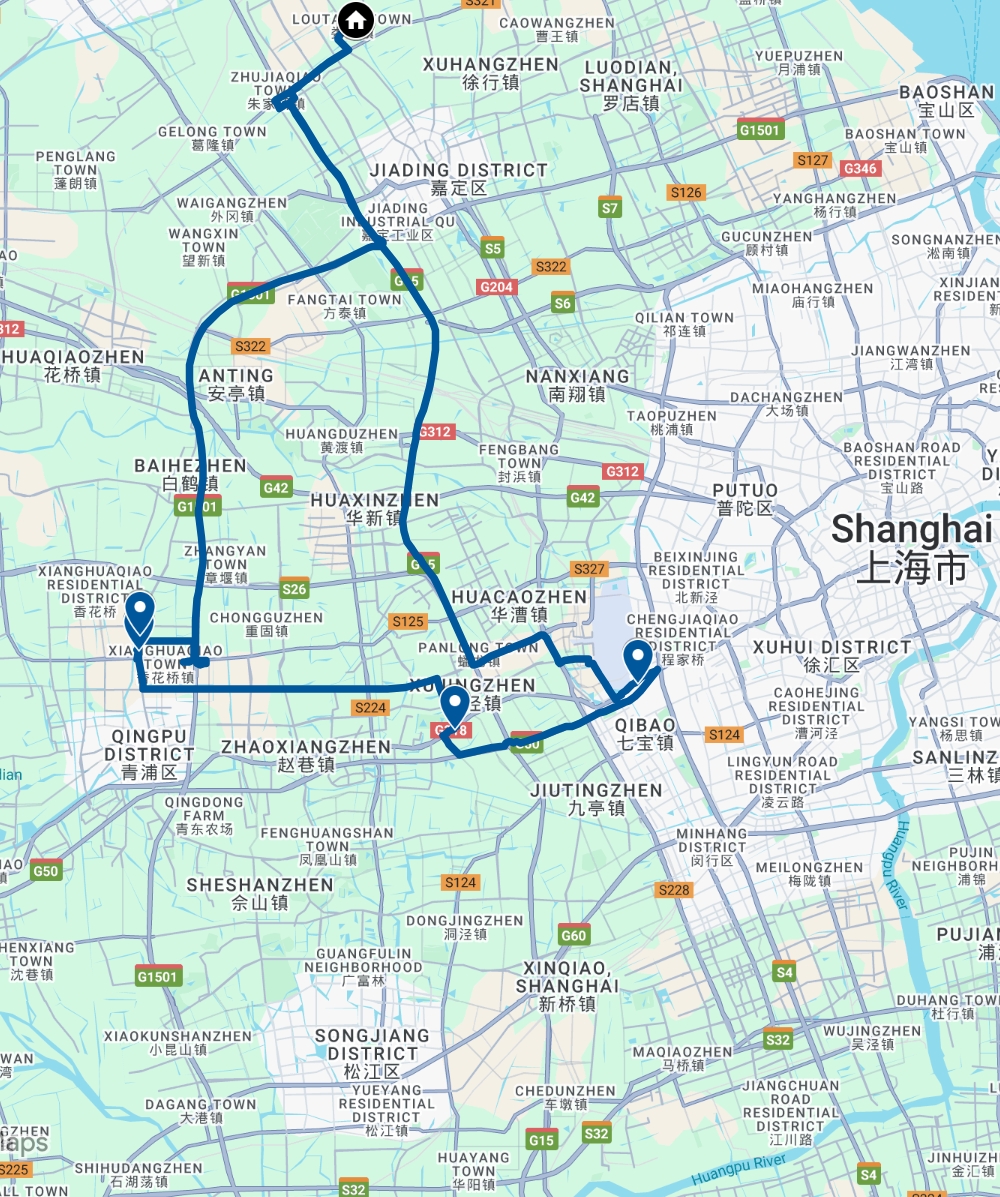}}\\
    \subfigure[Cluster 2 Period 2]{\includegraphics[width=0.24\linewidth]{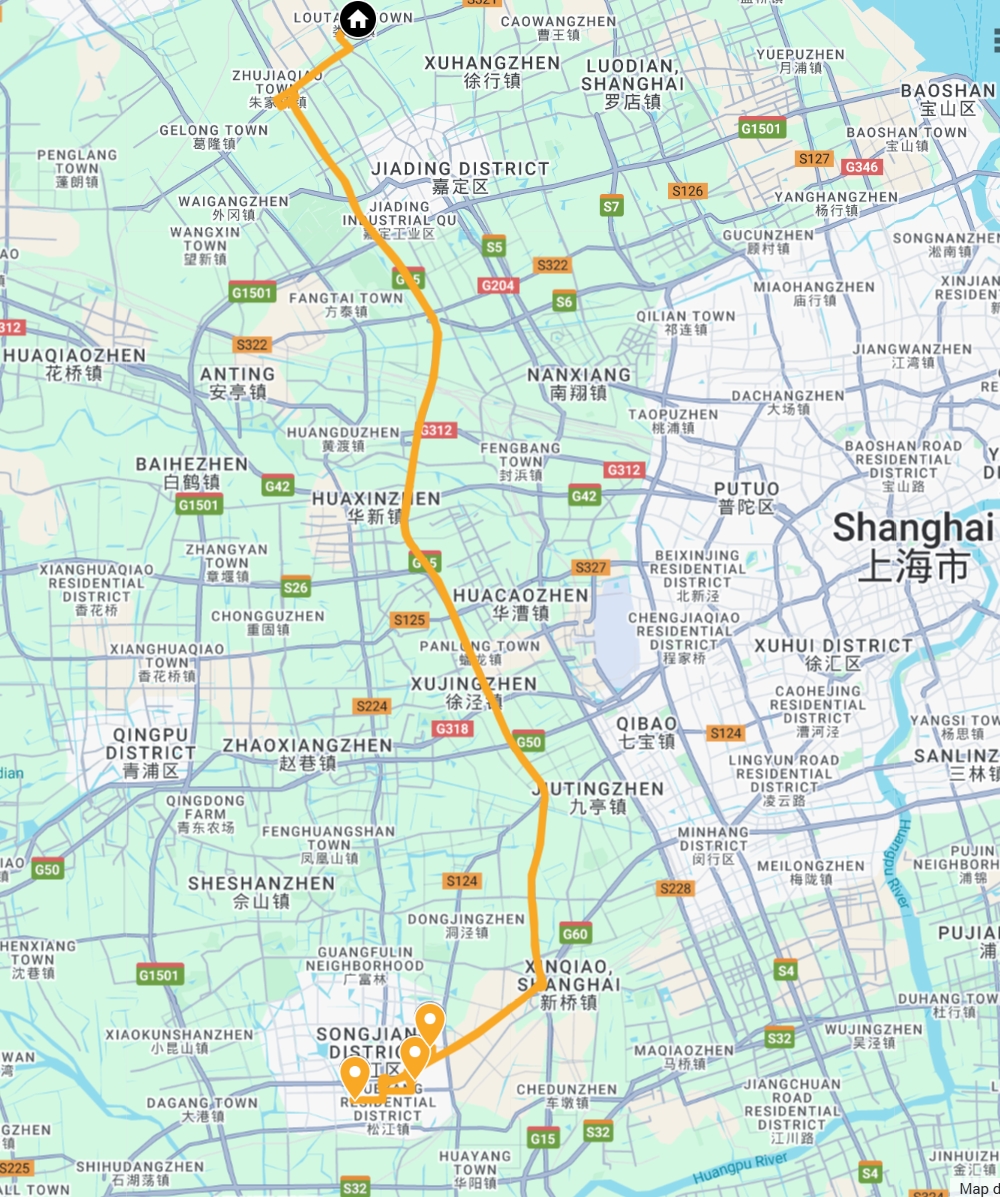}}
    \subfigure[Cluster 2 Period 3]{\includegraphics[width=0.24\linewidth]{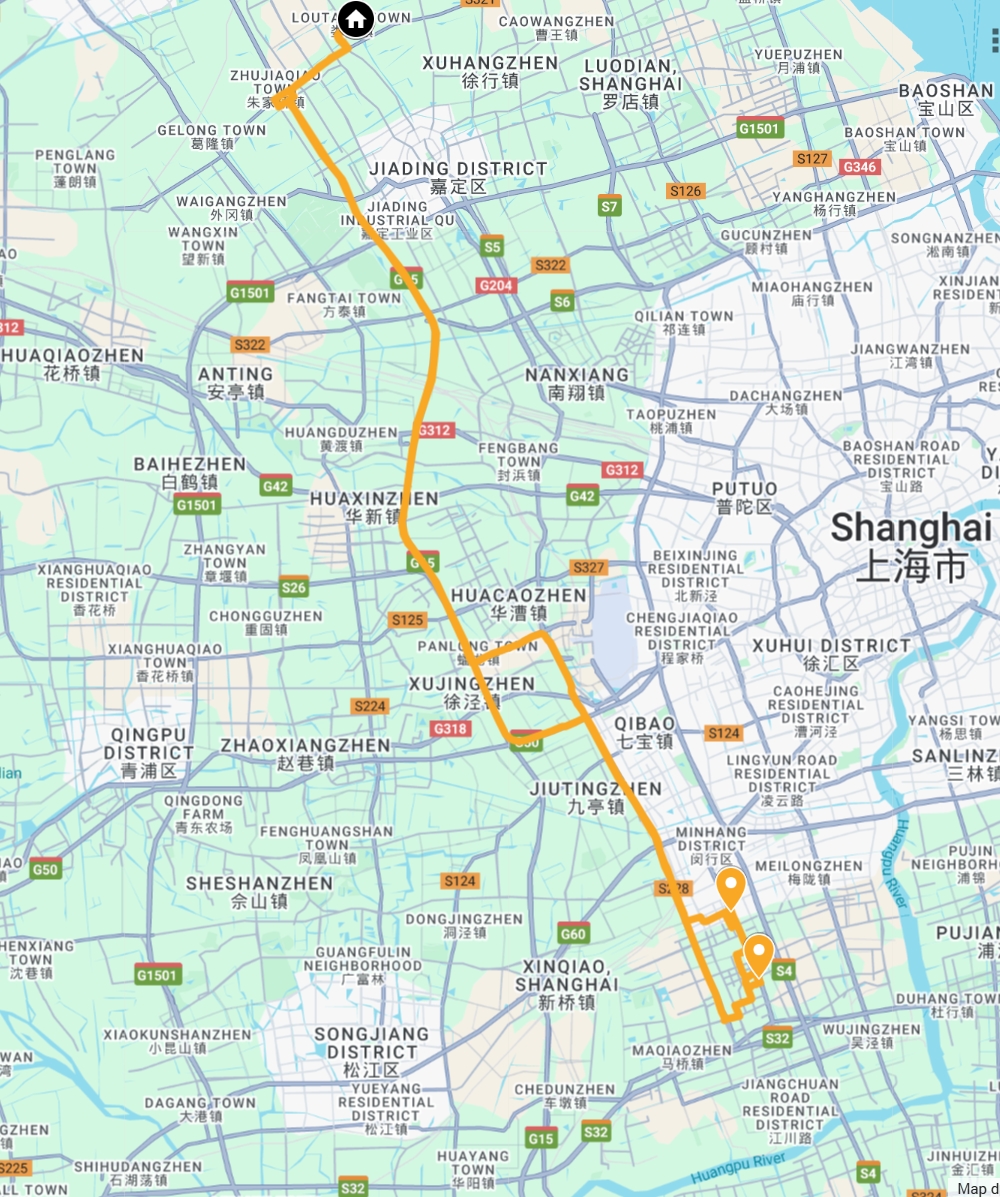}}
    \subfigure[Cluster 2 Period 4]{\includegraphics[width=0.24\linewidth]{fig-2-fle-2-13.png}}
    \subfigure[Cluster 2 Period 5]{\includegraphics[width=0.24\linewidth]{fig-2-fle-2-4.png}}
        \caption{Complete solution visualization of the flexible service policy of instance 2 under real-life data.}
    \label{fig:oa-real-policy}
\end{figure*}

\begin{algorithm}[!htb]
\DontPrintSemicolon
\textbf{Input}: dual variables, $\psi_{t_1},\psi_{t_2},\delta_{it_1}$; constants, $U_t^i$.\\
\textbf{Initialize}: $\varphi=\frac{1+\sqrt{5}}{2}$; $L=0$; $R=\sum_{t\in[T]}\bar{\zeta}_t^i$. \\
$s_1=R-\frac{R-L}{\varphi}, s_2=L+\frac{R-L}{\varphi}$\\
$c_1=costCal(\lfloor s_1\rfloor,\psi_{t_1},\psi_{t_2}); \ c_2=costCal(\lceil s_2\rceil,\psi_{t_1},\psi_{t_2})$\\
\While{$R-L>3$}{
\If{$c_1<c_2$}{
$R=\lceil s_2\rceil;\ s_2=s_1, s_1=R-\frac{R-L}{\varphi}$\\
$c_2=costCal(\lceil s_2\rceil,\psi_{t_1},\psi_{t_2}); \ c_1=costCal(\lfloor s_1\rfloor,\psi_{t_1},\psi_{t_2})$
}
\Else{
$L=\lfloor s_1\rfloor;\ s_1=s_2, s_2=L+\frac{R-L}{\varphi}$\\
$c_1=costCal(\lfloor s_1\rfloor,\psi_{t_1},\psi_{t_2});\ c_2=costCal(\lceil s_2\rceil,\psi_{t_1},\psi_{t_2})$
}
}
$c_L = costCal(L,\psi_{t_1},\psi_{t_2});\ c_{L+1} = costCal(L+1,\psi_{t_1},\psi_{t_2})$\\
$c_{R-1} = costCal(R-1,\psi_{t_1},\psi_{t_2});\ c_R = costCal(R,\psi_{t_1},\psi_{t_2})$\\
$c^* = \min\{c_L, c_{L+1}, c_{R-1}, c_R\}$\\
\Return{cost=$c^*+\delta_{it_1}+\psi_{t_2}\sum_{t\in [t_1,t_2-1]}U^i_t$; {\rm corresponding} $s^*\in\{L,L+1,R-1,R\}$.}\\
\SetKwFunction{FMain}{$costCal$}
\SetKwProg{Fn}{Function}{:}{}
\Fn{\FMain{$s, \psi_{t_1},\psi_{t_2}$}}{
$inv^c=\mathcal{LP}(s)$ // LP formulation in Theorem \ref{wd_theorem}\\
$cost = \frac{1}{T}inv^c+(\psi_{t_1}-\psi_{t_2})s$\\
\Return $cost$
}
\caption{Golden Section Search.}
\label{goldenSection}
\end{algorithm}

\begin{algorithm}[!htb]
\DontPrintSemicolon
\textbf{Input}: The optimal order-up-to level when dual variables being 0, $s$;\\
\hphantom{\textbf{Input}: }the dual value, $dual=\psi_{t_1}-\psi_{t_2}$; \\
\hphantom{\textbf{Input}: }the constant value, $constant=\delta_{it_1}+\psi_{t_2}\sum_{t\in [t_1,t_2-1]}U^i_t$;\\
\hphantom{\textbf{Input}: }dictionary of order-up-to levels and corresponding worst-case expected inventory costs, $dict$.\\
\textbf{Initialize}: $s^{*}=s$\\
\While{not terminate}{
\If{$dual>0$}{
$s^{\prime}=s^{*}-1$
}
\Else{
$s^{\prime}=s^{*}+1$
}
\If{$dict$ {\rm includes} $s^{\prime}$}{
$diff=dict(s^{\prime})-dict(s^{*})$\\
\If{$diff\geq |dual|$}{
$inv^c=dict(s^*)$\\
Terminate
}
\ElseIf{$diff \leq \varepsilon$ or $|dual| \leq 5*diff$}{
\If{$dual>0$}{
$s^{*} -= 1$
}
\Else{
$s^* += 1$
}
}
\Else{
$s^*, inv^c, dict$ = \textbf{goldenSection}($dual, dict$) // Algorithm \ref{goldenSection}\\
Terminate
}
}
\Else{
$inv^c=\mathcal{LP}(s^{\prime})$ // LP formulation in Theorem \ref{wd_theorem}\\
Include $(s^{\prime},inv^c)$ in $dict$
}
\If{$s^{*}=0$ or $s^{*}=$ {\rm maximum inventory capacity}}{
$inv^c=dict(s^*)$\\
Terminate
}
}
$cost=\frac{1}{T}inv^c+dual*s^*+constant$\\
\Return{$s^*, cost, dict$}
\caption{Cost Update.}
\label{updateCvw}
\end{algorithm}
\clearpage

\end{APPENDICES}

\end{document}